\documentclass[11pt,oneside, reqno]{amsart}
\usepackage[utf8]{inputenc}
\usepackage{amsmath}
\usepackage{amsthm,ifthen, amsfonts, amssymb,
srcltx, amsopn, color, enumerate}
 \usepackage{mathabx}
\usepackage{amssymb}
\usepackage{bbm}
\usepackage[mathscr]{euscript}
\usepackage{enumerate}
\usepackage{graphicx}
\usepackage{tikz}
\usepackage{tikz-cd}
\usepackage{color}
\usepackage{multicol}
\usepackage{mathrsfs}
\usepackage{hyperref}
\usepackage[normalem]{ulem}
\usepackage{float}
\usepackage{subcaption}
\usepackage{mathabx}
\usepackage{hieroglf}
\usepackage[T1]{fontenc}
\usepackage[cmtip,arrow]{xy}
\usepackage{pb-diagram, pb-xy}
\usepackage{overpic}
\usepackage{verbatim}
\renewcommand{\widehat}{\hat}

\newcommand{\showcommentsbox}{yes}

\newsavebox{\commentbox}
{\ifthenelse{\equal{\showcommentsbox}{yes}}%
{\footnotemark
        \begin{lrbox}{\commentbox}
        \begin{minipage}[t]{1.25in}\raggedright\sffamily\tiny
        \footnotemark[\arabic{footnote}]}
{\begin{lrbox}{\commentbox}}}%
{\ifthenelse{\equal{\showcommentsbox}{yes}}%
{\end{minipage}\end{lrbox}\marginpar{\usebox{\commentbox}}}
{\end{lrbox}}}

\definecolor{Green}{RGB}{30, 150, 30}

\usetikzlibrary{arrows}
\usepackage{tikz-cd} 

\newtheorem{thm}{Theorem}[section]
\newtheorem{prop}[thm]{Proposition}
\newtheorem{claim}[thm]{Claim}
\newtheorem{lem}[thm]{Lemma}

\newtheorem{cor}[thm]{Corollary}

\newtheorem{thmi}{Theorem}
\newtheorem{cori}[thmi]{Corollary}

\theoremstyle{definition}
\newtheorem{defn}[thm]{Definition}
\theoremstyle{definition}

\theoremstyle{definition}

\theoremstyle{definition}
\newtheorem{remark}[thm]{Remark}
\theoremstyle{definition}

\theoremstyle{definition}

\newcommand*\Z{\mathbb{Z}}

\newcommand*\R{\mathbb{R}}

\newcommand*\Stab{\operatorname{Stab}}

\newcommand*\diam{\operatorname{diam}}

\newcommand*\nest{\sqsubseteq}
\newcommand*\propnest{\sqsubsetneq}

\newcommand*\mf[1]{\mathfrak{#1}}
\newcommand*\mc[1]{\mathcal{#1}}
\newcommand*\trans{\pitchfork}

\newcommand{\tsh}[1]{\left\{\kern-.7ex\left\{#1\right\}\kern-.7ex\right\}}

\newcommand{\fontact}{{\mathcal C}}
\newcommand{\orth}{\bot}

\newcommand{\stab}{\operatorname{Stab}}
\newcommand*\lk{\operatorname{lk}}
\newcommand*\st{\operatorname{st}}

\newcommand{\squidable}{admissible\ }
\newcommand{\Sat}{\operatorname{Sat}}
\newcommand{\cay}{\operatorname{Cay}}

\newcommand*\link{\operatorname{lk}}

\newcommand{\BS}{X}
\newcommand{\ST}{\mc{S}(T)}

\title[Equivariant HHS via quasimorphisms]{Equivariant hierarchically hyperbolic structures for 3--manifold groups via quasimorphisms}

\author{Mark Hagen}
\address{School of Mathematics, University of Bristol, Bristol, UK}
\email{markfhagen@posteo.net}
\author{Jacob Russell}
\address{Department of Mathematics, Rice University, Houston, TX, USA}
\email{jacob.russell@rice.edu}
\author{Alessandro Sisto}
	\address{Maxwell Institute and Department of Mathematics, Heriot-Watt University, Edinburgh, UK}
	\email{a.sisto@hw.ac.uk}
\author{Davide Spriano}
\address{Mathematical Institute, University of Oxford, Oxford, UK}
\email{spriano@maths.ox.ac.uk}

\begin{document}
\begin{abstract}
	Behrstock, Hagen, and Sisto classified 3--manifold groups admitting a hierarchically hyperbolic space structure. However, these structures were not always equivariant with respect to the group. In this paper, we classify 3--manifold groups admitting equivariant hierarchically hyperbolic structures. The key component of our proof is that the admissible groups introduced by Croke and Kleiner always admit equivariant hierarchically hyperbolic structures. For non-geometric graph manifolds, this is contrary to a conjecture of Behrstock, Hagen, and Sisto and also contrasts with results about CAT(0) cubical structures on these groups. Perhaps surprisingly, our arguments involve the construction of suitable quasimorphisms on the Seifert pieces, in order to construct actions on quasi-lines.
\end{abstract}

\maketitle

\tableofcontents

\section{Introduction}
Fundamental groups of $3$--manifolds are a major source of inspiration in geometric group theory,  providing a great part of the motivation for the notion of Gromov-hyperbolicity and all its generalisations, the study of actions on nonpositively-curved spaces, and the increasingly important role of special cube complexes.

One notion of ``coarse nonpositive curvature'', inspired partly by special cube complexes, is \emph{hierarchical hyperbolicity}. Hierarchically hyperbolic spaces and groups were introduced in~\cite{BHSI} as a means of isolating geometric features common to mapping class groups and certain CAT(0) cubical groups. After the definition took an easier-to-verify form in~\cite{BHSII},  a budding study of hierarchical hyperbolicity has emerged. This has included
\begin{itemize}
	\item finding new examples of hierarchically hyperbolic spaces and groups~\cite{BHSII,BerlaiRobbio,Miller,CHHS,BHS3,Berlyne, BerlyneRussell,vokes,HMS:artin,Veech,HS,RobbioSpriano,Russell:multicurve,Hughes,NguyenQing};
	\item development of new tools~\cite{DHS_boundary,DHS_corr,RST_convexity,CHHS,spriano_hyperbolic_1,Russell_Relative_Hyperbolicity};
	\item establishment of geometric and algebraic consequence of hierarchical hyperbolicity \cite{BHS3,BHS4,ANS:exponential,HagenPetyt,Petyt:quasicubical,HHP,DMS:cube}.
\end{itemize}

Very roughly, a \emph{hierarchically hyperbolic space structure} on a space $W$ consists of a set $\mf{S}$ indexing a collection of $\delta$--hyperbolic space $\{\mc{C}(U)\}_{U \in \mf{S}}$ and a collection of projection maps $\{\pi_U \colon W \to \mc{C}(U)\}_{U\in\mf{S}}$ satisfying a collection of axioms that allow for the coarse geometry of $W$ to be recovered from these projections; see~\cite[Definition 1.1]{BHSII} for the precise definition. Often, $W$ is a finitely-generated group $G$ equipped with a word metric. In this case, stronger results can be achieved when $G$ is not only a hierarchically hyperbolic space (HHS), but has a  structure that is compatible with the group action. 
These hierarchically hyperbolic \emph{groups} (HHG) are defined precisely in Definition~\ref{defn:HHG}, but essentially this means that $G$ acts cofinitely on $\mathfrak S$, with elements $g\in G$ inducing isometries $\fontact (U)\to \fontact (gU)$ so that all of the expected diagrams involving these isometries and the projections from the definition of an HHS commute.   

The difference between HHSs and HHGs is illustrated by the fact that being an HHS is a quasi-isometry invariant property, but being an HHG is not~\cite{PetytSpriano}.  While considerable geometric information can be gleaned from merely knowing that $G$ is an HHS (e.g. finiteness of the asymptotic dimension~\cite{BHS3} or control of quasiflats~\cite{BHS4}), one gets much more from the HHG property (e.g. semihyperbolicity~\cite{HHP,DMS:cube} and the Tits alternative~\cite{DHS_boundary,DHS_corr}, or the consequences listed in Corollary~\ref{cori:consequences}).

The first examples of hierarchically hyperbolic spaces beyond mapping class groups and some cube complexes were the fundamental groups of closed orientable $3$--manifolds whose prime decompositions excludes Nil and Sol pieces~\cite{BHSII}.  However, the hierarchically hyperbolic structures constructed for such groups in~\cite{BHSII} are in general non-equivariant.  In the present paper, we use new combinatorial techniques to produce \emph{equivariant} hierarchically hyperbolic structures for $3$--manifold groups. 
While many of the consequences of hierarchical hyperbolicity were known previously for $3$--manifold groups, we find this  satisfying as a  complete answer to the question of hierarchical hyperbolicity for $3$--manifold groups:

\begin{thmi}[Theorem \ref{thm:HHGs_and_3_manifolds}]\label{thmi:main}
	Let $M$ be a closed oriented $3$--manifold.  Then $\pi_1M$ is a hierarchically hyperbolic group if and only if $M$ has no  Nil, Sol, or non-octahedral flat manifolds in its prime decomposition. 
\end{thmi}

In light of the previous characterisation of which 3-manifold groups are HHSs,  Theorem \ref{thmi:main} says the only additional obstruction to being HHG are non-octahedral flat manifolds in the prime decomposition.

Theorem \ref{thmi:main} disproves a conjecture of  Behrstock--Hagen--Sisto that there were examples of non-geometric graph manifold groups  that were hierarchically hyperbolic \emph{spaces}, but not hierarchically hyperbolic \emph{groups}; see~\cite[Remark 10.2]{BHSII}. This is a surprising result as this conjecture had a compelling heuristic justification. We explain this heuristic justification and how we circumvent it, then discuss the outline of our proof of Theorem \ref{thmi:main}.

\subsection{Comparison with cubulations: lines vs quasi-lines}
To explain the justification for the original belief that some graph manifold groups were not HHGs, we start with the \emph{octahedral} hypothesis in Theorem \ref{thmi:main}.  This says that the flat pieces are quotients of $\mathbb E^3$ by crystallographic groups with point group conjugate into $O_3(\mathbb Z)$ (see~\cite[Definition 2.2]{Hagen:crystallographic} or~\cite[Theorem 7.1]{Hoda:crystallographic}).  For crystallographic groups in any dimension, being octahedral is equivalent to cocompact cubulation~\cite{Hagen:crystallographic}. Petyt--Spriano showed that this is in turn equivalent to being an HHG~\cite{PetytSpriano}.  So, while every crystallographic group is an HHS via a quasi-isometry to $\mathbb Z^n$, many crystallographic groups, such as the $(3,3,3)$--triangle group, are not HHGs.

There is a similar obstruction to cocompactly cubulating $\pi_1M$ when $M$ is a non-geometric graph manifold~\cite{HP_charge}.  Specifically, $\pi_1M$ can be cocompactly cubulated if it is \emph{flip} in the sense of~\cite{KL_flip}, that is, in every Seifert piece there is a ``horizontal'' surface whose boundary circles are fibres in the adjacent Seifert pieces.  The idea  behind the obstruction to cubulation is then: if $\widetilde T\subset\widetilde M$ is an elevation of a JSJ torus to the universal cover, and $\pi_1M$ is cubulated, then the walls in $\widetilde M$ cut through $\widetilde T$ in at least two intersecting families of parallel lines.  If the ``flip'' condition fails, then in some $\widetilde T$, there will be at least three such families, and the dual cube complex will contain $\widetilde T\times \mathbb R\cong\mathbb E^3$, preventing cocompactness.  So, the obstruction to cocompact cubulation arises from specific $\mathbb Z^2$ subgroups getting ``over-cubulated'', as is the case with crystallographic groups.

The suspicion (confirmed in~\cite{PetytSpriano}) that cocompact cubulation is equivalent to the existence of an HHG structure for virtually abelian groups, together with the restrictions on cubulating graph manifolds, motivated the now disproven belief that non-flip graph manifold groups could fail to be HHGs. 

The proof of Theorem~\ref{thmi:main} shows that constructing an HHG structure needs less than is needed to cocompactly cubulate.   Roughly, in a cocompact cubulation of $\pi_1M$, the immersed walls in $M$ cut through each Seifert piece in a collection of surfaces whose boundary circles map to fibers in adjacent blocks; for each Seifert piece $B$ we thus need a $\pi_1B$--action on a line where certain elements act loxodromically and specific others fix points.  For an HHG structure, we only need an action of $\pi_1B$ on a \emph{quasi-line} such that the central $\mathbb Z$ acts loxodromically, but the $\mathbb Z$ subgroups corresponding to the fibers of the adjacent Seifert pieces act with bounded orbits.  The latter constraint is satisfiable even if $M$ is not flip.  This explains the involvement of quasimorphisms in our proof. The idea of using quasimorphisms in building HHG structures originated in this project, but has already found additional applications to Artin groups \cite{HMS:artin} and extensions of subgroups of mapping class groups \cite{Veech}.

Another simple application of these actions on quasi-lines is that central extensions of hyperbolic groups by $\mathbb{Z}$  are HHGs.
\begin{cori}[Corollary \ref{cor:central_extensions_by_Z}]
	If a group $G$ is a central extension 
	$\mathbb Z \hookrightarrow G \stackrel{\pi}{\twoheadrightarrow} F $
	where $F$ is a non-elementary hyperbolic group, then $G$ is a hierarchically hyperbolic group.
\end{cori}

\noindent While these central extensions were known to be  hierarchically hyperbolic spaces by virtue of being quasi-isometric to $\mathbb Z \times F$, it did not appear to be known that they are in fact HHGs.

We now discuss the proof of Theorem \ref{thmi:main} in more detail.

\subsection{Reduction to graph manifolds and admissible groups}\label{subsec:graph_man_intro}
Let $M$ be a closed oriented $3$--manifold.  The proof of the forward direction of Theorem \ref{thmi:main}, that the existence of an HHG structure for $\pi_1M$ implies that $M$ has no Nil, Sol, or non-hyperoctahedral pieces in it prime decomposition, is a consequence of results in~\cite{PetytSpriano,BHSII,RST_convexity}. The idea is that we can push the HHG structure of $\pi_1 M$ to the fundamental groups of each of $M$'s prime pieces, implying they cannot be Nil, Sol, or non-octahedal flat.

The main part of the paper is  therefore devoted to other direction of Theorem \ref{thmi:main}, namely that if the prime decomposition of $M$ excludes Nil, Sol, and non-octahedral flat manifolds, then $\pi_1M$ is an HHG. As the geometric cases can largely be handled by appealing to results in the literature, the main new ingredient we need is that non-geometric graph manifold groups are HHGs.

\begin{cori}[Corollary~\ref{cor:graph_man}]\label{cori:graph_manifolds}
	If $M$ is a 3-dimensional non-geometric graph manifold, then $\pi_1 M$ is a hierarchically hyperbolic group.
\end{cori}

With Corollary \ref{cori:graph_manifolds} in hand, we can deduce the general case of Theorem \ref{thmi:main} using the fact that a group that is hyperbolic relative to HHGs is itself an HHG; see~\cite{BHSII}.

Our proof of Corollary~\ref{cori:graph_manifolds} only relies on the specific way a graph manifold group decomposes into a graph of groups. Hence, instead of working in the specific case of graph manifolds, we work in the setting of \emph{\squidable graphs of groups}. This is a class  of groups introduced by Croke and Kleiner to abstract the structure of $\pi_1 M$, when $M$ is a non-geometric graph manifold \cite{CrokeKleiner}.  Roughly, an \squidable graph of groups is a nontrivial finite graph of groups $\mathcal G$    where each edge group is $\mathbb Z^2$ and each vertex group $G_\mu$ has infinite cyclic center $Z_\mu$ with quotient $F_\mu = G_\mu / Z_\mu$ a non-elementary hyperbolic group. Additionally, the various edge groups need to be pairwise non-commensurable inside each vertex group.   The exact definition is Definition~\ref{defn:admissible}. Hence, hierarchical hyperbolicity of $\pi_1M$ is a special case of:

\begin{thmi}[Theorem~\ref{thm:squidable_HHS}, Proposition~\ref{prop:Hyperbolicity_of_links}]\label{thmi:main_technical}
	Let $\mathcal G$ be an \squidable graph of groups.  Then $\pi_1\mathcal G$ is a hierarchically hyperbolic group.  Moreover, if each quotient $F_\mu$ is a free group, then the associated hyperbolic spaces are quasi-isometric to trees.
\end{thmi}

Recently,  Nguyen and Qing showed that every admissible group that acts geometrically on a CAT(0) space is a hierarchically hyperbolic space~\cite[Theorem A]{NguyenQing}. Their result focuses on CAT(0) geometry and does not in general produce equivariant structures. Our proof of Theorem \ref{thmi:main_technical} will employ a much more combinatorial framework that will ensure equivariance and avoid the need for the action on a CAT(0) space.

\subsection{Consequences}\label{subsec:consequences}
Equivariant hierarchical hyperbolicity for fundamental groups of \squidable graphs of groups has several immediate consequences for these groups.
\begin{cori}{cori:consequences}\label{cori:consequences}
	Let $\mathcal G$ be an \squidable graph of groups, and let $G= \pi_1 \mc{G}$.  Then:
	\begin{enumerate}
		\item $G$ acts properly and coboundedly on an injective metric space, and is hence semihyperbolic;
		
		\item if $G$ is virtually torsion-free, then $G$ has uniform exponential growth;

		\item the action of  $G$ on the Bass--Serre tree is the largest (hence universal) acylindrical action of $G$ on a hyperbolic space;
		
		\item the Morse boundary of $G$ is an $\omega$--cantor set. In particular, it is totally disconnected.
	\end{enumerate}
\end{cori}

\begin{proof} 
	The first assertion follows from the fact that $G$ is an HHG (Theorem~\ref{thmi:main_technical}) by ~\cite[Corollary 3.8, Lemma 3.10]{HHP}.
	
	For the other assertions, we will need that  the $\nest$--maximal domain in the HHG structure is $G$--equivariantly quasi-isometric to the Bass--Serre tree $T$ for $\mc{G}$.  We prove this in Proposition~\ref{prop:Hyperbolicity_of_links}. Because the definition of an \squidable graph of groups ensures that $T$ has infinitely many ends, \cite[Corollary 4.8]{ANS:exponential} implies that $G$ has uniform exponential growth. It follows from~\cite[Theorem A]{ABD}  that the action of $G$ on $T$ is the largest acylindrical action of $G$ on a hyperbolic space\footnote{Theorem A of~\cite{ABD}, as written, can be read as suggesting that $3$--manifold groups are hierarchically hyperbolic \emph{groups}, although at the time they were only known to be hierarchically hyperbolic \emph{spaces} in the stated generality.  But, as noted in~\cite[Remark 5.3]{ABD}, Theorem A holds for $3$--manifold groups without needing an HHG structure.  Alternatively, by Theorem~\ref{thmi:main}, the statement in~\cite{ABD} holds once one excludes non-octahedral flat pieces from the prime decomposition.}.  The last item on the Morse boundary follows from \cite[Corollary A.8]{Russell:multicurve} (using HHGs) or \cite[Theorem 1.2]{CCS_omega_Morse_boundary} (using graphs of groups).
\end{proof}

\subsubsection*{HHGs on quasi-trees vs cubical groups}
We also note the following consequence for the question of when hierarchically hyperbolic structures are forced to arise from cubulation.  Corollary~\ref{cori:graph_manifolds} provides a hierarchically hyperbolic structure in which the constituent hyperbolic spaces are all quasi-isometric to trees.  Such hierarchically hyperbolic structures also arise on fundamental groups of compact special cube complexes~\cite{BHSI} and more generally, groups acting geometrically on cube complexes admitting \emph{factor systems} \cite{BHSI,HS}.  However, there are many examples of graph manifolds whose fundamental groups are virtually special but not virtually \emph{compact} special, and indeed not even virtually cocompactly cubulated~\cite{HP_charge}.  Hence Corollary~\ref{cor:graph_man} provides examples of groups that are not cocompactly cubulated, but do admit HHG structures in which the hyperbolic spaces are all quasi-trees. 

\subsection{Proof ingredients: combinatorial HHS and quasi-morphisms}



To prove admissible groups are HHGs, we employ the recent \emph{combinatorial HHS} machinery from~\cite{CHHS}. For a group $G$, this requires constructing a simplicial complex $Y$ and a graph $W$, which are then combined in a graph $Y^{+W}$. Intuitively, the role of those spaces is as follows: the complex $Y$ encodes the index set of a hierarchically hyperbolic structure, the complex $Y$ and the graph $Y^{+W}$ together encode the associated hyperbolic spaces, and $W$ is the (equivariant) quasi-isometry model of $G$.

In our case, the space $Y$ is an augmented version of the Bass-Serre tree, where each vertex is ``blown up'' to contain a copy of the coset it represents. 
This technique of building combinatorial HHSs by ``blowing up the vertex groups''  in some naturally-occurring hyperbolic graph is quite flexible, and  has analogues in a number of other contexts.  For example, it is applied in the context of certain Artin groups in~\cite{HMS:artin}, extensions of lattice Veech groups in~\cite{Veech}, and extensions of multicurve stabilisers in~\cite{Russell:multicurve}. In~\cite{CHHS}, it is  explained how to build combinatorial HHSs for  right-angled Artin groups and mapping class groups by respectively blowing up the Kim--Koberda \emph{extension graph} \cite{KimKoberda} and the curve graph.

 In a general combinatorial HHS, $Y$ is a simplicial complex with a $G$--action  that has finitely many orbits of links of simplices, and  $W$ is a graph whose vertices are maximal simplices of $Y$,  where the action of $G$ on $Y$ induces an isometric action of $G$ of $W$. Given $Y$ and $W$, the graph $Y^{+W}$ is constructed from $Y^{(1)}$ by joining every vertex of the maximal simplex $\Sigma$ to every vertex of the maximal simplex $\Delta$ by an edge whenever $\Sigma$ and $\Delta$ represent adjacent vertices of $W$.  The group $G$ acts naturally on the resulting graph $Y^{+W}$.

The spaces $Y$, $W$, and $Y^{+W}$ encode the HHS structure as follows. The elements of the index set correspond to the links $\link(\Delta)$ of non-maximal simplices $\Delta$ of $Y$ (including the empty simplex, whose link is $Y$).  The hyperbolic space associated to $\link(\Delta)$ is the subgraph $\link(\Delta)^{+W}$ of $Y^{+W}$ spanned by the vertices in $\link(\Delta)\subset Y$.  Accordingly, we have to choose the edges of $W$ in a way that ensures that all of these spaces (including $Y^{+W}$ itself) are hyperbolic, while also ensuring that the action of $G$ on $W$ is proper and cobounded. 

Hierarchical hyperbolicity demands not only the construction of a collection of hyperbolic spaces, but also a coarse projection from $W$ to each $\link(\Delta)^{+W}$ (satisfying a list of properties~\cite[Definition 1.1]{BHSII}).  To arrange this, the definition of a combinatorial HHS requires the following: consider all of the simplices $\Delta'\subset Y$ with the same link as $\Delta$, and remove their vertex sets (and incident edges) from $Y^{+W}$ to obtain a graph $Y_\Delta$, which contains $\link(\Delta)^{+W}$.  We ask that the inclusion $\link(\Delta)^{+W}\hookrightarrow Y_\Delta$ is a quasi-isometric embedding, for each non-maximal simplex $\Delta$.  The exact definition of a combinatorial HHS is Definition~\ref{defn:CHHS}, which involves some additional (combinatorial) conditions.

\subsubsection*{Our combinatorial HHS and the role of quasimorphisms}
Given an \squidable graph $\mathcal G$ of groups,  let $T$ be the Bass--Serre tree. The idea for constructing the simplicial complex $Y$ for $ \pi_1 \mc{G}$ is as follows: ``blow up'' each vertex $v$ of $T$ to become the cone on a discrete set whose elements correspond to  the associated coset of the vertex group. Two such cones are then graph-theoretically joined according to the edges of $T$, resulting in a $3$--dimensional simplicial complex.  The action of $\pi_1\mathcal G$ on $T$ induces an action on $Y$, by construction.


Having constructed the simplicial complex $Y$, we now need to construct the graph $W$ whose vertices are maximal simplices of $Y$  and will serve as the geometric model for $ \pi_1 \mc{G}$. This is where quasimorphisms come in.

Specifically, for each vertex group $G_\mu$, we construct an action of $G_\mu$ of a quasi-line $L_\mu$ so that the center $Z_\mu$ of $G_\mu$ acts loxodromically, and each cyclic subgroup conjugate to the images of  the center of adjacent vertex groups acts elliptically.
This is achieved by first choosing an appropriate quasimorphism and then using a result of Abbott--Balasubramanya--Osin~\cite{ABO} to promote it to an action on a quasi-line; see Lemma \ref{lem:Generating_sets_that_makes_lines}.

Using the action on this quasi-line, each vertex groups in our \squidable graph of group is \emph{equivariantly} quasi-isometric to the product $L_\mu \times F_\mu$, where $F_\mu$ is the hyperbolic quotient $G_\mu / Z_\mu$. Balls in the $L_\mu$ factor therefore give us coarse ``level surfaces'' in this product. Moreover, if $\omega$ is adjacent to $\mu$, the fact that the center $Z_\omega$ acts elliptically on $L_\mu$ means that $Z_\omega$ is sent into  one  of these coarse level surfaces by the  edge maps in $\mc{G}$. 

Now, maximal simplices of $Y$ consist of an edge $\{u,v\}$ of $T$ and a pair of elements $s,t$ in the corresponding cosets of  the vertex groups. Using the above product structure, each of $s$ and $t$  determine a level surface in each of the vertex groups corresponding to $u$ and $v$, and these two level surface will intersect in uniformly bounded subsets. Roughly, we define $W$ so that there is an edge between two vertices if these bounded diameter subsets associated to the two maximal simplices of $Y$ are close; see see Definition~\ref{defn:W} and Proposition~\ref{prop:P_vs_sigma} for details. This definition will make $W$ an equivariant quasi-isometric model for $ \pi_1 \mc{G}$.

The definition of edges in $W$ will ensure that the extra edges in $Y^{+W}$ are only added between vertices of $Y$ that are uniformly close under the collapse map $Y \to T$. Hence $Y^{+W}$ will be quasi-isometric to the Bass--Serre tree $T$ and hence hyperbolic. The other hyperbolic spaces coming from our combinatorial HHS are all either bounded diameter or correspond to one of the two factors of the product $L_\mu \times F_\mu$ for one of the vertex groups. One set of spaces will be quasi-isometric to the quasi-lines $L_\mu$, while the other will be quasi-isometric to hyperbolic cone-offs of the $F_\mu$. The quasi-isometric embedding conditions for these hyperbolic spaces are verified using a combination of closest point projection in the Bass--Serre tree $T$ with the hyperbolic geometry of the $F_\mu$ factor of each vertex group.

\subsection{Outline}
Section~\ref{sec:background} contains background on coarse geometry, graphs of groups and hierarchical hyperbolicity. This includes the definition of  an admissible graph of groups (Section \ref{subsec:graphs_of_groups}) and combinatorial HHS (Section \ref{subsec:HHS}). Section \ref{sec:main_results} presents the statements of our main results in more detail and deduces Theorem~\ref{thmi:main} from Theorem~\ref{thmi:main_technical}. The rest of the paper is devoted to the proof of Theorem~\ref{thmi:main_technical}.
In Section \ref{sec:quasilines}, we use quasimorphism to produce actions of central $\mathbb{Z}$--extensions on quasi-lines. In Section~\ref{sec:blow_up}, we construct the simplicial complex $Y$ and the graph $W$ that will comprise our combinatorial HHS for an \squidable graph of groups. Section~\ref{sec:proof} contains the proof that $(Y,W)$ is a combinatorial HHS. Section~\ref{subsec:combinatorial_conditions} focuses on  describing the link of simplices of $Y$ and verifying the combinatorial parts of  the definition of a combinatorial HHS. Section~\ref{subsec:hyperbolicity} contains the proof that $Y^{+W}$ and the $\link(\Delta)^{+W}$ are hyperbolic. Section~\ref{subsec:QI_embedding}   is devoted to checking that the inclusions $\link(\Delta)^{+W}\to Y_\Delta$ are quasi-isometric embeddings.

\subsection*{Acknowledgments}
Hagen was partly supported by EPSRC New Investigator Award EP/R042187/1. Russell was supported by NSF grant DMS-2103191. Spriano was partly supported by the Christ Church Research Centre. We would like to thank  the LabEx of the Institut Henri Poincar\'e (UAR 839 CNRS-Sorbonne Universit\'e) for their support during the trimester program ``Groups acting on Fractals, Hyperbolicity and Self-similarity''.   We thank the referee for their careful reading and numerous helpful comments.

\section{Preliminaries}\label{sec:background} 

\subsection{Coarse Geometry and Groups}
We recall some basic notions from coarse geometry and outline some techniques we will use repeatedly. For a metric space $Y$, we will use $d_Y$ to denote the distance in the space $Y$. The metric spaces we will consider will be undirected graphs, which we always equip with the path metric coming from declaring each edge to have length 1. For a graph $Y$, we let $Y^{(0)}$ denote the set of vertices of $Y$.

Let $\kappa \geq 1$, $\xi \geq 0$ and $f \colon Y \to Q$ be a map between metric spaces. The map $f$ is a \emph{$(\kappa,\xi)$--quasi-isometric embedding} if for all $x,y \in Y$ we have \[\frac{1}{\kappa} d_Q(f(x),f(y)) - \xi \leq d_Y(x,y) \leq \kappa d_Q(f(x),f(y)) +\xi. \] The map $f$ is \emph{$\xi$--coarsely onto} (or coarsely surjective) if for all $q \in Q$, there exists $y \in Y$ so that $d_Q(q ,f(y)) \leq \xi$. If $f$ is a $(\kappa,\xi)$--quasi-isometric embedding that is $\xi$--coarsely onto, we say $f$ is a \emph{$(\kappa,\xi)$--quasi-isometry}. A \emph{$\xi$--quasi-inverse} of $f$ is a map $h \colon Q \to Y$ so that $d_Y(y, h(f(y)) \leq \xi$ for each $y \in Y$. The map $f$ will be \emph{$(\kappa,\xi)$--coarsely Lipschitz} if \[ d_Y(x,y) \leq \kappa d_Q(f(x),f(y)) +\xi \] for all $x,y \in Y$. We often omit the constants when their specific value is not relevant. Note that the map $f$ is a quasi-isometry if and only if $f$ is coarsely Lipschitz and has a coarsely Lipschitz quasi-inverse (where the constants on either side of this equivalence determine the constants on the other).

A \emph{(quasi)-geodesic} in a metric space $Y$ is an (quasi)-isometric embedding of a closed interval $I \subseteq \mathbb{R}$ into $Y$. When $Y$ is a graph, we additionally require that the endpoints of the (quasi)-geodesic are vertices of $Y$.

At times it will be convenient to work with coarsely defined maps. A \emph{$\xi$--coarse map} from a metric space $Y$ to a metric space $Q$ is a function $f \colon Y \to 2^Q$ where for each $y \in Y$, $f(y)$ is a subset of $Q$ with diameter at most $\xi$. By a slight abuse of notation, we still write $f\colon Y \to Q$ to denote a coarse map. We say that a coarse map is coarsely Lipschitz, coarsely onto, a quasi-isometric embedding, a quasi-inverse or a quasi-isometry if it satisfies the same inequalities as described in the previous paragraph (where the distance between two sets is the minimal distance between two elements).

For graphs, we frequently use the following criteria to determine whether a map is coarsely Lipschitz and when an inclusion is a quasi-isometric embedding. The proofs are left as straightforward exercises.

\begin{lem}[Locally Lipschitz is Lipschitz]\label{lem:bounded_images_of_edges}
	For each $\xi \geq 0$ and $\kappa \geq 0$, there exists $\xi' \geq 0$ and $\kappa' \geq 1$ so that the following holds. Let $Y$ and $Q$ be graphs and suppose $f_0 \colon Y^{(0)} \to Q$ is a $\xi$--coarse map.  Let $f \colon Y \to Q$ be the map that extends $f_0$ by sending each edge $e$ of $Y$ to union of the images of the endpoints of $e$ under $f_0$. If  $d_Q(f_0(x),f_0(y)) \leq \kappa$ for each $x,y \in Y^{(0)}$ that are joined by an edge of $Y$, then $f$ is a $(\kappa',\kappa')$--coarsely Lipschitz $\xi'$--coarse map.
\end{lem}

\begin{lem}[Coarse retracts are undistorted]\label{lem:coarse_retract}
	Let $Y$ and $Q$ be graphs and assume there is an injective simplicial map $i \colon Q \to Y$. If there is a $(\kappa,\kappa)$--coarsely Lipschitz $\kappa$--coarse map $f \colon Y^{(0)} \to Q$ so that $f (i(Q)) = i(Q)$ and for each $ q \in Q$, $d_Q( q, i^{-1} \circ f \circ i (q)) \leq \kappa$, then the map $i \colon Q \to Y$ is a quasi-isometric embedding with constants determined by $\kappa$.
\end{lem}

We will apply Lemma \ref{lem:coarse_retract} exclusively in the case where $Q$ is a connected subgraph of $Y$. In this case, we emphasise that the map $f$ is coarsely Lipschitz with respect to the intrinsic path metric on $Q$ and not the metric the $Q$ inherits as a subset of $Y$. A map $f$ satisfying the conditions of Lemma \ref{lem:coarse_retract} is called  a \emph{coarse retract} of $Y$ to $Q$.

We say a graph $Y$ is \emph{$\delta$--hyperbolic} if for any geodesic triangle in $Y$, the $\delta$--neighborhood of any two sides covers the third side. Special cases of hyperbolic graphs are \emph{quasi-trees} and \emph{quasi-line}, which are graphs that are quasi-isometric to a tree or line respectively. We will need to use some ideas from the theory of relatively hyperbolic groups and spaces. Given a collection of coarsely connected subsets $\mc{Q}$ of a graph $Y$, we define the \emph{electrification of $Y$ with respect to $\mc{Q}$} to be the space obtained from $Y$ by adding an additional edge between $x,y \in Y^{(0)}$ whenever there is $Q \in \mc{Q}$ so that $x,y \in Q$. We denote this electrification by $\widehat{Y}_{\mc Q}$. We say that $Y$ is \emph{hyperbolic relative to $\mc{Q}$} if $\widehat{Y}_Q$ is $\delta$--hyperbolic for some $\delta \geq 0$ and if it satisfies the \emph{bounded subset penetration} property; see \cite[Definition 3.7]{Sisto_Met_Rel} for full details.

Many of the graphs we will study will be the Cayley graphs of groups. 

\begin{defn}
	Let $G$ be a group and $J$ be a symmetric generating set for $G$. We let $\cay(G,J)$ denote the simplicial graph whose vertices are the elements of $G$ and where two elements $g,h$ are joined by an edge if $g^{-1} h \in J$. 
\end{defn}

\noindent Note that the generating set $J$ does not need to be finite; in fact we will consider non-locally finite Cayley graphs throughout the paper. 

Suppose a group $G$ is acting by isometries on metric space $Y$. We say $G$ acts \emph{coboundedly} if there  exists a bounded set $B$ such that $G\cdot B = Y$. We say the  action of $G$ on $Y$ is \emph{metrically proper} if for any bounded diameter subset $K$ of $Y$, the set $\{g \in G : g\cdot K \cap K \neq \emptyset\}$ is finite. A version of the Milnor-Schwartz lemma say that if a finitely generated group $G$ groups act metrically proper and coboundedly on a metric space $Y$, then the orbit map gives a quasi-isometry $\cay(G,J) \to Y$ for any finite generating set $J$. 

A finitely generated group $G$ is \emph{hyperbolic} if for some (and hence any) finite generating set $J$, the  graph $\cay(G,J)$ is $\delta$--hyperbolic for some $\delta \geq 0$. A finitely generated group $G$ is \emph{hyperbolic relative} to a finite collection of subgroups $\{Q_1,\dots,Q_n\}$ if for some (and hence any) finite generating set $J$, the  graph $\cay(G,J)$ is hyperbolic relative to the collection of all cosets of the $Q_i$'s. In particular, the Cayley graph $\cay(G,J \cup Q_1\dots Q_n)$ is hyperbolic. 

The next lemma is a useful tool that allows to verify that the electrification of a quasi-tree with respect to quasiconvex subsets is again a quasi-tree.

\begin{lem}\label{lem:bottleneck}
	For all $\delta,\kappa\geq 1$ there exists $\delta'$ such that the following holds.  Let $\Gamma$ be a 
	graph that is $(\delta,\delta)$--quasi-isometric to a tree and $\mc{Q}$ be a collection of 
	$\kappa$--quasiconvex subspaces of $\Gamma$. Then  the electrification $\widehat{\Gamma}_\mc{Q}$ is 
	$(\delta',\delta')$--quasi-isometric to a tree.
\end{lem}

\begin{proof}
	We use the following consequence of Manning's bottleneck criterion~\cite[Theorem 4.6]{Manning:pseud}, formulated in~\cite[Section 3.6]{Old_BBF} (see also~\cite[Proposition 2.3]{Veech}):  a space is a quasi-tree if and only if there exists $\xi$ as follows: for any two points $x,y$, path $p$ between them and point $z$ on a geodesic between $x$ and $y$, we have $d(z, p)\leq \xi$. Moreover, the constants of the quasi-isometry to a tree and $\xi$ each determine the other
	
	Let $\xi$ be such a constant for the quasi-tree $\Gamma$ and let $\widehat{\Gamma} = \widehat{\Gamma}_\mc{Q}$. Our goal is to find an analogous $\widehat{\xi}$ for $\widehat{\Gamma}$. Let \(x,y\) be two points of $\Gamma^{(0)} = \widehat{\Gamma}^{(0)}$ and let $\widehat{\beta}$ be a \( \widehat{\Gamma}\)-geodesic between them. Let
	\(\widehat{z}\) be a point on \(\widehat{\beta}\) and \(\widehat{\gamma}\) be some path in $\widehat{\Gamma}$ connecting $x$ and $y$ between them. Let $\beta$ be a $\Gamma$--quasi-geodesic between $x,y$. 
	By \cite[Corollary 2.6]{KapovichRafi}, the Hausdorff distance in $\widehat{\Gamma}$ between $\beta$ and $\widehat{\beta}$ is uniformly bounded by some $R$. Thus, there exists $z \in \beta$ such that $d_{\widehat{\Gamma}}(\widehat{z}, z) \leq R$. Let $\gamma$ be the $\Gamma$--path obtained from $\widehat{\gamma}$ by replacing $\widehat{\Gamma}-\Gamma$ edges with geodesics of $\Gamma$. Since $\Gamma$ is a quasi-tree, there is a point $p\in \gamma$ with $d_{\Gamma}(p, z) \leq \xi$. If $p$ is also a point of $\widehat{\gamma}$ we are done.  Otherwise, $p$ is on a geodesic with endpoints on a $\kappa$-quasiconvex $Q_i$, we have $d_{\Gamma}(p, Q_i) \leq \kappa$. As $Q_i$ is coned-off in $\widehat{\Gamma}$ and $\widehat{p}$ intersects $Q_i$, we obtain $d_{\widehat{\Gamma}}(p, \widehat{\gamma})\leq \kappa+1$. 
	By the triangular inequality, $$d_{\widehat{\Gamma}}(\widehat{z}, \widehat{\gamma}) \leq d_{\widehat{\Gamma}}(\widehat{z}, z) + d_{\widehat{\Gamma}}(z,p) + d_{\widehat{\Gamma}}(p, \widehat{\gamma}).$$
	As each of the above quantities is uniformly bounded, we get the claim.
\end{proof}

We conclude with a lemma relating quotients and Cayley graphs with respect to infinite generators. 

\begin{lem}\label{lem:quotient_VS_cone-off}
	Let $G$ be a group and $N\unlhd G$ a normal subgroup. For any generating set $K$ of $G$ satisfying $N\subseteq K$ the quotient map $\pi\colon G \to G/N$ induces a $(2,1)$--quasi-isometry \[\pi \colon \mathrm{Cay}(G, K) \to \mathrm{Cay}(G/N, \pi(K)).\]
\end{lem}
\begin{proof}
	Let $\Gamma = \cay(G, K)$ and $\Omega = \cay(G/N, \pi(K) )$. By construction, the map $\pi$ gives a 1-Lipschitz map $\Gamma \to \Omega$. For each $x \in G/N$, let $\theta(x)$ be the be an element of the coset $gN$ in $G$ so that $\pi(gN) = x$. 
	Given any $x_1,x_2 \in G/N$ with $x_1^{-1} x_2 \in \pi(K)$ we can find $y_1$ in the same coset as $\theta(x_1)$ and $y_2$ in the coset as $\theta(x_2)$ so that $y_1^{-1}y_2 \in K$. Since  each coset $gN$ has diameter $1$ in $\Gamma$, we have $d_\Omega(\pi(x_1), \pi(x_2)) \leq d_\Gamma(x_1, x_2)\leq 2d_\Omega(\pi(x_1), \pi(x_2)) + 1$.
\end{proof}

\subsection{Graphs of groups}\label{subsec:graphs_of_groups}
We start with recalling some definitions and notations from Bass--Serre theory. For a comprehensive background, we refer the reader to \cite{ScottWall}.
Firstly, we recall that for Bass--Serre it is useful to use the language of bi-directed graphs.

\begin{defn}
	A \emph{bi-directed graph} $\Gamma$ consists of sets $V(\Gamma)$, $E(\Gamma)$ and maps 
	\begin{align*}
		E(\Gamma) &\to V(\Gamma) \times V(\Gamma);  &  E(\Gamma) &\to E(\Gamma) \\
		\alpha &\mapsto (\alpha^{+}, \alpha^{-}) &  \alpha & \mapsto \bar{\alpha}
	\end{align*}
	satisfying $\bar{\bar{\alpha}} = \alpha$, $\bar{\alpha}\neq \alpha$ and  $(\bar{\alpha})^{-} = \alpha^{+}$.
\end{defn}
The elements of $V(\Gamma)$ are called \emph{vertices},  the ones of $E(\Gamma)$ are called \emph{edges}, the vertex $\alpha^{-}$ is the \emph{source} of $\alpha$, $\alpha^{+}$ is the \emph{target} and $\bar{\alpha}$ is the \emph{reverse edge}. A bi-directed graph $\Gamma$ is \emph{finite} if both $V(\Gamma), E(\Gamma)$ are finite sets.
A \emph{subgraph} of $\Gamma$ is a bi-directed graph $\Gamma'$ such that $V(\Gamma') \subseteq V(\Gamma)$ and $E(\Gamma') \subseteq E(\Gamma)$. 
Given a bi-directed graph $\Gamma$, it is standard to associate to it a an undirected graph $\vert \Gamma \vert$, where the vertices are the elements of $V(\Gamma)$ and the edges are pairs of edges of the form $\{\alpha,\bar{\alpha}\}$. We call these pairs of edges $\{\alpha,\bar \alpha\}$  \emph{undirected edges} of $\mc{G}$. An \emph{orientation} on an undirected edge is choice of one of the directed edges. We say that a bi-directed graph $\Gamma$ is \emph{connected}, respectively a \emph{tree} if $\vert \Gamma \vert$ is connected, respectively a tree.
We say that a subgraph $T$ of $\Gamma$ is a \emph{spanning tree} if $V(T) = V(\Gamma)$ and $T$ is a tree.

The correspondence between $\Gamma$ and $\vert \Gamma \vert$ gives an equivalence between undirected graphs and bi-directed graphs.  The reason behind distinguishing the two classes is that the language of undirected graphs is more natural when considering graphs as metric spaces, whereas bi-directed graphs highlight combinatorial properties used to describe graphs of groups.

\begin{defn}
	A \emph{graph of groups} $\mc{G}$ consists of a finite connected bi-directed graph $\Gamma$, two collections of groups $\{G_\mu \mid \mu \in V(\Gamma)\}$ and $\{G_\alpha \mid \alpha \in E(\Gamma)\}$ satisfying $G_{\alpha} = G_{\bar{\alpha}}$, and injective homomorphisms $\tau_{\alpha} \colon G_{\alpha} \to G_{\alpha^{+}}$ for each $\alpha \in E(\Gamma)$. 
\end{defn}

\begin{defn}\label{defn:graph_of_groups}
	Let $\mc{G}= (\Gamma, \{G_\mu\}, \{G_\alpha\}, \{\tau_{\alpha}\})$ be a graph of groups. We define the group $F\mc{G}$ as:
	\[F\mc{G} = \left(\bigast_{\mu \in V(\Gamma)}G_\mu\right) \ast \left(\bigast_{\alpha \in E(\Gamma)} \langle t_{\alpha} \rangle\right).\]
	Let $\mathrm{Sp}(\Gamma)$ be a spanning tree of $\Gamma$. The \emph{fundamental group} of $\mc{G}$ with respect to $\mathrm{Sp}(\Gamma)$, denoted by $\pi_1(\mc{G},\mathrm{Sp}(\Gamma))$, is the group obtained adding the following relations to $F\mc{G}$:
	\begin{enumerate}
		\item $t_\alpha = t_{\bar{\alpha}}^{-1}$;
		\item $t_\alpha = 1$ if $\alpha \in E(\mathrm{Sp}(\Gamma))$;
		\item $t_\alpha \tau_{\alpha}(x)t_\alpha^{-1} = \tau_{\bar{\alpha}}(x)$ for all $x \in G_\alpha$. \label{item:edges_transition} 
	\end{enumerate}
\end{defn}

Given a graph of groups $\mathcal{G}$ we can associate to it the Bass--Serre tree $T$; see e.g.~\cite[Section 4]{ScottWall}. This is the bi-directed graph whose vertices are cosets of the vertex groups, and two cosets are joined by a directed edge if there are representatives $gG_\mu$ and $hG_\omega$ such that the vertices $\mu$ and $\omega$  are connected by an edge $\alpha$ with $\alpha^+ = \mu$  and $h t_\alpha = g $. For a vertex $v$ of $T$ we let $\check{v}$ denote the vertex $\mu$ of $\mc{G}$ so that $v = gG_\mu$. Similarly, given an edge $e$ of $T$, we define $\check{e}$ to be the edge of $\mc{G}$ joining $\check{e}^+$  and $\check{e}^-$.

If the vertex $v \in T^{(0)}$ is the coset $gG_\mu$, then stabiliser  $\stab_{\pi_1 \mc{G}}(v)$ is the conjugate of the vertex group $g G_\mu g^{-1}$. Similarly, for each edge $e$ of $T$, the stabiliser  $\stab_{\pi_1 \mc{G}}(e)$ is $g \tau_{\bar \alpha}(G_\alpha) g^{-1} = g t_\alpha \tau_\alpha(G_\alpha) t_\alpha^{-1} g^{-1}$ where $\check{e} = \alpha$, and $g$ is an element of $ \pi_1 \mc{G}$ so that $gG_{\check{e}^-}$ and $gt_\alpha G_{\check{e}^+}$ are the vertices $e^-$ and $e^+$ respectively.

Even though $T$ is a bi-directed graph, we will at times think of it as a metric space. When we do this, we are implicitly referring to $|T|$, the undirected graph obtained from $T$. We will use $E$ to denote unoriented edges of $T$ and $e$ to denote an orientation on $E$.

Given a graph of groups, we want to provide a geometric model that encodes the geometry of the entire fundamental group. To achieve this we will take the cosets of the vertex groups and join them together using the information coming from the tree and the edge group. We call the resulting space the \emph{Bass--Serre space}. In order to keep track of the geometry of the edge spaces, it is useful to introduce a combinatorial notion of edges with midpoints.

\begin{defn}
	A \emph{subdivided edge} is a (undirected) graph isomorphic to the graph with vertices \(v_0, v_1, v_2\) and edges  between \(v_0\) and \(v_{1}\), and between \(v_1\) and \(v_2\). The vertex \(v_1\) is called the \emph{middle vertex}. Two vertices \(x,y\) of a graph \(\Gamma\) are \emph{connected by a subdivided edge} if there is a subgraph of \(\Gamma\) isomorphic to a subdivided edge with \(v_0 = x\) and \(v_2 = y\).
\end{defn}

\begin{defn}[Bass--Serre space]\label{defn:BS_Space}
	Let $\mc{G}$ be a graph of finitely generated groups. For each vertex group $G_\mu$ and edge group $G_\alpha$ fix once and for all finite symmetric generating sets $J_\mu$ and  $J_\alpha$ respectively, such that $J_\alpha = J_{\bar{\alpha}}$ and $\tau_{\alpha}(J_\alpha) \subseteq J_{\alpha^{+}}$.  We build the  \emph{Bass--Serre space} $\BS$ for the graph of groups $\mc{G}$ in three steps.

	\textbf{Step 1: vertex spaces.} For each vertex $v = gG_\mu$ of $T$, we define \(X_v\) to be the graph with vertex set \(gG_{\mu}\) and with an edge between \(x,y \in gG_{\mu}\) if \(x^{-1}y\in J_\mu\).  We call each $X_v$ the \emph{vertex space} for $v \in T^{(0)}$. Because each vertex group injects into $ \pi_1 \mc{G}$, each $X_v$ is graphically isomorphic to  the Cayley graph of $G_{\check{v}}$ with respect to the generating set $J_{\check{v}}$.

	\textbf{Step 2: subdivided edges.} Given an undirected edge $E$ of $T$, pick an orientation $e \in E$ and let $\alpha = \check{e}$,  $\mu = \alpha^+$,  and $\omega = \alpha^-$.  Fix an element $g \in  \pi_1 \mc{G}$ so that $g G_{\omega} = X_{e^-}$ and $gt_\alpha G_\mu = X_{e^+}$. For each $a \in G_{\alpha}$, add a subdivided edge between  $ g\tau_{\bar \alpha}(a) \in g G_\omega = X_{e^-}$ and $g t_\alpha \tau_\alpha(a) \in gG_\mu = X_{e^+}$. By Definition \ref{defn:graph_of_groups}.\eqref{item:edges_transition}, if $x =g \tau_{\bar{\alpha}}(a)$, then $xt_\alpha = g \tau_{\bar{\alpha}}(a) t_\alpha = g t_\alpha \tau_{\alpha}(a)$.  Hence, all such $x$ and $xt_\alpha$ are joined by a subdivided edge and the addition of these subdivided edges does not rely on our specific choice of representative $g \in  \pi_1 \mc{G}$. Since $t_\alpha^{-1} = t_{\bar \alpha}$, the addition of these subdivided edges is also independent of the orientation chosen for $E$.
	
	\textbf{Step 3: edges spaces.} Let $E$ be an undirected edge of \(T\) with orientation $e$.  Let $e^+ = v$, $e^- =w$, and $\check{e}= \alpha$. For each subdivided edge added between \(X_v\) and \(X_w\) there is a middle vertex. Let \(a, b\) be two of these middle vertices and \(x, y\) the vertices of \(X_v\) adjacent to them. To complete the Bass--Serre space, we add an edge between any two such $a,b$ if $x^{-1}y \in \tau_\alpha(J_\alpha) $ in $ \pi_1 \mc{G}$. This is independent of the orientation for $E$ because if $p,q$ are the vertices of $X_w$ adjacent to $a$ and $b$ respectively, then $p = x t_\alpha$ and $q = y t_\alpha$. Thus $p^{-1}q = t_\alpha^{-1}x^{-1} y t_\alpha$, which implies $x^{-1}y \in \tau_\alpha(J_\alpha) $ if and only if $p^{-1}q \in \tau_{\bar \alpha}(J_\alpha)$ by Definition \ref{defn:graph_of_groups}.\eqref{item:edges_transition}.  
	
	For each (directed) edge $e$ in $T$, we use $X_e$ to denote the (undirected) graph whose vertices are all of the middle vertices of the subdivided edges between $X_{e^+}$ and $X_{e^-}$ with the edges defined as above. We call $X_e$ the \emph{edge space} for $e$ and note that $X_e = X_{\bar e}$. Each edge space $X_e$ is graphically isomorphic to the Cayley graph of the edge group $G_{\check{e}}$ with generating set $J_{\check{e}}$. We let  \(\tau_e \colon X_e\to X_{e^+}\) denote be the map that associates to each middle vertex the only vertex of \(X_v\) adjacent to it, and we define \(\tau_{\bar{e}}\) analogously.  Figure \ref{fig:edge_space} gives a schematic of the edge spaces and $\tau_e$ maps.
	
	The  \emph{Bass--Serre space} $\BS$ for the graph of groups $\mc{G}$ is the  space constructed from taking all the vertex spaces in Step 1, adding in all the subdivided edges from Step 2, and then adding in all the edges of the edge spaces in Step 3. The group $ \pi_1 \mc{G}$ acts on the disjoint union of the vertex spaces by left multiplication. This action can be extended to the subdivided edges and edge spaces to give an action of $ \pi_1 \mc{G}$ of $\BS$ by isometries. The edge space maps $\tau_e$ and $\tau_{\bar{e}}$ are equivariant with respect to this action.
\end{defn}

\begin{remark}\label{rmk:tau_coarsely_Lipschitz}
	For every $x,y\in X_e$, we have $d_X(x, \tau_e(x)) = 1$ and  $d_{X_{e^+}}(\tau_e(x),\tau_e(y)) \leq d_{X_{e}}(x,y) +2 $.
\end{remark}

\begin{center}
	\begin{figure}
		\begin{tikzpicture}
			\node[anchor=south west,inner sep=0] (image) at (0,0) {\includegraphics[width=0.6\textwidth]{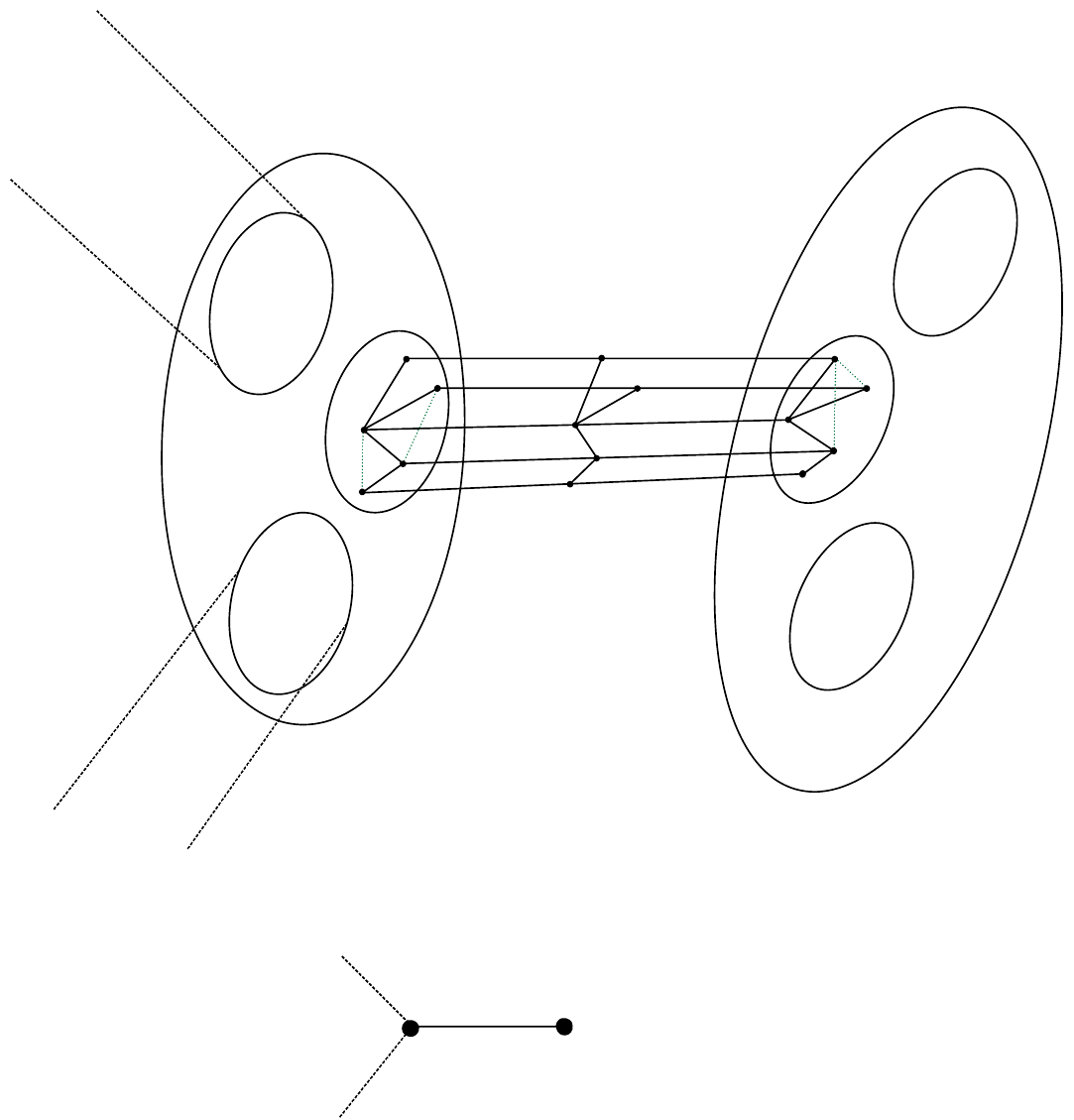}};
			\begin{scope}[x={(image.south east)},y={(image.north west)}]
				\node at (0.3, 0.9) {$X_v $};
				\node at (0.55, 0.71) {$X_e $};
				\node at (0.9, 0.95) {$X_w $};
				\node at (0.39, 0.11) {$v$};
				\node at (0.46, 0.11) {$e $};
				\node at (0.53, 0.11) {$w $};
				\node at (0.23,0.6) {$\tau_e(X_e)$};
			\end{scope}
		\end{tikzpicture}\caption{The cosets corresponding to the edge $e$ are connected by a subdivided edge. In the picture, we assume that $e^+ = v$. To every edge of $X_e$ corresponds an edge in $X_v, X_w$, but additional edges might be present.}\label{fig:edge_space}
	\end{figure}
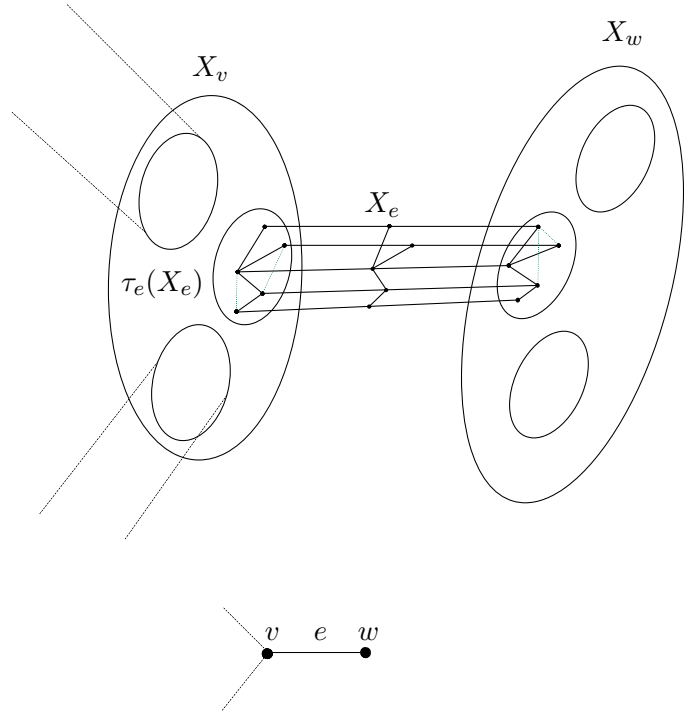
\end{center}

While the inclusion of the vertex and edge spaces in to the Bass--Serre space are simplicial injections, their images maybe very distorted in the total metric on $X$. However, as there are only finitely many vertex and edge groups, we have uniform control over this distortion

\begin{lem}\label{lem:uniform_distortion}
	Let $\mc{G}$ be a graph of groups with Bass--Serre tree $T$ and Bass--Serre space $\BS$. There exists a monotone diverging function $h \colon [0,\infty) \to [0,\infty)$ so that for each vertex $v$ and edge $e$ of $T$ we have
	\[ d_{X_v} (x,y)  \leq h \bigl( d_X(x,y) \bigr) \text{ and }  d_{X_e} (x,y)  \leq h \bigl( d_X(x,y) \bigr) \]
	for any $x,y \in X_v$ or $x,y \in X_e$.
\end{lem}

\begin{proof}
	For each $v \in T^{(0)}$, define $h_v \colon [0,\infty) \to [0,\infty)$ to be $$h_v(r) = \max_{\{x,y \in X : d_{X_v}(x,y) \leq r\}} \{ d_{X_v}(x,y)\}.$$ Because  $X$ is locally finite and $ \pi_1 \mc{G}$ acts transitively on the vertices of the vertex and edge spaces respectively,  $h_v$ exists and is a monotone diverging function.  We similarly define $h_e$ for each edge $e$ of $T$. If two vertices, $v$ and $w$,  or two edges, $e$ and $f$, are in the same $ \pi_1 \mc{G}$--orbit then $h_v = h_w$ and $h_e = h_f$. Since $ \pi_1 \mc{G}$ acts of $T$  with finitely many orbits of edges and vertices, we can find the desired $h$ by taking the minimum over all of these finitely many orbits.
\end{proof}

Croke and Kleiner introduce the following class of \emph{admissible} graphs of groups to abstract the properties of the graphs of groups structure of the fundamental groups of non-geometric graph manifolds \cite{CrokeKleiner}. This will be the class of graphs of groups that we will study.

\begin{defn}\label{defn:admissible}
	Let $\mc{G} = (\Gamma, \{G_\mu\}, \{G_\alpha\}, \{\tau_{\alpha}\})$ be a graph of groups.  We say $\mc{G}$ is \emph{\squidable} if the following hold:
	\begin{enumerate}
		\item \label{item:edges_exist} $\Gamma$ contains at least 1 edge.
		\item\label{item:center_and_quotient} Each vertex group $G_\mu$ has center $Z_\mu$ that is a infinite cyclic group, and $G_\mu / Z_\mu = F_\mu$ is a non-elementary hyperbolic group. 
		\item\label{item:edge_groups} Each edge group $G_\alpha$ is isomorphic to $\mathbb{Z}^2$. 
		\item\label{item:finite_index} If $\alpha$ is an edge with $\mu = \alpha^+$ and $\omega = \alpha^-$, then $\langle \tau_{\alpha}^{-1}(Z_\mu),\tau_{\bar \alpha}^{-1}(Z_\omega) \rangle$ is a finite index subgroup of $G_\alpha \cong \mathbb{Z}^2$. 
		\item\label{item:non_commensurable} If $\alpha_1$,$\alpha_2$ are distinct edges with $\alpha_1^+ = \alpha_2^+$, then
		\begin{itemize}
			\item for each $g \in  \pi_1 \mc{G}$, $g\tau_{\alpha_1}(G_{\alpha_1})g^{-1}$ is not commensurable with $\tau_{\alpha_2}(G_{\alpha_2})$;
			\item for each $g \in  \pi_1 \mc{G} - \tau_{\alpha_1}(G_{\alpha_1})$, $g\tau_{\alpha_1}(G_{\alpha_1})g^{-1}$ is not commensurable with $\tau_{\alpha_1}(G_{\alpha_1})$  .
		\end{itemize}
		
	\end{enumerate}
\end{defn}

We conclude this section with a few basic consequence of Definition \ref{defn:admissible}. First we apply a theorem of Bowditch to obtain that the hyperbolic quotients, $F_\mu$, are actually hyperbolic relative to the  subgroups coming from the incident edge groups.

\begin{lem}\label{lem:F_mu_are_relatively_hyp}
	Let $\mc{G}$ be an \squidable graph of groups. For each vertex $\mu$, let $\pi_\mu$ be the quotient map $\pi_\mu \colon G_\mu \to F_\mu$, where $F_\mu$ is the quotient $G_\mu / Z_\mu$. The group $F_\mu$ is hyperbolic relative to the collection $\{ \pi_\mu(\tau_\alpha(G_\alpha)) : \alpha \text{ is an edge with } \alpha^+ = \mu\}$.
\end{lem}

\begin{proof}

	Let $I_\mu$ be the set of edges $\alpha$ of $\mc{G}$ with $\alpha^+ = \mu$ and let $A_\alpha = \tau_\alpha(G_\alpha)$ for each $\alpha \in I_\alpha$. We want to show that $\{\pi_\mu(A_\alpha)\colon \alpha \in I\mu\}$ is an almost malnormal collection of quasiconvex subgroups as this implies $F_\mu$ is relatively hyperbolic by \cite[Theorem 7.11]{Bowditch_Rel_Hyp}. 
	
	We first establish that $\pi_\mu(A_\alpha)$ is virtually cyclic for each $\alpha\in I_\mu$.   By construction,  $\pi_\mu(A_\alpha)$ is the quotient of $A_\alpha$ by $A_\alpha \cap Z_\mu$.
	Since $A_\alpha \cong \mathbb{Z}^2$ and  $F_\mu$ is hyperbolic, $A_\alpha$ must intersect $Z_\mu$ in a non-trivial subgroup. Hence $\pi_\mu(A_\alpha)$ must be virtually cyclic. Note, this implies each $\pi_\mu(A_\alpha)$ is quasiconvex in $F_\mu$ as virtually cyclic subgroups of hyperbolic groups are always quasiconvex.
	
	We now show the set $\{\pi_\mu(A_\alpha) : \alpha \in I_\mu\}$ is an almost malnormal collection of subgroups of $F_\mu$. Since each $\pi_\mu(A_\alpha)$ is virtually cyclic, if the collection fails to be almost malnormal,  there must be $\alpha_1,\alpha_2 \in I_\mu$ so that some conjugate of $\pi_\mu(A_{\alpha_1})$ is commensurable to $\pi_\mu(A_{\alpha_2})$ in $F_\mu$. Because $Z_\mu \cong \mathbb{Z}$ and each $A_\alpha \cong \mathbb{Z}^2$, this would imply a conjugate of $A_{\alpha_1}$ is commensurable to $A_{\alpha_2}$ in $ \pi_1 \mc{G}$. As this would contradict Definition \ref{defn:admissible}.\eqref{item:non_commensurable}, we must have that $\{\pi_\mu(A_\alpha) : \alpha \in I_\mu\}$ is an almost malnormal. The lemma now follows by applying \cite[Theorem 7.11]{Bowditch_Rel_Hyp}.
\end{proof}

\noindent Lastly, we note that by choosing appropriate infinite generating sets for the vertex groups $G_\mu$, we can make Cayley graphs that are quasi-isometric to the hyperbolic groups $F_\mu$ as well as the electrification of $F_\mu$ by the cyclic subgroups from the incident edge groups.
Recall each vertex group is a central extension $Z_\mu \to G_\mu \to F_\mu$ where $Z_\mu$ is  cyclic and $F_\mu$ is hyperbolic.

\begin{lem}\label{lem:cone-off_edge_groups}
	Let $\mc{G}$ be an \squidable graph of groups. Let $I_\mu$ be the set of edges $\alpha$ of $\mc{G}$ with $\alpha^+ = \mu$, then let $\mc{E}_\mu = \bigcup_{\alpha \in I_\mu} \tau_\alpha (G_\alpha)$. For each finite generating set $J_\mu$  of $G_\mu$ we have:
	\begin{enumerate}
		\item The quotient map $\pi_\mu \colon G_\mu \to F_\mu$ induces a quasi-isometry $$\pi_\mu \colon \cay(G_\mu, J_\mu \cup Z_\mu) \to \cay(F_\mu, \pi_\mu(J_\mu)),$$ in particular $\cay(G_\mu, J_\mu \cup Z_\mu)$ is hyperbolic and hyperbolic relative to the collection $\{g\tau_\alpha(G_\alpha) : \alpha \in I_\mu \text{ and } g \in G_\mu\}$.
		\item The quotient map $\pi_\mu \colon G_\mu \to F_\mu$ induces a quasi-isometry $$\pi_\mu \colon \cay(G_\mu, J_\mu \cup \mc{E}_\mu) \to \cay(F_\mu, \pi_\mu(J_\mu \cup \mc{E}_\mu)).$$ Hence, $ \cay(G_\mu, J_\mu \cup \mc{E}_\mu)$ is hyperbolic and will be a quasi-tree whenever $F_\mu$ is virtually free.
	\end{enumerate} 
	The quasi-isometry constants are independent of $\mc{G}$.
\end{lem}
\begin{proof}
	The fact that the map are quasi-isometries follows from Lemma \ref{lem:quotient_VS_cone-off}. The first relative hyperbolicity follows from Lemma~\ref{lem:F_mu_are_relatively_hyp}. For the second, since $F_\mu$ is hyperbolic relative to the subgroups $\{\pi_\mu(\tau_\alpha(G_\alpha)) : \alpha \in I_\mu\}$ (Lemma~\ref{lem:F_mu_are_relatively_hyp}),  the graph $\cay(F_\mu, \pi_\mu(J_\mu \cup \mc{E}_\mu))$ is hyperbolic. Moreover, if $F_\mu$ is virtually free, then $\cay(F_\mu,\pi_\mu(J_\mu))$ is a quasi-tree. Hence the fact that $\cay(F_\mu, \pi_\mu(J_\mu \cup \mc{E}_\mu))$ is a quasi-tree  is a consequence of Lemma~\ref{lem:bottleneck}.\end{proof}

\subsection{Hierarchically hyperbolic groups}\label{subsec:HHS}
As we will not directly require the full definition of a hierarchically hyperbolic space, we will only review the necessary data to define a hierarchically hyperbolic group. We direct the reader to \cite{BHSII} or \cite{HHS_Survey} for complete details on the HHS axioms.

Fix $E \geq 1$. An \emph{$E$--hierarchically hyperbolic space (HHS) structure} on a geodesic metric space $\mc{X}$ starts with a set $\mf{S}$ indexing a collection of $E$--hyperbolic spaces $\{\mc{C}(V)\}_{V \in \mf{S}}$. For each $V \in \mf{S}$, there is an $(E,E)$--coarsely Lipschitz, $E$--coarsely surjective \emph{projection map} $\varphi_V \colon \mc{X} \to \mc{C}(V)$. The set $\mf{S}$ is also equipped with three combinatorial relations:  nesting ($\nest$), orthogonality ($\perp$), and  transversality ($\trans$). To be a hierarchically hyperbolic space structure for $\mc{X}$, the set $\mf{S}$ and these relations and projections need to satisfy a number of axioms. The most relevant for us are:
\begin{itemize}
	\item Every pair of distinct elements of $\mf{S}$ is related by exactly one of $\nest$, $\perp$, or $\trans$.
	\item $\trans$ and $\perp$ are both symmetric, while $\nest$ is a partial order.
	\item If $V \perp W$ and $U \nest V$, then $U \perp W$.
	\item If $V \propnest W$ or $V \trans W$, then there exists a distinguished subset $\rho_W^V \subseteq \mc{C}(W)$ with diameter at most $E$.
\end{itemize}

We use $\mf{S}$ to denote the entire HHS structure (the spaces, projections, relations, and distinguished subsets) and the pair $(\mc{X},\mf{S})$ to denote the hierarchically hyperbolic space $\mc{X}$ equipped with the specific HHS structure $\mf{S}$.
An HHS structure can be transferred across a quasi-isometry $f \colon \mc{Y} \to \mc{X}$, by replacing the projection maps $\varphi_V$ with $\varphi_V \circ f$. 
In particular, if a finitely generated group $G$  acts  metrically properly and coboundedly on an HHS $(\mc{X},\mf{S})$, then $\mf{S}$ is also a hierarchically hyperbolic space structure for $G$ equipped with any word metric (or equivalently any Cayley graph of $G$ with respect to a finite generating set). However, the maps and relations defining $\mf{S}$ need not be equivariant with respect to the action of $G$. If the HHS structure is compatible with the group action, then we can have the following stronger definition of a hierarchically hyperbolic \emph{group}.

\begin{defn}\label{defn:HHG}
	Suppose a finitely generated group $G$ is acting isometrically on an $E$--hierarchically hyperbolic space $(\mc{X},\mf{S})$. We say $\mf{S}$ is  an $E$--\emph{hierarchically hyperbolic group} structure if
	
	\begin{enumerate}
		\item \label{item:geometric_action} $G$ acts metrically properly and coboundedly on $\mc{X}$.
		\item\label{item:HHG_bijections} There is an $\nest$, $\perp$, and $\trans$ preserving action of $G$ on the index set $\mf{S}$ by bijections.
		\item\label{item:HHG_finite_orbits} $\mf{S}$ has finitely many $G$--orbits.
		\item\label{item:HHG_equivariance} For each $V \in \mf{S}$ and $g\in G$, there exists an isometry $g_V \colon \mc{C}(V) \rightarrow \mc{C}(gV)$ satisfying the following for all $V,U \in \mf{S}$ and $g,h \in G$.
		\begin{itemize}
			\item The map $(gh)_V \colon \mc{C}(V) \to \mc{C}(ghV)$ is equal to the map $g_{hV} \circ h_V \colon \mc{C}(V) \to \mc{C}(hV)$.
			\item For each $x \in \mc{X}$, $g_V(\varphi_V(x)) = \varphi_{gV}(g \cdot x)$. 
			\item If $V \trans U$ or $V \propnest U$, then $g_V(\rho_U^V) = \rho_{gU}^{gV}$. 
		\end{itemize}
	\end{enumerate}
	We say $G$ is a \emph{hierarchically hyperbolic group} (HHG) if there exists an HHS $(\mc{X},\mf{S})$ so that $\mf{S}$ is an $E$--HHG structure for $G$ for some $E \geq 1$.
\end{defn}

Modulo the incompleteness of our description of a hierarchically hyperbolic space structure, the above definition of a hierarchically hyperbolic group is precise.  

There are examples of  finitely generated groups that have hierarchically hyperbolic \emph{space}  structures, but do not have any hierarchically hyperbolic \emph{group} structures. In fact, there are groups that are not HHGs, but have finite index subgroups that are HHGs \cite{PetytSpriano}.

We will need the following proposition, originally due to Paul Plummer, to check that certain 3--manifold groups are not HHG.

\begin{prop}[Invariant quasiflats for virtually abelian subgroups]\label{lem:flat_torus}
	Let $(G,\mathfrak S)$ be an HHG.  Let $A\subset G$ be a virtually $\Z^k$ subgroup for some $k\geq 1$.  Then there exists $\ell\geq k$ and $U_1,\ldots,U_\ell\in\mathfrak S$ such that the 
	following hold:
	\begin{enumerate}
		\item $\{U_1,\ldots,U_\ell\}$ is $A$--invariant.
		\item $U_i\orth U_j$ for $1\leq i<j\leq \ell$.
		\item There exists $L<\infty$ such that $\diam(\varphi_V(A))\leq L$ for $V\in\mathfrak S-\{U_1,\ldots,U_\ell\}$.
		\item For each $i\leq \ell$, the image $\varphi_{U_i}(A)$ of $A$ in $\fontact (U_i)$ is a quasi-line.
	\end{enumerate}
	Hence the ($A$--invariant) hierarchically quasiconvex hull $F_A$ of $A$ is quasi-isometric to $\Z^\ell$.
\end{prop}

Hierarchically quasiconvex hulls are discussed in~\cite[Section 6]{BHSII}.  

\begin{proof}[Proof of Proposition~\ref{lem:flat_torus}]
	We adopt the standard convention that for $a,b \in G$ and $V \in \mf{S}$, $d_V(a,b)$ denotes $d_V(\varphi_v(a),\varphi_V(b))$. Let $\mathbbm{1}$ denote the identity in $G$ and equip both $A$ and $G$ with word metrics from finite generating sets.
	
	Apply \cite[Theorem 5.1]{PetytSpriano} to obtain a nonempty $A$--invariant set of elements $U_1,\ldots,U_\ell\in\mathfrak S$ such that
	\begin{itemize}
		\item $U_i\orth U_j$ for $1\leq i<j\leq \ell$;
		\item if $W\in\mathfrak S$ has the property that $\varphi_W(A)$ is unbounded, then $W\nest U_i$ for some $i$;
		\item $\varphi_{U_i}(A)$ is unbounded for each $i\leq \ell$.
	\end{itemize}
	
	\textbf{Each $\varphi_{U_i}(A)$ is a quasiline:}  Since $\ell<\infty$, there is a finite-index subgroup $\ddot A\leq A$ such that $\ddot A\cdot U_i=U_i$ for all $i$.  Assume, by 
	passing to a further finite-index subgroup, that $\ddot A\cong\Z^k$.  In particular, $\ddot A$ acts on each of the $E$--hyperbolic spaces $\mc{C}(U_i)$.
	
	Since $\ddot A$ has finite index in $A$, and $\varphi_{U_i}$ is $(E,E)$--coarsely lipschitz and $\ddot A$--equivariant, we have that $\varphi_{U_i}(A)$ 
	and 
	$\ddot A\cdot\varphi_{U_i}(\mathbbm{1})$ lie at finite Hausdorff distance. In particular, the above choice of the $U_i$ implies the orbit $\ddot A\cdot \varphi_{U_i}(\mathbbm{1})$ is unbounded for each $U_i$.  Proposition 3.1 of~\cite{CCMT} therefore provides four options for the action of $\ddot A$ on $\mc{C}(U_i)$: \emph{focal}, \emph{general}, \emph{horocyclic}, or \emph{lineal}. We verify the the action must be lineal.
	
	Since $\ddot A$ is abelian, it does not contain a free sub-semigroup and hence the action on $\mc{C}(U_i)$ action cannot be \emph{focal} or \emph{general}. By~\cite[Theorem 3.1]{DHS_corr}, any infinite-order element of $\ddot A$ is loxodromic on $\fontact (U_i)$, so the action is not 
	\emph{horocyclic}.  Hence the action is \emph{lineal}. In particular, the orbit $\ddot A\cdot \varphi_{U_i}(\mathbbm{1})$, with the 
	metric inherited from $\mc{C}(U_i)$, is $(C,C)$--quasi-isometric to $\Z$ and $C$--quasiconvex, 
	where $C$ depends on $\ddot A$ and the HHS constant $E$. Up to enlarging $C$, we can assume that $\varphi_{U_i}(A)$ is a $C$--quasiconvex $(C,C)$--quasiline.  Moreover, since there are finitely many $i$, 
	we 
	can assume that the same constant $C$ works for all $i$.

	\textbf{Bounding remaining domains:}  We now bound the diameter of $\varphi_V(A),\ V\not\in\{U_1,\ldots,U_\ell\}$.
	
	\begin{claim}\label{claim:all_bounded}
		There exists $L\ge0$ such that $\diam(\varphi_V(A))\leq L$ for all $V\in\mathfrak S-\{U_1,\ldots,U_\ell\}$.
	\end{claim}
	
	\begin{proof}[Proof of Claim~\ref{claim:all_bounded}]
		It suffices to prove the claim for the finite-index subgroup $\ddot A$ of $A$, since the maps $\varphi_V$ are all $(E,E)$--coarsely lipschitz.  Choose $a_1,\ldots,a_k\in 
		A$ such that $a_1,\ldots,a_k$ generate the finite-index subgroup $\ddot A$ isomorphic to $\Z^k$.  For any $g \in G$, let $\operatorname{Big}(g)$ be the set $\{W \in \mf{S} : \diam(\varphi_W(\langle g \rangle)) = \infty\}$. By \cite[Lemma 6.7]{DHS_boundary}, $\operatorname{Big}(g)$ is a finite, pairwise orthogonal subset of $\mf{S}$ for any $g \in G$. Moreover, $\operatorname{Big}(g)$ is non-empty whenever $g$ has infinite order by \cite[Proposition~6.4]{DHS_boundary}.
		
		We claim that $\operatorname{Big}(a_i)$ is a non-empty subset of $\{U_1,\ldots,U_\ell\}$ for all $i$.  Let $W \in \operatorname{Big}(a_i)$. By \cite[Theorem 5.1]{PetytSpriano} or \cite[Lemma 6.3, Proposition 6.4]{DHS_boundary}, there is $m \in \mathbb{N}$ so that $a_i^m$ fixes $W$ and has unbounded orbits on $\mc{C}(W)$.
		By the choice of the $U_j$, 
		there exists $j$ such that $W\nest U_j$.  If $W \neq U_j$, then $W\propnest U_j$. Hence, $\rho^W_{U_j}$ is defined and is a subset of $\fontact (U_j)$ of diameter at most $E$. 
		Since $\ddot A$ has unbounded orbits in $\fontact( U_j)$, there is  $g\in 
		\ddot A$ such that $d_{U_j}(\rho^W_{U_j},g\rho^W_{U_j})>10^9E$. By the definition of an HHG and the fact that $\ddot A$ 
		fixes $U_j$, we have $g\rho^W_{U_j}=\rho^{gW}_{U_j}$, so $gW\neq W$.  Now, $ga^m_ig^{-1}$ has unbounded orbits on $\fontact (gW)$, but 
		$ga^m_ig^{-1}=a^m_i$. Hence, $W,gW\in \operatorname{Big}(a_i)$, but they are not orthogonal by~\cite[Lemma 1.5]{DHS_boundary}.  This contradicts that the elements of  $\operatorname{Big}(a_i)$ are pairwise orthogonal. Hence, $W = U_j$.
		
		Since we have shown that $\operatorname{Big}(a_i)\subseteq\{U_1,\ldots,U_\ell\}$ for all $i$, \cite[Proposition 6.4]{DHS_boundary} provides  a constant $D(a_i)$ such that $\diam(\varphi_V(\langle a_i\rangle))\leq 
		D(a_i)$ for all $V \in \mf{S}-\{U_1,\dots, U_\ell\}$.  Let $D=\max_{1\leq i\leq k}D(a_i)$.  For any $b\in\ddot A$, write $b=a_1^{n_1}\cdots a_k^{n_k}$.  Since 
		$a_1^{-n_1}V\not\in\{U_1,\ldots,U_k\},$ we have $$d_V(\mathbbm{1},b)\leq d_{a_1^{-n_1}V}(\mathbbm{1},a_2^{n_2}\cdots a_k^{n_k})+d_{a_1^{-n_1}V}(\mathbbm{1},a_1^{-n_1})\leq d_{a_1^{-n_1}V}(\mathbbm{1},a_2^{n_2}\cdots 
		a_k^{n_k})+D,$$
		and we get $d_V(\mathbbm{1},b)\leq kD$ by induction.  This bounds $\diam(\varphi_V(\ddot A))$, which proves the claim.
	\end{proof}
	
	This proves the enumerated statements.  The distance formula in a HHG~\cite[Theorem 4.5]{BHSII} now shows that the hull $F_A$ of $A$ is quasi-isometric to the product $\prod_{i=1}^\ell\varphi_{U_i}(A)$, i.e., to 
	the 
	product of $\ell$ quasilines, i.e., to $\Z^\ell$.  Since $\ddot A\cong\Z^k$ acts properly on $F_A$, we must have $k\leq \ell$.
\end{proof}

\subsection{Combinatorial hierarchical hyperbolicity}
To verify that our \squidable groups are hierarchically hyperbolic groups, we will employ the \emph{combinatorial hierarchical hyperbolicity} machinery introduced in~\cite{CHHS}. This allows us to forgo checking the axioms directly, and instead extract hierarchical hyperbolicity from an action on a well chosen simplicial complex. We  recall the required  definitions and theorems for this approach.

\begin{defn}[Join, link, and star]
	Let $Y$ be a flag simplicial complex. If $Q,Z$ are disjoint flag subcomplexes of $Y$ so that every vertex of $Q$ is joined by an edge to $Z$, then the \emph{join} of $Q$ and $Z$, $Q \star Z$, is the subcomplex of $Y$ spanned by $Q$ and $Z$. Given a simplex $\Delta$ of $Y$, the \emph{link} of $\Delta$, $\lk(\Delta)$, is the subcomplex of $Y$ spanned by the vertices of $Y$ that are joined by an edge to all the vertices of $\Delta$. The \emph{star} of $\Delta$, $\st(\Delta)$, is the join $\Delta \star \lk(\Delta)$. We consider $\emptyset$ as a simplex of $Y$ whose link and star are both $Y$.
\end{defn}

\begin{defn}
	Given a flag simplicial complex $Y$, a \emph{$Y$--graph} is any graph $W$ whose vertices are maximal simplices of $Y$. Here maximal means not contained in a larger simplex.
	
	If $W$ is a $Y$--graph for the flag simplicial complex $Y$, we define the $W$--augmented graph $Y^{+W}$ as the graph with the same vertex set as $Y$ and with two types of edges: 
	\begin{enumerate}
		\item ($Y$--edge) If two vertices $y_1,y_2 \in Y$ are joined by an edge in $Y$, then $y_1$ and $y_2$ are joined by an edge in $Y^{+W}$.
		\item ($W$--edge) If $\Delta_1$ and $\Delta_2$ are maximal simplices of $Y$ that are joined by an edge in $W$, then each vertex of $\Delta_1$ is joined by an edge to each vertex of $\Delta_2$ in $Y^{+W}$.
	\end{enumerate}
	We note that if a group $G$ acts by simplicial automorphisms on $Y$ that is an isometry of $Y^{+W}$, then there is an induced action by isometries of $G$ on $W$.
\end{defn}

\begin{defn}\label{defn:saturation}
	Let $\Delta$ and $\Delta'$ be simplices of the flag simplicial complex $Y$. We write $\Delta \sim \Delta'$ if $\lk(\Delta) = \lk(\Delta')$. We define the \emph{saturation} of $\Delta$, $\Sat(\Delta)$, to be the set of vertices of $Y$ contained in a simplex in the $\sim$--equivalence class of $\Delta$. That is $x \in \Sat(\Delta)$  if and only if there exists $\Delta'  \sim \Delta$  so that $x$ is  a vertex of $\Delta'$.
\end{defn}

\begin{defn}\label{defn:Y_Delta}
	Let $W$ be a $Y$--graph. For each simplex $\Delta$ of $Y$, define $Y_\Delta$  to be the subgraph of $Y^{+W}$ spanned by the vertices of $Y^{+W} - \Sat(\Delta)$. 
	
	Define $\mc{C}(\Delta)$ to be the subgraph of $Y_\Delta$ spanned by the vertices in $\lk(\Delta)$. Note, we are taking the link in $Y$, not in $Y^{+W}$, and then considering the subgraph of $Y_\Delta$ induced by those vertices. We give $\mc{C}(\Delta)$ its intrinsic path metric (as opposed to the metric induced as a subset of $Y_\Delta$). By construction, we have $\mc{C}(\Delta) = \mc{C}(\Delta')$ whenever $\Delta \sim \Delta'$.	Note, since $\emptyset$ is a simplex of $Y$ with $\lk(\emptyset) = Y$, we have $Y_\emptyset = \mc{C}(\emptyset) = Y^{+W}$. 
\end{defn}

\begin{defn}\label{defn:CHHS}
	Let $\delta \geq 0$, $Y$ be a flag simplicial complex and $W$ be a $Y$--graph. The pair $(Y,W)$  is a \emph{$\delta$--combinatorial HHS} if the following are satisfied.
	\begin{enumerate}[(I)]
		\item \label{CHHS:finite_complexity} Any chain of the form $\lk(\Delta_1) \subsetneq \lk(\Delta_2) \subsetneq \dots$ has length at most $\delta$.
		\item \label{CHHS:hyp_X} For each non-maximal simplex $\Delta \subset Y$, the space $\mc{C}(\Delta)$ is $\delta$--hyperbolic.
		\item \label{CHHS:geom_link_condition} For each non-maximal simplex $\Delta$, the inclusion $\mc{C}(\Delta) \to Y_\Delta$ is a $(\delta,\delta)$--quasi-isometric embedding.
		\item \label{CHHS:comb_nesting_condition} Whenever $\Delta$ and $\Omega$ are non-maximal simplices of $Y$, there exists a (possibly empty) simplex $\Pi$ of $\lk(\Delta)$ such that $\lk(\Delta \star \Pi) \subseteq \lk(\Omega)$ and for all non-maximal simplices $\Lambda$ of $Y$ so that $\lk(\Lambda) \subseteq \lk(\Delta) \cap \lk(\Omega)$ either
		\begin{enumerate}
			\item $\mathrm{diam}(\mc{C}(\Lambda))<\delta$ or;
			\item $\lk(\Lambda) \subseteq \lk(\Delta \star \Pi)$.
		\end{enumerate}
		\item \label{CHHS:C=C_0_condition} For each non-maximal simplex $\Delta \subset Y$ and $x,y \in \lk(\Delta)$, if $x$ and $y$ are not joined by a $Y$--edge of $Y^{+W}$, but are joined by a $W$--edge of $Y^{+W}$, then there exits simplices $\Lambda_x, \Lambda_y \subseteq \lk(\Delta)$ so that  $x \in \Lambda_x$, $y \in \Lambda_y$, and $\Delta \star \Lambda_x$ is joined by an edge of $W$ to $\Delta \star \Lambda_y$.
	\end{enumerate}
\end{defn}

\begin{thm}[{\cite[Thoerem 1.8]{CHHS}}]\label{thm:cHHS_implies_HHS}
	Let $(Y,W)$ be a $\delta$--combinatorial HHS.  
	\begin{enumerate}
		\item The graph $W$ is  a connected  and a hierarchically hyperbolic space.
		\item Suppose $G$ is a finitely generated group that acts on $Y$ by simplicial automorphism. If there are finitely many $G$--orbits of links of simplices of $Y$, and the action of $G$ on $Y$ induces a metrically proper and cobounded action on $W$, then $G$ is a hierarchically hyperbolic group.
	\end{enumerate}
\end{thm}

\section{Statements of the main results}\label{sec:main_results}

We now state the main result of the paper, summarise where the various parts of the proof are found, and then deduce our application to 3-manifolds.

\begin{thm}\label{thm:squidable_HHS}
	Let $\mc{G}$ be an {\squidable} graph of groups. Let $\ST$ and $W = W_{r,R}$ be the spaces from Definitions \ref{defn:ST} and \ref{defn:W}.  For sufficiently large choices of $r\geq 0$ and $R \geq0$, the pair $(\mathcal S(T),W)$ is a $\delta$--combinatorial HHS with $\delta$ determined by $\mc{G}$.
	
	Moreover, $ \pi_1 \mc{G}$ is an HHG, because $\pi_1 \mc{G}$ acts on  $\ST$ with finitely many orbits of links of simplices, and the action on the set of maximal simplices of $\mathcal S(T)$ extends to a metrically proper and cobounded action on $W$. 
\end{thm}

\begin{proof}
	Item \eqref{CHHS:finite_complexity} is verified in Lemma \ref{lem:finite_complexity}. Item \eqref{CHHS:hyp_X} is verified in Proposition \ref{prop:Hyperbolicity_of_links} and Lemma \ref{lem:description_of_links}. Item \eqref{CHHS:geom_link_condition} is immediate when $\Delta=\emptyset$ or $\mathcal C(\Delta)$ is bounded, while the other cases are verified in Lemmas \ref{lem:tentancles_qi} and Lemma \ref{lem:quasilines_qi} (with Corollary \ref{cor:simplex_type} guaranteeing that all cases are covered). Item \eqref{CHHS:comb_nesting_condition} is Lemma \ref{lem:pseduo_wedges} and, finally, Item \eqref{CHHS:C=C_0_condition} is Lemma \ref{lem:C=C0_for_W}.
	
	The statement on orbits of links is verified in Lemma \ref{lem:ST_cocompact}, while the metrically proper and cobounded action is shown in Lemma \ref{lem:W_is_a_G_space}. The conclusion that $ \pi_1 \mc{G}$ is an HHG then follows from~Theorem \ref{thm:cHHS_implies_HHS}.
\end{proof}

Theorem \ref{thm:squidable_HHS} proves that non-geometric graph manifolds are HHG.

\begin{cor}\label{cor:graph_man}
	If $M$ is a non-geometric graph manifold, then $\pi_1 M$ is a hierarchically hyperbolic group, where the hyperbolic spaces in the HHG structure are all quasi-isometric to trees.
\end{cor}

\begin{proof}
	Since $\pi_1 M$ has the structure of an \squidable graph of groups, we can apply Theorem \ref{thm:squidable_HHS}. The hyperbolic spaces in the HHS structure coming from a combinatorial HHS are the $\mathcal C(\Delta)$ (as stated in \cite[Theorem 1.18]{CHHS}). These spaces are all quasi-isometric to trees in the case of $\pi_1 M$ by Proposition \ref{prop:Hyperbolicity_of_links} (and Lemma \ref{lem:description_of_links} for the bounded $\mathcal C(\Delta)$).
\end{proof}

We can now combine Corollary \ref{cor:graph_man} and  Corollary \ref{cor:central_extensions_by_Z} with results from the literature to classify when a 3-manifold group has an HHG structure in terms of the geometry of the prime pieces. We say that a flat 3-manifold is \emph{octahedral} if it is the quotient of $\mathbb R^3$ by a $3$--dimensional crystallographic group whose point group is conjugate in $GL_3(\mathbb R)$ into $O_3(\mathbb Z)$.

\begin{thm}\label{thm:HHGs_and_3_manifolds}
	Let $M$ be a closed oriented $3$--manifold.  $\pi_1 M$ is a hierarchically hyperbolic group if and only if $M$ has no  Nil, Sol, or non-octahedral flat manifolds in its prime decomposition. 
\end{thm}

\begin{proof}
	We first show that if $M$ has no Nil, Sol, or non-octahedral flat manifolds in its prime decomposition, then $\pi_1 M$ is an HHG. Since being hyperbolic relative to HHGs  will make $\pi_1 M$ an HHG \cite[Theorem 9.1]{BHSII}, it suffices to prove that $\pi_1 M$ is an HHG whenever $M$ is prime and not a Nil, Sol, or non-octahedral flat manifold.
	
	We first analyse the possible geometric cases from the geometrisation theorem.
	\begin{itemize}
		\item $S^3,S^2\times \mathbb R, \mathbb H^3$. In this case the fundamental group is hyperbolic, whence a hierarchically hyperbolic group.
		\item $\mathbb R^3$. The fact that the fundamental group of a manifold with geometry $\mathbb R^3$ is an HHG if and only if the manifold is octahedral follows from \cite[Theorem 4.4]{PetytSpriano}.
		\item $\mathbb H^2\times \mathbb R, \mathrm{PSL}_2(\mathbb R)$.  In these cases,  the fundamental group is a central  extension of $\mathbb{Z}$ by  a hyperbolic surface group, so we can apply Corollary  \ref{cor:central_extensions_by_Z} to conclude it is an HHG (the $\mathbb H^2 \times \mathbb R$ case was previously known, see, e.g. \cite[Proposition 3.1]{Hughes}). For later purposes, note that this case  also yields HHG fundamental groups when $M$ is a $\mathbb H^2 \times \mathbb R$ manifold with toroidal boundary.
	\end{itemize}
	
	In the non-geometric case,   $\pi_1 M$ is  hyperbolic relative to subgroups each of which is either $\mathbb Z^2$ or the fundamental group of a non-geometric graph manifold (this is a consequence of \cite[Theorem 0.1]{Dahmani:combination} and is stated explicitly as \cite[Theorem 9.12]{AFW}; see also \cite[Corollary E]{BigdelyWise}). Each peripheral is therefore an HHG, so the conclusion follows from \cite[Theorem 9.1]{BHSII}.
	
	We now assume $\pi_1 M$ has an HHG structure $\mf{S}$ and show $M$ cannot have a Nil, Sol, or non-octahedral flat manifold in its prime decomposition. If $M$ is prime and has either Nil or Sol geometry, then $\pi_1 M$ cannot be an HHG since it would not have quadratic Dehn function, contradicting~\cite[Corollary 7.5]{BHSII}. If $M$ is prime and is a non-octahedral flat manifold, then $\pi_1 M$ is not an HHG by \cite[Theorem 4.4]{PetytSpriano}. 
	
	For the non-prime case, let $M_1 \# \cdots \# M_n$ be the prime decomposition of $M$. Then $\pi_1 M$ is hyperbolic relative to $\pi_1 M_1, \dots, \pi_1 M_n$.  As the peripheral subgroups in a relative hyperbolic group, each $\pi_1M_i$ is strongly quasiconvex in $\pi_1 M$. Combining \cite[Proposition 5.7]{RST_convexity} and \cite[Proposition 5.6]{BHSII}, we have that restricting the projections in the HHG structure $\mf{S}$ to the subgroup $\pi_1M_i$ produces an HHS structure for $\pi_1M_i$ (but not necessarily an HHG structure). As before, this says $ M_i$ cannot have Nil or Sol geometry. 
	
	To rule out non-octahedral flat geometry, suppose $M_i$ is a flat manifold. Then, $\pi_1 M_i$ is virtually $\mathbb{Z}^3$ and Proposition \ref{lem:flat_torus} says there are $U_1,\dots, U_\ell \in \mf{S}$ so that
	\begin{itemize}
		\item $\{U_1,\dots,U_\ell\}$ is pairwise orthogonal and  $\pi_1 M_i$--invariant;
		\item $\diam(\varphi_V(\pi_1 M_i))$ is uniformly bounded for all $V \in \mf{S} -\{U_1,\dots,U_\ell\}$;
		\item for each $i \in \{1,\dots, \ell\}$, $\varphi_{U_i}(\pi_1 M_i)$ is a quasi-line in $\mc{C}(U_i)$;
		\item the \emph{hierarchically quasiconvex hull} of $\pi_1 M_i$ is quasi-isometric to $\mathbb{Z}^\ell$.
	\end{itemize}

	Since $\pi_1 M_i$ is strongly quasiconvex in $\pi_1 M$, it is undistorted  and the hierarchically quasiconvex hull of $\pi_1 M_i$ is uniformly close to $\pi_1 M_i$ in $\pi_1 M$. Hence $\pi_1M_i$ acts properly and cocompactly on its hierarchically quasiconvex hull. As $\pi_1M_i$ it  virtually $\mathbb{Z}^3$, this implies the number $\ell$ in the bulleted properties is $3$. Hence, we can make an HHG (and not just HHS) structure for $\pi_1 M_i$ by using the three quasilines $\varphi_{U_1}(\pi_1 M_i), \varphi_{U_2}(\pi_1 M_i), \varphi_{U_3}(\pi_1 M_i)$ and a finite number of bounded diameter spaces (this is the standard HHG structure on $\mathbb{Z}^3$ with the $\varphi_{U_3}(\pi_1 M_i)$ replacing the $x,y,z$ axes). By \cite[Theorem 4.4]{PetytSpriano}, this means $M_i$ must be an octahedral flat manifold.
\end{proof}

\begin{remark}[Concrete description of octahedral flat $3$--manifolds]\label{rem:octahedral}
	Combining~\cite{Hagen:crystallographic} and~\cite{Hoda:crystallographic}, a crystallographic group is octahedral (in any dimension) if and only if it is cocompactly cubulated if and only if it is \emph{Helly}.  In~\cite{PetytSpriano}, it is shown that for crystallographic groups, this is equivalent to being an HHG.  However, the octahedral flat $3$--manifolds can be explicitly listed, following Scott~\cite{Scott:geometry}.  Specifically, if $M$ is a compact orientable flat $3$--manifold, $M$ is octahedral if and only if is one of the following:
	\begin{itemize}
		\item the $3$--torus;
		\item made by gluing opposite faces of a cube with a $\frac12$ or $\frac14$ twist, on one pair;
		\item made by gluing opposite faces of a hexagonal prism with a $\frac13$--twist of the hexagonal faces;
		\item the \emph{Hantzsche–Wendt manifold}, which has point group $(\mathbb Z/2\mathbb Z)^2$.
	\end{itemize}
	The third one is tricky to visualise as octahedral, but here is an explanation in pictures instead of matrices.  Consider the tiling of $\mathbb E^3$ by hexagonal prisms; this is the universal cover of $M$ and so $\pi_1M$ acts freely with quotient the $3$--manifold described above.  In Figure~\ref{fig:prism}, we show one of these cells, $P$.  
	
	\begin{figure}[h]
		\centering
		\includegraphics[width=0.25\textwidth,height=4cm]{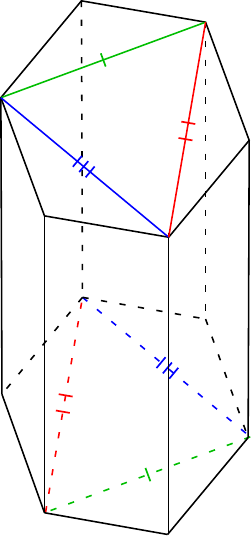}
		\caption{The $\frac13$--twist prism manifold is octahedral. The set of planes through lines of the same colors are preserved by the $\frac{1}{3}$--twist.}
		\label{fig:prism}
	\end{figure}
	
	Consider the six coloured segments in the figure, three in each of the two hexagonal faces of $P$.  As indicated by the colours/labels, these come in three pairs of parallel segments, with each hexagonal face contributing one of the segments in each pair.  Each parallel pair lies in a uniquely determined plane in $\mathbb E^3$.  This set of three planes is invariant under an order $3$ rotation of $P$ about the central vertical line.  Hence the $\pi_1M$--orbit of this family of 3 planes is a set of planes in $\mathbb E^3$ falling into three parallelism classes.  Cubulating the resulting wallspace (see e.g.~\cite{ChatterjiNiblo}) therefore gives a proper cocompact action of $\pi_1M$ on the standard tiling of $\mathbb R^3$ by $3$--cubes, whence $\pi_1M$ is octahedral by~\cite{Hagen:crystallographic} or~\cite{Hoda:crystallographic}.  One can also  directly compute a basis invariant under the point group.	
	
	According to \cite{Scott:geometry}, there is only one more compact oriented flat $3$--manifold.  This is also constructed from a hexagonal prism by identifying opposite faces, but the hexagons are identified using a $\frac16$ twist.  (So, one can still cubulate $\pi_1M$ as above, but this gives an action on $\mathbb R^6$, which is not cocompact.)  This manifold is not octahedral since $O_3(\mathbb Z)$ does not have an orientation-preserving element of order 6.
\end{remark}

\section{Quasi-lines from quasimorphisms}\label{sec:quasilines}
We now use quasimorphisms to construct the actions on quasi-lines. This is both an essential ingredient in our construction of a combinatorial HHS for an admissible graph of group and the key to proving that central extensions of $\mathbb{Z}$ by hyperbolic groups are HHGs.

We first build quasimorphisms for central extensions where the center is unbounded.

\begin{lem}\label{lem:cohomology} 
	Suppose that the central extension  of groups $\Z \xrightarrow{\iota} G \xrightarrow{\pi} F$ corresponds to a bounded element of $H^2(F,\Z)$. Then there exists a quasimorphism $\phi \colon G\to \Z$ which is unbounded on $\iota(\Z)$.
\end{lem}

\begin{proof}
	The fact that the cohomology class associated to the central extension is bounded implies that there exists a (set-theoretic) section $s\colon F\to G$ so that there are only finitely many possible values of $s(f_1)s(f_2)s(f_1f_2)^{-1}$ as $f_1,f_2$ vary in $F$. Hence, if we define $c \in H^2(F,\Z)$ by $c(f_1,f_2) = \iota^{-1}\left(s(f_1)s(f_2)s(f_1f_2)^{-1}\right)$, then the absolute value of $c(f_1,f_2)$ is bounded independently of $f_1,f_2$.
	
	We now define $\phi$. Any $x\in G$ can be written in a unique way as $s(f_x)\iota(t_x)$ for $f_x \in F$ and $t_x \in \Z$.  Hence we can set $\phi(x)=t_x$. To show that $\phi$ is a quasimorphism note that, since $\iota(\Z)$ is central and $s(f_1)s(f_2) = \iota(c(f_1,f_2)) s(f_1f_2)$, we have
	$$xy=s(f_x)\iota(t_x)s(f_y)\iota(t_y)=s(f_x)s(f_y)\iota(t_x+t_y)=s(f_xf_y)\iota(c(f_x,f_y)+t_x+t_y).$$
	Hence, $\phi(xy)=\phi(x)+\phi(y)+c(f_x,f_y)$, and we are done since the absolute value of $c(f_x,f_y)$ is uniformly bounded.
\end{proof}

We now use quasimorphisms to show that the vertex groups of an \squidable graph of groups have the desired action of a quasi-line. 

\begin{lem}\label{lem:Generating_sets_that_makes_lines}
	Let $\mc{G}= (\Gamma, \{G_\mu\}, \{G_\alpha\}, \{\tau_{\alpha}\})$ be an \squidable   graph of groups. For each edge $\alpha$ of $\mc{G}$, denote $C_{\alpha}=\tau_\alpha ((\tau_{\bar{\alpha}})^{-1}(Z_{\alpha^-}))<G_{\alpha^+}$.
	Each vertex group $G_\mu$ has an infinite generating set $S_\mu$ so that the following hold. 
	\begin{enumerate}
		\item $Cay(G_\mu,S_\mu)$ is quasi-isometric to a line,\label{item:quasi-line}
		\item the inclusion $Z_\mu\hookrightarrow \mathrm{Cay}(G_\mu,S_\mu)$ is a $Z_\mu$--equivariant quasi-isometry, \label{item:unbounded_Z}
		\item for each edge $\alpha$ with $\alpha^+=\mu$, $C_{\alpha}$ is bounded in $\mathrm{Cay}(G_\mu,S_\mu)$ (in fact, the bound is uniform over all $\alpha,\mu$ since there are finitely many).\label{item:bounded_C}
	\end{enumerate}
\end{lem}
\begin{proof}
	By \cite[Lemma 4.15]{ABO}, if one can find an unbounded homogeneous quasimorphism $\bar{\phi}:G_\mu\to \R$, then there exists a generating set $S_\mu$ such that $\mathrm{Cay}(G_\mu, S_\mu)$ is quasi-isometric to a line and an element $g\in G$ acts loxodromically if and only $\bar{\phi}(g) \neq 0$. In particular, items (\ref{item:quasi-line}), (\ref{item:unbounded_Z}) and (\ref{item:bounded_C}) are equivalent to the existence of a quasimorphism $\phi:G_\mu\to \R$ so that $\phi(Z_\mu)$ is unbounded, but $\phi(C_{\alpha})$ is uniformly bounded for all edges $\alpha$ with $\alpha^+=\mu$ (then $\bar{\phi}$ is the homogenization of $\phi$). Note that each $C_\alpha$ does not intersect $Z_\mu$ in $G_\mu$. Hence, the quotient map  $\pi_\mu \colon G_\mu \to F_\mu$ is injective on $C_{\alpha}$ and we have that $\pi_\mu(C_{\alpha})< \pi_\mu(\tau_\alpha(G_\alpha))$ is infinite.
	
	We now construct certain auxiliary quasimorphisms. The first one, $\phi_\mu$, is just the homogenization of the quasimorphism from Lemma \ref{lem:cohomology}, which we can apply by condition Definition~\ref{defn:admissible}.\eqref{item:center_and_quotient} and the fact that every cohomology class of a hyperbolic group is bounded~\cite{Mineyev}. The other ones are constructed as follows. We claim that for each edge $\alpha$ of $\mc{G}$ with $\alpha^+=\mu$, there is a homogeneous quasimorphism $\psi_\alpha \colon F_\mu\to \R$ so that $\psi_\alpha(\pi_\mu(c_\alpha))=1$, where $c_\alpha$ is a fixed generator of $C_{\alpha}$, and $\psi_\alpha(c_{\alpha'})=0$ for all other edges $\alpha'$ with $\alpha'^+=\mu$.
	
	By Lemma~\ref{lem:F_mu_are_relatively_hyp}, $F_\mu$ is hyperbolic relative to the subgroups $\pi_\mu(\tau_\alpha(G_\alpha))$ for edge of $\mc{G}$ with $\alpha^+ = \mu$. In particular the subgroups $\pi_\mu(\tau_\alpha(G_\alpha))$ are (jointly) hyperbolically embedded in $F_\mu$. We can then appeal to \cite[Theorem 4.2]{HullOsin} to find the required quasimorphism. (The construction of Epstein--Fujiwara \cite{EpsteinFujiwara} should also be applicable to construct such quasimorphisms).
	
	Let $\phi_\alpha=\psi_\alpha\circ \pi_\mu$ and  observe that
	$$\phi:= \phi_\mu -\sum_{\alpha^+=\mu} \phi_\mu(c_\alpha)\phi_\alpha$$
	satisfies all the required properties. Thus, $\phi$ is the desired quasimorphism.
\end{proof}

To prove the first two bullet points of  Lemma \ref{lem:Generating_sets_that_makes_lines}, we do not need the full definition of an \squidable graph of groups. That is, if we have a central extension of groups $\Z \hookrightarrow G \stackrel{\pi}{\twoheadrightarrow} F $ corresponds to a bounded element of $H^2(F,\Z)$, then \cite[Lemma 4.15]{ABO} says the quasi-morphism from Lemma \ref{lem:cohomology} produces  a generating set $S$ for $G$ so that $\cay(G,S)$ is a quasi-line where the inclusion of the central $\Z$ is a $\Z$--equivariant quasi-isometry. This construction allows us to prove all such central extension are HHG.

\begin{cor}\label{cor:central_extensions_by_Z}
	If a group $G$ is a central extension 
	$Z \hookrightarrow G \stackrel{\pi}{\twoheadrightarrow} F $
	where $Z$ is an infinite cyclic group and $F$ is a hyperbolic group, then $G$ is a hierarchically hyperbolic group.
\end{cor}

\begin{proof}
	Let $z$ be the generator for $Z$ and $J$ be a finite symmetric generating set for $G$ that contains $z$. We will identify $F$ with the quotient $G / Z$ and write elements of $F$ as coset of $Z$.  As described in the paragraph before Corollary \ref{cor:central_extensions_by_Z}, there is a generating set $S$ for $G$ so that $\cay(G,S)$ is a quasi-line and the inclusion of $Z$ into $\cay(G,S)$ is a $Z$--equivariant quasi-isometry. Let $L = \cay(G,S)$ and $H = \cay(F, \pi(J))$. We will prove  that the diagonal action of $G$ on $L \times H$ is metrically proper and  cobounded (where we fix, say, the $\ell_1$--metric on said product). This will imply that $G$ is an HHG as any group acting metrically properly and coboundedly on a product of hyperbolic spaces preserving the factors is an HHG; see \cite[Section 8.3]{BHSII} or  \cite[Proposition 3.1]{Hughes}.

	To prove coboundedness, let $r$ be large enough that every point in $L$ is within $r$ of an element of $Z$. Hence, for any vertex $(k,hZ)$ of $L \times H$, there is a power $z^n$ of $z$ so that $d_L(k,z^n h) \leq 2r$.  Thus $z^nh \cdot (1,Z) = (z^n h, hZ)$ is within $2r$ of $(k,hZ)$ and hence the action of $G$ of $L\times H$ is cobounded. 
	
	Moving on to metric properness, let $B^L(r)$ and $B^H(r)$ be the balls of radius $r \geq 0$ around the identity element in $L$ and $H$ respectively. Since $G$ acts coboundedly on $L \times H$, every bounded diameter set of $L \times H$ is contained in some $G$--translate of $B^L(r) \times B^H(r)$ for some $r$. Hence it suffices to prove that the set of $g \in G$ such that $g(B^L(r) \times B^H(r)) \cap B^L(r) \times B^H(r) \neq \emptyset$ is finite.
	
	If  $g(B_L(r) \times B^H(r)) \cap (B^L(r) \times B^H(r))\neq \emptyset$, then $gB^\ast(r) \cap B^\ast(r) \neq \emptyset$ for $\ast = L$ or $H$. The set $\{g\in G : gB^H(r) \cap B^H(r) \neq \emptyset\}$ is contained in the set $\{g \in G : d_H(gZ,Z) \leq 2r\}$. However, because $F$ is finitely generated, the later is the union of finitely many cosets of $Z$. Now, since orbit maps of the action of $Z$ on $L$ are quasi-isometries, each coset of $Z$ can only contain finitely many element $g$ for which $g B^L(r) \cap B^L \neq \emptyset$. Together, these say that the set $$\bigl\{ g \in G : g\bigl(B^L(r) \times B^H(r)\bigr) \cap \bigl(B^L(r) \times B^H(r)\bigr) \neq \emptyset\bigr\}$$ is finite.
\end{proof}

We now translate the content of Lemma \ref{lem:Generating_sets_that_makes_lines} into the language and notation of Bass--Serre space. Firstly, let us introduce the analogues of $\mathrm{Cay}(G_\mu, S_\mu)$.
\begin{defn}[Space $L_v$]\label{defn:L_v}
	Let $\mc{G}= (\Gamma, \{G_\mu\}, \{G_\alpha\}, \{\tau_{\alpha^{\pm}}\})$ be an \squidable  graph of groups. Let $\BS$ be the Bass--Serre space associated to $\mc{G}$ and $T$ be the Bass--Serre Tree of $\mc{G}$. For each vertex $\mu$ of $\mathcal{G}$, let $S_{\mu}$ be a generating set of $G_{\mu}$ as in Lemma \ref{lem:Generating_sets_that_makes_lines}.
	Without loss of generality we can assume $J_{\mu} \subseteq S_{\mu}$, where $J_{\mu}$ is the fixed finite generating set of $G_{\mu}$. 
	For a vertex $v \in T$ with $\mu= \check{v}$, let $gG_\mu$ be the corresponding coset of $G_\mu$. Define $L_v$ to be the graph with vertex set $gG_\mu$ and edges connecting $x, y \in gG_{\mu}$ if $x^{-1}y \in S_\mu$. Since $L_v$ is obtained from $X_v$ by adding extra edges to the same vertex set, there is a distance-non-increasing map $p_v \colon X_v \to L_v$ that is the identity on the vertices.\end{defn}

\begin{prop}\label{prop:geometry_of_Zs}
	Let $\mc{G}$ be an \squidable graph of groups with Bass--Serre tree $T$ and Bass--Serre space $X$. Let $e$ be an edge of $T$, with $v=e^+$ and $w = e^-$. Let $g,h \in G$ be such that $g Z_{\check w} \subseteq \tau_{\bar e}(X_w) \subset X_w$ and $h Z_{\check v}\subseteq X_v$. There exists $\xi \geq 1$, depending only on $\mc{G}$, so that:
	
	\begin{enumerate}
		\item \label{item:bounded_Zs}  $\diam(p_v\circ \tau_e \circ \tau_{\bar e}^{-1}(gZ_{\check w})) \leq \xi$.
		\item \label{item:Zs_QI_embed_in_Xv}  The restriction of $p_v$ to $hZ_{\check{v}}$ (seen with the induced metric of $X_v$) is a $(\xi,\xi)$--quasi-isometry.  In particular, the cosets $hZ_{\check{v}}$ are undistorted in $X_v$.     
		\item \label{item:distance_to_Zs} Let $x\in X_v$. Then
		$$d_{X_v}(x,\tau_e \circ \tau_{\bar e}^{-1}(gZ_{\check w})) \leq \xi d_{L_v}(p_v(x),p_v\circ \tau_e \circ \tau_{\bar e}^{-1}(gZ_{\check w})) + \xi$$ 
	\end{enumerate}

\end{prop}

\begin{proof}
	Item~\eqref{item:bounded_Zs} is a verbatim translation of Lemma~\ref{lem:Generating_sets_that_makes_lines}.\eqref{item:bounded_C} in the setting of Bass--Serre spaces.  The bound is independent of $e,g,h$ since there are finitely many orbits of vertices and edges.
	
	For the proof of \eqref{item:Zs_QI_embed_in_Xv},  fix a representative $h$ of $hZ_{\check{v}}$. This determines a map $Z_{\check{v}} \to hZ_{\check{v}}$ defined as $z \mapsto hz$. Note that this maps is not canonical, as it depends on the choice of $h$, but this will not be a problem.  
	We consider three different metrics on the set $Z_{\check{v}}$: the intrinsic word metric $d_Z$ on $\mathrm{Cay}(Z_{\check{v}})$, the restriction of $(X_v, d_{X_v})$ using the inclusion $Z_{\check{v}} \to hZ_{\check{v}}\subseteq X_v$ and the restriction of $(L_v, d_{L_v})$ using the map $p_v$. In particular, by choosing an appropriate generating set on $\mathrm{Cay}(Z_{\check{v}})$, the maps $\mathrm{Cay}(Z_{\check{v}})\to (hZ_{\check{v}}, d_v) \to (hZ_{\check{v}}, d_L)$ are all distance non-increasing. By  Lemma~\ref{lem:Generating_sets_that_makes_lines}.\eqref{item:unbounded_Z}, the composition is a quasi-isometry, yielding that the $p_v\vert_{hZ_{\check{v}}} \colon (hZ_{\check{v}}, d_v) \to L_v$ is a quasi-isometry.
	
	For the proof of \eqref{item:distance_to_Zs}, we will denote $\tau_e \circ \tau_{\bar e}^{-1}(gZ_{\check{w}})$ by $C_e$. Let $x \in \tau_e(X_e)$ and $x' \in C_e$ so that $d_{X_v}(x,C_e) = d_{X_v}(x,x')$. Now, there exist $g \in G$ so that $x$ is contained in the coset $gZ_{\check{v}}$ in $X_v$. Because we have proved Item \eqref{item:Zs_QI_embed_in_Xv}, there is $\kappa\geq 1$, depending only on $\mc{G}$ so that the restriction of $p_v$ to $gZ_{\check{v}}$ is a $(\kappa,\kappa)$--quasi-isometry. In particular, there must be $\bar x \in C_e$ so that $d_{L_v}(p_v(\bar x),p_v(gZ_{\check{v}})) \leq \kappa$. Moreover, we can choose $\kappa$ so that $\diam(p_v(C_e))\leq \kappa$ as well. Using that $p_v \colon X_v \to L_v$ is distance non-increasing, we now have \[d_{L_v}(x,x') \leq d_{X_v}(x,x') \leq d_{X_v}(x,\bar x) \leq \kappa d_{L_v}(p_v(x),p_v(\bar x)) +\kappa \leq \kappa d_{L_v}(p_v(x),p_v(x')) + \kappa^2 +\kappa, \] which implies \[d_{L_v}(p_v(x),p_v(C_e)) \leq d_{X_v}(x,C_e) \leq \kappa d_{L_v}(p_v(x),p_v(C_e)) +\kappa^2 + \kappa.\] The result follows by taking $\xi = \kappa^2 + \kappa$.
\end{proof}

We remark that the statements of Proposition \ref{prop:geometry_of_Zs} are concerned only with the metrics of the vertex  spaces $X_v$  and not on the metric on all of $\BS$, where $X_v$ and $X_e$ maybe distorted. In the sequel, we will often use Proposition \ref{prop:geometry_of_Zs} to establish a uniform bound on distances in $X_v$ or $X_e$ and then use Lemma \ref{lem:uniform_distortion} to translate this into a uniform bound on distances in $\BS$.

\section{Defining a combinatorial HHS: a blow-up of the Bass--Serre tree}\label{sec:blow_up}
In this section, we describe how to construct the simplicial complex and graph that make a combinatorial HHS for an \squidable graph of group. We prove that this construction satisfies the requirements of Theorem \ref{thm:cHHS_implies_HHS} in Section \ref{sec:proof}. 

For the remainder of this section, let  $\mc{G} = (\Gamma, \{G_\mu\}, \{G_\alpha\}, \{\tau_\alpha\})$ be an \squidable  graph of groups (Definition \ref{defn:admissible}) and fix $G =  \pi_1 \mc{G}$. As in Section \ref{subsec:graphs_of_groups}, we fix generating sets $J_\mu$ and $J_\alpha$ for the vertex and edge groups of $\mc{G}$. Let $T$ denote the Bass--Serre tree of $\mc{G}$ and $\BS$ the Bass--Serre space from Definition \ref{defn:BS_Space}. For vertices $v$ and edges $e$ of $T$, $X_v$ and $X_e$ will denote the vertex and edge spaces of $\BS$ respectively. Recall that $T^{(0)}$ is the set $ \{gG_{\mu} : g \in G, \mu \in 
V(\mc{G})\}$ and that for each $v \in T^{(0)}$, the elements of $X^{(0)}_v$ are precisely the elements of the coset $g G_{\mu} = v$.  For an edge $e$ of $T$, the maps $\tau_e$ and $\tau_{\bar{e}}$ denote the maps from the edge space $X_e$ into the vertex spaces $X_{e^+}$ and $X_{e^-} = X_{\bar{e}^+}$ described in Definition \ref{defn:BS_Space}.

We also fix the generating sets $S_\mu$ from Lemma \ref{lem:Generating_sets_that_makes_lines} for the vertex groups $G_\mu$ that produce Cayley graphs that are quasi-lines.  Accordingly, for each vertex $v \in T^{(0)}$ we have the quasi-line $L_v$ from Definition \ref{defn:L_v}, which is the Cayley graph of the coset $gG_\mu = v$ with respect to the generating set $S_\mu$. As described in Definition \ref{defn:L_v}, there is a 1--Lipschitz  map $p_v \colon X_v \to L_v$. 

The simplicial complex for our combinatorial HHS will be the following complex $\ST$ that is a ``blow-up'' of the Bass--Serre tree $T$ to include the vertices of each vertex space $X_v$ at each vertex $v \in T$.

\begin{defn}\label{defn:ST}
	Let $Q =\bigsqcup_{v \in T^{(0)}} X_v^{(0)}$. Define the function  $\nu 
	\colon T^{(0)} \sqcup Q \to T^{(0)}$ as the identity on $T^{(0)}$ and as $\nu(s) = v$ if $s 
	\in X_v$.
	
	Let $\ST$ be the flag simplicial complex with vertex set $T^{(0)} \sqcup Q$ and the following two types of edges. First, each $s\in Q$ is connected to $\nu(s)$. Second, two vertices $s,t \in \ST$ are connected if $\nu(s), \nu(t)$ are adjacent in $T$.
	Observe that $\nu$ extends to a simplicial map $\ST \to T$ that we still denote $\nu$. 
\end{defn}

Having constructed our simplicial complex, we now need to define a graph $W$ whose vertices are the maximal simplices of $\ST$. We start be describing the maximal simplices of $\ST$.

\begin{lem}[Maximal simplicies in $\ST$]\label{lem:maximal_simplices}
	The maximal simplices of $\ST$ are exactly the simplices of the form $\{s, \nu(s), t, \nu(t)\}$, where $s,t\in \ST^{(0)} -T^{(0)}$ and $\nu(s), \nu(t)$ are adjacent in $T$. We denote such a simplex by $\Sigma(s,t)$.  
\end{lem}

\begin{proof}
	Consider a simplex $\Sigma=\{s,\nu(s),\nu(t),t\}$ of $\ST$, where $\nu(s),\nu(t)$ are adjacent in $T$.  Suppose 
	that $\Sigma$ is non-maximal.   There then exists a vertex $u$ of $\ST$ that is adjacent to each of 
	$s,\nu(s),t,\nu(t)$.  Since $\nu$ is simplicial, this means that $\nu(u)$ is equal to, or adjacent to, each of 
	$\nu(s),\nu(t)$.  Since $T$ is triangle-free and $\nu(s),\nu(t)$ are adjacent, $\nu(u)$ cannot be adjacent to 
	both $\nu(s)$ and $\nu(t)$, so, without loss of generality, $\nu(u)=\nu(s)$.  Since $u$ is different from $s$ 
	and $\nu(s)$, we therefore have that $\nu^{-1}(\nu(s))$ contains a $3$--cycle with vertices $s,u,\nu(s)$.  
	This contradicts the definition of $\ST$.  Thus $\Sigma$ is a maximal simplex.
	
	Conversely, let $\Sigma$ be a maximal simplex of $\ST$.  Since $\nu:\ST\to T$ is simplicial, $\nu(\Sigma)$ is 
	either a vertex of $T$ or an edge of $T$.  If $\nu(\Sigma)$ is a vertex, then  $T$ has some vertex $v$ adjacent to $\nu(\Sigma)$  as $T$ is a connected graph with at least two vertices.  But then $v\star\Sigma$ is a 
	simplex of $\ST$ properly containing $\Sigma$.  Hence $\nu(\Sigma)$ must be an edge 
	joining two vertices, $\nu(s)$ and $\nu(t)$, of $T$.  So, $\Sigma$ has the form $\Delta_s\star\Delta_t$, where 
	$\Delta_s$ is a simplex projecting to $\nu(s)$ and $\Delta_t$ projects to $\nu(t)$.  Maximality of $\Sigma$ 
	implies that $\Delta_s,\Delta_t$ are edges, as required.
\end{proof}

Our goal is to define the edges in $W$ so that $G$ has a metrically proper and cobounded action on $W$, and so that we can verify  the conditions of a combinatorial HHS. To accomplish  the former, we want to associate to each maximal simplex $\Sigma(s,t)$ a uniformly bounded diameter subset of $\BS$ and then declare two maximal simplices to be joined by any edge if their corresponding bounded diameter subsets are close in $\BS$. To facilitate this, we use the following ``coarse level sets'' of the map $p_v \colon X_v \to L_v$ from Definition \ref{defn:L_v}.

\begin{defn}[``Level surfaces'']\label{defn:level_surfaces}
	Let $v \in T^{(0)}$ and $s \in \BS^{(0)}_{v}$. For $r\ge0$, define $\sigma_r(s)$ to be the set $$\sigma_r(s):=\{x \in \BS_{v} : d_{L_{v}}(p_v(s),p_v(x)) \leq r\}.$$
\end{defn}

While we will not use this fact directly, it is helpful to think of the vertex spaces $X_v$ as having a product structure in which the subspaces parallel to one factor are the $\sigma_r(s)$ and the subspaces parallel to the other factors are cosets of the $gZ_\mu$. Thus, if one compares to the motivating case of a graph manifold, we can think of the $\sigma_r(s)$ as the ``level surfaces'' of the vertex spaces and the $gZ_\mu$ (which are quasi-isometric to the $L_v$)  can be thought of as ``lines''.

The intersection of these ``level surfaces'' gives us a bounded diameter subset of $X$ associated to a maximal simplex.

\begin{defn}[Coarse points of maximal simplices]\label{defn:P_sets}
	Let $\mc{N}_c(\cdot)$ denote the $c$--neighborhood of a set in $X$. Given $s,t \in \ST^{(0)} - T^{(0)}$ so that $\nu(s)$ and $\nu(t)$ are joined by an edge $e$ of $T$ with $e^+ = \nu(s)$, define $$ P_r(s,t) := \mc{N}_1(\sigma_{r}(s)) \cap \mc{N}_1( \sigma_r(t)).$$ Since $\sigma_r(s)$ and $\sigma_r(t)$ are in different vertex spaces ($X_v$ vs $X_w$), $P_r(s,t)$ is precisely the set of vertices $x\in X_e$ so that $\tau_e(x) \in \sigma_t(s)$ and $\tau_{\bar e}(x) \in \sigma_r(t)$.
\end{defn}

\begin{lem}\label{lem:R_for_sigmas_to_intersect}
	There exists $r_0>0$ such that for all $r\geq r_0$ there exists $\xi \geq 0$ so that the following holds. Let $s, t \in \ST^{(0)} - T^{(0)}$ be such that $\nu(s), \nu(t)$ are joined by an edge of $T$. Then
	\begin{itemize}
		\item  $P_{r}(s,t) $  is non-empty and has diameter at most $\xi$;
		\item the map $(s,t) \to P_r(s,t)$ is  a $(\xi,\xi)$--coarsely Lipschitz, $\xi$--coarse map from $L_{\nu(s)} \times L_{\nu(t)}$ to $\BS$.
	\end{itemize}
\end{lem}

\begin{proof}
	Let $v = \nu(s)$ and $w=\nu(t)$, then let $e$ be the edge of $T$ from  $e^-=w$ to $e^+=v$.  
	The key tool for the proof is the following quasi-isometry from $X_e$ to $L_v \times L_w$.
	
	\begin{claim}\label{claim:Diagonal_map_is_a_QI}
		For each edge $e$ of $T$, the diagonal map $\Phi_e \colon X_e \to L_{e^+} \times L_{e^-}$ given by $$\Phi_e(x) = (p_{e^+} 
		\circ \tau_{e}(x), p_{e^{-}} \circ \tau_{\bar{e}}(x))$$ is a uniform quasi-isometry with $\Phi_e(g\cdot x) = g \cdot \Phi_e(x)$  for each $g \in \stab_G(e)$ and $x \in \BS_e$.
	\end{claim}
	
	\begin{proof} Let $e^+ =v, e^-=w$, then let $\mu = \check{v}$, $\omega = \check{w}$, and $\alpha = \check{e}$. Equip $G_\alpha$ with the metric coming from $\cay(G_\alpha,J_\alpha)$ and $gt_\alpha Z_\mu$, $gZ_\omega$ with the metrics as subsets of $X_v$. By Proposition \ref{prop:geometry_of_Zs}.\eqref{item:Zs_QI_embed_in_Xv}, this metric is quasi-isometric to any intrinsic metric on the cosets coming from a finite generating set of $Z_\mu$ and $Z_\omega$.
		
		Let $g \in G$ so that $w$ is the coset $g G_\omega$ and $v$ is the coset $gt_\alpha G_\mu$. If we let $z_\mu$ and $z_\omega$ be arbitrary elements of $Z_\mu$ and $Z_\omega$ respectively, define $\phi \colon gt_\alpha Z_\mu \times gZ_\omega \to G_\alpha$ by  $$(gt_\alpha z_\mu, g z_\omega) \mapsto \tau_{\alpha}^{-1}(z_\mu)\tau_{\bar \alpha}^{-1} (z_\omega).$$ Definition \ref{defn:admissible} says $G_\alpha \cong \mathbb{Z}^2$ and that $\langle \tau_{\alpha}^{-1}(Z_\mu), \tau_{\bar \alpha}^{-1}(Z_\omega) \rangle$ is a finite index subgroup of $G_\alpha$. Hence, $\phi$ is a quasi-isometry. 
		
		Now define a map $\theta \colon G_\alpha \to X_e$ by $\theta (a) = \tau_{\bar e}^{-1}(g \tau_{\bar \alpha}(a))$. The construction of the Bass--Serre space tells us $\theta$ gives an isometry  $\theta \colon \cay(G_\alpha,J_\alpha) \to X_e$. 
		Moreover, $\theta(a)$ also equals $\tau_{e}^{-1}( gt_\alpha \tau_{\alpha}(a)$ for each $a \in G_\alpha$. Thus, we have $$p_v \circ \tau_e( \theta(\phi(gt_\alpha z_\mu, g z_\omega))) = p_v( gt_\alpha z_\mu) \text{ and } p_w \circ \tau_{\bar e}( \theta(\phi(gt_\alpha z_\mu, g z_\omega))) = p_w( g z_\omega)$$ for each $z_\mu\in Z_\mu$ and $z_\omega \in Z_\omega$. 
		As  $p_v\vert_{gt_\alpha Z_\mu} $ and $p_w\vert_{g Z_\omega}$ are uniform quasi-isometries by Proposition \ref{prop:geometry_of_Zs}.\eqref{item:Zs_QI_embed_in_Xv}, $p_v \times p_w \colon gt_\alpha Z \times gZ_\omega \to L_v \times L_w$ is a quasi-isometry. Hence, $$(p_v \times p_w) \circ (\tau_e \times \tau_{\bar e}) = (p_v \times p_w) \circ \phi^{-1} \circ \theta^{-1} $$ is a quasi-isometry (here $\phi^{-1}$ is any quasi-inverse of $\phi$ that inverts $\phi$ on its image).
		
		The constants of all these quasi-isometries can be chosen independent of $\alpha$ because $\mc{G}$ has finitely many edges.\end{proof}

	Let $\Phi_e \colon \BS_e \to L_v \times L_w$ be the quasi-isometry
	$$x \to (p_{v}(\tau_{\bar{e}}(x)), p_{w}(\tau_{e}(x)))$$ from Claim \ref{claim:Diagonal_map_is_a_QI}. Because $\Phi_e$ is coarsely onto, there exists $x \in \BS_e$ and $r_0 >0$ so that $p_v \circ \tau_{e}(x)$ is within $R_0$ of $p_v(s)$ in $L_v$ and $p_w \circ \tau_{\bar e}(x)$ is within $r_0$ of $p_w(t)$ in $L_w$. Thus, $x \in\mc{N}_1(\sigma_{r}(s)) \cap \mc{N}_1( \sigma_r(t)) = P_r(s,t)$ is non-empty for each $r \geq r_0$.
	
	When $P_r(s,t)$ is non-empty, then $\Phi_e(P_r(s,t)) \subseteq p_v(\sigma_r(s)) \times p_w(\sigma_r(t))$, which is a bounded diameter subset of $L_v \times L_w$. Since $\Phi_e$ is a quasi-isometry and the inclusion of $\BS_e$ into $\BS$ is 1-Lipschitz, $P_r(s,t)$ is then a bounded diameter subset of both $\BS_e$ and $\BS$.
	
	Finally,  the map $\Psi_e \colon  L_v\times L_w \to \BS_e$ given by $(s,t) \to P_r(s,t)$  is a quasi-inverse of $\Phi_e$. Since the inclusion of $\BS_e$ into $\BS$ is 1-Lipschitz, the extension $\Psi_e$ into $\BS$ will be coarsely Lipschitz (and in fact uniformly so since there are only finitely many orbits of edges).	
\end{proof}

We can now define the edges in our graph $W$, relying on the ``level surfaces'' $\sigma_r(s)$ from Definition \ref{defn:level_surfaces} and coarse points $P_r(s,t)$ that arise as their intersections as in Definition \ref{defn:P_sets}.

\begin{defn}\label{defn:W} 
	
	($W$--edges.)
	For $r, R>0$, let $W=W_{r, R}$ be the graph  defined as follows. The vertices of $W$ are the maximal simplices of $\ST$. Two simplices $ \Sigma(s,t)$ and $\Sigma(s', t')$ are adjacent if and only if one of the following holds.
	\begin{itemize}
		\item $\nu(s) = \nu(s')$ and  $d_{\BS}(P_{r}(s,t),P_{r}(s',t')) \leq R$. 
		\item $s = s'$ and $d_{\BS}(\sigma_{r}(t), \sigma_{r}(t') )\leq R +2$.
	\end{itemize}
\end{defn}

\begin{remark}
	In both cases, being joined by an edge of $W$ implies that the two maximal simplices share a common vertex of $T$.
\end{remark}

The first type of edge of $W$ is needed to assure that  $W$ is quasi-isometric to the Bass--Serre space $\BS$.  The second type is needed to arrange a fine combinatorial constraint in the definition of combinatorial HHS. To prove that $G$ acts metrically properly on $W$ we start by showing the second type of edges gives a similar bound to the first type.

\begin{prop}
	\label{prop:P_vs_sigma}
	There exists $r_1 \geq0$ so that for each $r \geq r_1$, there exists a monotone diverging function $\Phi \colon [0,\infty) \to [0,\infty)$ so that the following holds. Consider two vertices $v_1,v_2$ of the Bass--Serre tree $T$ at distance $2$ from each other, with $w$ being the vertex at distance $1$ from both. Suppose that $t_1,t_2,s\in \ST^{(0)} - T^{(0)}$ are such that $\nu(t_i)=v_i$ and $\nu(s)=w$. Then
	$$d_{\BS}(P_{r}(s,t_1),P_{r}(s,t_2))\leq \Phi(d_{\BS}(\sigma_{r}(t_1),\sigma_{r}(t_2))).$$
\end{prop}
\begin{proof}
	
	The bulk of the technical work of the proposition is contained in the following claim.
	
	\begin{claim}\label{claim:sigma_L_product}
		There exist $r_1 \geq 0$ so that for every $r \geq r_1$ there is a constant $c \ge 0$ and a monotone diverging function $\Phi' \colon [0,\infty) \to [0,\infty)$ so that the following holds. Let $s,t \in \ST^{(0)}- T^{(0)}$ so that $v = \nu(t)$ and $w =\nu(s)$ are joined by an edge $e$ of $T$ with $e^+=v$ and $e^- = w$. Let $\mu = \check{v}$ and $\omega = \check{w}$ .
		\begin{enumerate}
			\item \label{item:Z_intersects_sigma} If $g \in G$ is so that $gZ_\mu \subseteq X_v$, then  $\sigma_r(t) \cap gZ_\mu \neq \emptyset$. Equivalently, if $gZ_\omega \subseteq X_w$, then $\sigma_r(s) \cap gZ_\omega \neq \emptyset$.
			\item \label{item:vertical} There exist $g\in G$ (depending on $t$), so that $gZ_\omega \subseteq \tau_{\bar e}(X_e)$ and 
			$$ \tau_{e}(\tau_{\bar e}^{-1}( g Z_{\omega}))\subseteq \sigma_{r}(t )\cap \tau_e(X_e) \subseteq \mc{N}_c(\tau_{e}(\tau_{\bar e}^{-1}( g Z_{\omega})).$$ 
			\item \label{item:close_to_edge} For each $x\in \sigma_{r}(t)$ there exists $x'\in \tau_e(X_e)\cap \sigma_{r}(t)$ with $d_{\BS}(x,x')\leq \Phi'(d_X(x,\tau_e(X_e)))$.
			\item \label{item:line_cap_surface} If $g_1,g_2 \in G$ so that $g_iZ_\omega\subseteq X_w$ for $i=1,2$, then
			$$d_{\BS}(g_1 Z_\omega \cap \sigma_{r}(s), g_2 Z_\omega\cap \sigma_{r}(s))\leq \Phi'( d_{\BS}(g_1 Z_\omega, g_2 Z_\omega)).$$
		\end{enumerate}
	\end{claim}
	
	\begin{proof}
		\textbf{Proof of \eqref{item:Z_intersects_sigma}:} By Proposition \ref{prop:geometry_of_Zs}.\eqref{item:Zs_QI_embed_in_Xv}, the restriction of $p_v$ to any coset $gZ_\mu \subseteq X_v$ is a uniform quasi-isometry. In particular, $p_v(gZ_\mu)$ uniformly coarsely covers $L_v$. Hence, there is some $r_1 \geq0$ so that for all $r \geq r_1$, $p_v(\sigma_r(t)) \cap p_v(gZ_\mu) \neq \emptyset$, which implies $\sigma_r(t) \cap gZ_\mu \neq \emptyset$.
		
		\textbf{Proof of \eqref{item:vertical}:} The set of cosets  $gZ_\omega$ so that $gZ_\omega \subseteq \tau_{\bar e}(X_e)$ partition $\tau_{\bar e}(X_e)$. Since $\tau_{\bar e}$ is injective, the images of these cosets under $\tau_{\bar e}^{-1}$ will partition $X_e$.  Since $p_v \circ \tau_e(X_e)$ coarsely covers $L_v$ (Proposition \ref{prop:geometry_of_Zs}.\eqref{item:Zs_QI_embed_in_Xv}), this partition implies there must be an $r'_1 \geq 0$ and a coset $gZ_\omega \subset \tau_{\bar e}(X_e)$ so that $p_v(\sigma_{r_1'}(t)) \cap p_v(\tau_e\tau_{\bar e}(gZ_\omega)) \neq \emptyset$.
		By Proposition \ref{prop:geometry_of_Zs}.\eqref{item:bounded_Zs}, the diameter of $p_v(\tau_e\tau_{\bar e}(gZ_\omega))$ is uniformly bounded. Hence  there is some $r_1 \geq r'_1$,  so that whenever $r \geq r_1$,  we have $$\tau_e\tau_{\bar{e}}^{-1}(gZ_\omega)\subseteq \sigma_{r}(t) \cap \tau_e(X_e).$$
		
		Now consider $x\in \sigma_{r}(t) \cap \tau_e(X_e)$. By construction, $d_{L_v}(p_v(x), p_v(\tau_e\tau_{\bar{e}}^{-1}(gZ_\omega)))\leq 2r$, so the fact that $x$ is uniformly close in $X_v$ to $\tau_{e}(\tau_{\bar e}^{-1}( g Z_{\mu}))$ follows from Proposition \ref{prop:geometry_of_Zs}.\eqref{item:distance_to_Zs}. Since the inclusion $X_v \to X$ is distance non-increasing, there is $c \geq0$ depending on  $\mc{G}$ and $r$ so that $d_X(x,\tau_{e}(\tau_{\bar e}^{-1}( g Z_{\mu})) \leq c$.
		
		\textbf{Proof of \eqref{item:close_to_edge}:} Let $r_1$ be the lower bound from the proof of Item \eqref{item:Z_intersects_sigma} so that for all $r \geq r_1$, $\sigma_r(t) \cap gZ_\mu \neq \emptyset$ when $gZ_\mu \subseteq X_v$. Fix $x \in \sigma_r(t)$, and let $\bar{x} \in \tau_e(X_e)$ be a point realizing $d_{X}(p, \tau_e(X_e))$. There exists some coset $gZ_\mu\in \tau_e(X_e)$ such that $\bar{x} \in gZ_\mu$ (because the cosets partition $X_v$). Let $x'$ be a point of the intersection $\sigma_{r}(t) \cap gZ_\mu$. As $x, x'\in \sigma_{r}(t)$, we have $$\vert d_{L_v}( p_v(\bar{x}),p_v(x)) - d_{L_v}(p_v(\bar{x}),p_v( x'))\vert \leq r.$$ By Proposition \ref{prop:geometry_of_Zs}.\eqref{item:Zs_QI_embed_in_Xv}, the map $p_v\colon gZ_\mu \to L_v$ is a $(\kappa,\kappa)$--quasi-isometry for some $\kappa \geq 1$ determined by $\mc{G}$. As $\bar{x}$ and $x'$ both belong to $gZ_\mu$, we have 
		\begin{align*}
			d_X(\bar{x}, x') \leq d_{X_v}(\bar{x},x') \leq& \kappa d_{L_v}(p_v(\bar{x}),p_v(x')) + \kappa \\
			\leq& \kappa d_{L_v}(p_v(\bar{x}),p_v(x)) + \kappa r+ \kappa\\
			\leq& \kappa^2 d_{X_v}(\bar{x},x) + \kappa^3 + \kappa r+ \kappa \\
		\end{align*}
		
		By applying Lemma \ref{lem:uniform_distortion}, the above bound on $d_X(\bar{x},x')$ in terms of $d_{X_v}(\bar{x},x)$ produces a diverging monotone function $\Phi' \colon [0,\infty) \to [0,\infty)$ so that $d_X(\bar{x},x') \leq \Phi'(d_X(\bar{x},x))$. The triangle inequality now  yields
		\[d_X(x, x') \leq d_X(x,\bar{x}) +d_X(\bar{x},x') \leq d_X(x,\bar{x}) + \Phi'(d_X(x,\bar{x})).\] Since $d_X(x,\bar{x}) = d_X(x,\tau_e(X_e))$, this completes the proof of \eqref{item:close_to_edge}.
		
		\textbf{Proof of \eqref{item:line_cap_surface}:} Let $r_1$ be the lower bound from the proof of Item \eqref{item:Z_intersects_sigma} so that for all $r \geq r_1$, $\sigma_r(s) \cap gZ_\omega \neq \emptyset$ when $gZ_\omega \subseteq X_w$. Given $g_1Z_\omega$ and $g_2Z_\omega$ in $X_w$, let $x_i \in g_i Z_\omega$ so that $d_X(x_1,x_2) = d_X(g_1 Z_\omega,g_2 Z_\omega)$. Let $z_1,z_2 $ be the elements of $Z_\omega$ so that $x_i = g_i z_i$ for $i= 1,2$.
		
		We can assume that $x_1 \in \sigma_r(s) \cap g_1 Z_\omega$ by the following argument. Suppose $z$ is the element of $Z_\omega$ so that $g_1z$ is a point in $\sigma_r(s) \cap g_1 Z_\omega$. Because $Z_\omega$ is central in $G_\omega$, we have \[z z_1^{-1} x_1 = z z_1^{-1} g_1 z_1 =  g_1z \in  \sigma_r(s) \cap g_1 Z_\omega \] and \[ z z_1^{-1} x_2 = z z^{-1}_1 g_2 z_2 = g_2 z z_{2}^{-1} z_1 \in g_2 Z_\omega. \] Hence $z z_1^{-1} x_1 \in \sigma_r(s)$ and $z z_1^{-1} x_2$ are points in $g_1 Z_\omega$ and $g_2 Z_\omega$ that realise  $d_X(g_1 Z_\omega,g_2 Z_\omega)$.
		
		We can now proceed by a very similar argument as Item \eqref{item:close_to_edge} using $L_w$ instead of $L_v$. Let $y_2$ be a vertex in $\sigma_r(s) \cap g_2 Z_\omega$. After replacing $gZ_\mu$ with $g_2Z_\omega$ and $L_v$ with $L_w$ in the proof of Item \eqref{item:close_to_edge}, we can repeat the same calculations with $x_1 = x$, $x_2 = \bar{x}$, and $y_2 = x'$, to produce 
		\[ d_X(x_1,y_2)  \leq  d_X(x_1, x_2) + \Phi'(d_X(x_1, x_2)).\] Since  $x_1 \in \sigma_r(s) \cap g_1 Z_\omega$, $y_2 \in \sigma_r(s) \cap g_2 Z_\omega$,  and $d_X(x_1, x_2) = d_X(g_1 Z_\omega,g_2 Z_\omega)$, this completes the proof of \eqref{item:line_cap_surface} with the same function $\Phi'$ as in the proof of \eqref{item:close_to_edge}.
	\end{proof}

	We now use Claim \ref{claim:sigma_L_product} to prove Proposition \ref{prop:P_vs_sigma}.
	Let $t_1, t_1, s\in \ST^{(0)}- T^{(0)}$ be such that $\nu(t_i) = v_i$, $\nu(s) = w$ and $d_T(v_1, v_2) = 2$ with $w$ the only vertex at distance one from both. Let $e_i$ be the edge of $T$ such that $e_i^{+} = v_i$, $e_i^{-} = w$.  Let $r \geq r_1$ where $r_1 \geq 0$ is the lower bound on $r$ from Claim \ref{claim:sigma_L_product}. Consider points $x_i\in\sigma_{r}(t_i)$ so that
	$$d_X(x_1,x_2)=d_X(\sigma_{r}(t_1),\sigma_{r}(t_2)):=d.$$
	We have to show that the $P_{r}(s,t_1)$ and $P_r(s,t_1)$ are at most some function of $d$ apart in $X$. Because the edge spaces separate $X$, we have $d_X(x_i,\tau_{e_i}(X_{e_i}))\leq d$. By Claim \ref{claim:sigma_L_product}.\eqref{item:close_to_edge}, there are points $x'_i\in \tau_{e_i}(X_{e_i}) \cap \sigma_{r}(t_i)$ such that $d_X(x'_1,x'_2)\leq d+ 2\Phi'(d)$. Setting $\omega = \check{w}$,  Claim \ref{claim:sigma_L_product}.\eqref{item:vertical} gives us $c \geq 0$ and  $g_i \in G$ so that each $g_iZ_\omega\subseteq X_w$ and
	$$ \tau_{e}(\tau_{\bar e}^{-1}( g_iZ_{\omega}))\subseteq \sigma_{r}(t_i )\cap \tau_e(X_e) \subseteq \mc{N}_c(\tau_{e}(\tau_{\bar e}^{-1}( g_i Z_{\omega})).$$
	Since the map $\tau_{e} \circ \tau_{\bar e}^{-1}$ moves points distance at most 2, we have 
	$$d_X(g_1 Z_\omega, g_2 Z_\omega) \leq d_X(x'_1,x'_2)+2(c+2)\leq d+ 2\Phi'(d)+2c+4.$$
	Applying Claim \ref{claim:sigma_L_product}.\eqref{item:line_cap_surface}, we have points $x''_i\in g_iZ_\mu\cap \sigma_{r}(s)$ with
	$$d_X(x''_1,x''_2)\leq \Phi'( d_X(g_1 Z_\mu, g_2 Z_\mu)) \leq \Phi'(d+ 2\Phi'(d)+2C+4).$$
	
	The points $x''_i$ are not quite in  $P_{r}(s,t_i)$, but because  $\tau_{e}(\tau_{\bar e}^{-1}( g_iZ_{\omega}))\subseteq \sigma_{r}(t_i )\cap \tau_e(X_e) $, we have 
	$$\tau_{\bar{e_i}}^{-1}(x''_i)\in \tau_{e_i}^{-1}(\sigma_{r}(t_i)).$$
	Since $x''_i \in \sigma_r(s) \cap g_iZ_\omega$, we also have $$\tau_{\bar{e_i}}^{-1}(x''_i)\in \tau_{\bar{e_i}}^{-1}(\sigma_{r}(s)).$$ Hence  $$\tau_{\bar{e_i}}^{-1}(x''_i) \in P_r(s,t_i),$$ which implies 
	$$d_X(P_{r}(s,t_1),P_{r}(s,t_2))\leq \Phi'(d+ 2\Phi'(d)+2C+4)+2,$$
	as desired for Proposition \ref{prop:P_vs_sigma}.
\end{proof}

\begin{lem}\label{lem:W_is_a_G_space}
	Let $r_0$ and $r_1$ be as in Lemma \ref{lem:R_for_sigmas_to_intersect} and Proposition \ref{prop:P_vs_sigma} respectively. For all $r > \max\{r_0,r_1\}$ there is $R_0 = R_0(r) \geq 0$ so that 
	\begin{itemize}
		\item for all $R>R_0$ the graph $W_{r,R}$ is connected;
		\item the $G$--action on $W_{r,R}$ by $g\cdot \Sigma(s,t)= \Sigma(gs,gt)$ is metrically proper and cobounded.
	\end{itemize}
\end{lem}

\begin{proof}
	Fix $r \geq \max\{r_0,r_1\}$ and let $W = W_{r,R}$ for a choice of $R$ decided below.
	
	\textbf{$W$ is connected:}  Because the Bass--Serre tree $T$ is connected, given any two maximal simplices $\Sigma(s,t), \Sigma(s',t') $ of $\ST$, we can find a sequence of maximal simplices $\Sigma(s_i,t_i)$ so that $\nu(\Sigma(s_i,t_i))$ produces a path in $T$ from $\nu(\Sigma(s,t))$ to $\nu(\Sigma(s',t'))$. Hence, it suffices to prove that two vertices $\Sigma(s,t), \Sigma(s',t') \in W$ with $\nu(s) = \nu(s')$ can connected by a path in $W$. 
	
	First assume $\nu(s) = \nu(s')$ and $\nu(t) = \nu(t')$. Then $P_{r}(s,t)$ and $P_{r}(s',t')$ are both subsets of $X_e$.  Let $e$ be the edge of $T$ between $\nu(s)$ and $\nu(t)$ and let $h$ be an element of $G$ so that  $\stab_G(e) = h G_{\check{e}} h^{-1}$. Because $\stab_G(e)$ acts transitively on the vertices of $\BS_e$, there is $k \in \stab_G(e)$ so that $P_{r}(ks,kt) \cap P_{r}(s',t') \neq \emptyset$ and hence $\Sigma(ks,kt)$ is joined by a $W$--edge to $\Sigma(s',t')$. Because $\stab_G(e)$ is generated by the finite set $hJ_{\check{e}}h^{-1}$,  $\Sigma(s,t)$ will be connected to $\Sigma(ks,kt)$---and hence $\Sigma(s',t')$---if $\Sigma(gs,gt)$ is connected to $\Sigma(s,t)$ by a $W$--edge for each $g \in hJ_{\check{e}}h^{-1}$. There exists $R_1 \ge r$ depending only on $J_{\check{e}}$  so that $d_{\BS_e}(P_{r}(s,t),  P_{r}(gs,gt)) \leq R_1$ for all $g \in hJ_{\check{e}}h^{-1}$. Thus,  $\Sigma(gs,gt)$ is connected to $\Sigma(s,t)$ by a $W$--edge for each $g \in hJ_{\check{e}}h^{-1}$ provided $R \geq R_1$. 
	
	Now assume $\nu(s) = \nu(s')$, but $\nu(t) \neq \nu(t')$. Let $e_1$ be the edge of $T$ from $\nu(t)$ to $\nu(s)$ and $e_2$ be the edge of $T$ from $\nu(t')$ to $\nu(s') = \nu(s)$. Let $v = \nu(s)$ and $h$ be an element of $G$ so that $\stab_G(v) = h G_{\check{v}} h^{-1}$. Because $\stab_G(v)$ acts transitively on the vertices of $\BS_{v}$, there is $k \in \stab_G(v)$ so that $k \cdot \tau_{e_1}(P_{r}(s,t)) \cap \tau_{e_2}(P_{r}(s',t')) \neq \emptyset$. Thus, $d_{\BS}(P_{r}(ks,kt), P_{r}(s',t')) \leq 2$. Hence, if $R \geq 2$, then $\Sigma(ks,kt)$ is joined by a $W$--edge to $\Sigma(s',t')$. 
	Because $\stab_G(v)$ is generated by the finite set $hJ_{\check{v}}h^{-1}$, $\Sigma(s,t)$ will be connected to $\Sigma(ks,kt)$ if $\Sigma(gs,gt)$ is connected to $\Sigma(s,t)$ by a $W$--edge for each $g \in hJ_{\check{v}}h^{-1}$. 
	There exist $R_2 \geq r$ depending only on  $J_{\check{v}}$ so that $d_{\BS}(\tau_{e_1}(P_{r}(s,t)), g \tau_{e_1}(P_{r}(s,t)))\leq R_2$ for all $g \in hJ_{\check{v}}h^{-1}$. 
	Thus, $d_{\BS}(P_{r}(gs,gt),P_{r}(s,t)) \leq R_2+2$ for all $g \in hJ_{\check{v}}h^{-1}$ and hence  $\Sigma(gs,gt)$ is connected to $\Sigma(s,t)$ be a $W$--edge for each $g \in hJ_{\check{v}}h^{-1}$ whenever $R \geq R_2+2$.
	
	Because $R_1$ and $R_2$ depend only on the choice of finite generating set for the vertex and edge groups of $\mc{G}$, they can be chosen to be uniform for each vertex and edge of $\mc{G}$. Thus,  $W$ is connected whenever $R \geq R_0 = R_1 +R_2 +2$.

	\textbf{$G$ acts properly:} Let $K_W$ be a bounded subset of $W$ and let $K_X$ be the subset of $\BS$ that is the union $ \bigcup_{\Sigma(s,t) \in K_W}P_{r}(s,t)$. We note that when we have $W$--adjacent maximal simplices $\Sigma(s,t),\Sigma(s',t')$ then $d_X(P_{r}(s,t),P_{r}(s',t'))$ is uniformly bounded. Indeed, if the $W$--edge is as in the first bullet of Definition \ref{defn:W}, this is clear, and otherwise this follows from Proposition \ref{prop:P_vs_sigma}. Therefore, $K_X$ is a bounded subset of $X$. Since the action of $G$ on $\BS$ is metrically proper, the set $\{ g \in G : K_X \cap gK_X \neq  \emptyset\}$ is finite.  Now, $$\{ g \in G : K_W \cap gK_W \neq  \emptyset\} \subseteq \{ g \in G : K_X \cap gK_X \neq  \emptyset\}$$ because whenever $\Sigma(s,t)$ and $g\Sigma(s,t)$ are both in $K_W$,  $P_r(s,t)$ and $gP_r(s,t)$ are both contained in $K_X$. As the latter set is finite, the claim follows.

	\textbf{$G$ acts coboundedly:} Since $W$ is connected and $G$ acts cofinitely on the edges of $T$, it suffices to prove that for any edge $e$ of $T$, any two maximal simplices of $\ST$ that contain the edge $e$ have $\stab_G(e)$ translates that are $W$--adjacent.  Let $\Sigma(s,t)$ and $\Sigma(s',t')$ be two maximal simplices that contain the edge $e$. Since $\Stab_G(e)$ acts transitively on the vertices of $\BS_e$, there exists $g \in \Stab_G(e)$ so that $P_{r}(gs,gt) \cap P_{r}(s',t') \neq \emptyset$. Thus, $d_{\BS}(P_{r}(gs,gt),P_{r}(s',t')) \leq R$ and $\Sigma(gs,gt)$ is $W$--adjacent to $\Sigma(s',t')$ as desired.
\end{proof}

\section{Verification of combinatorial HHS axioms}\label{sec:proof}
We now verify that the pair $(\ST,W)$ from Section \ref{sec:blow_up} is a combinatorial HHS. For our \squidable graph of groups $\mc{G}$, we fix the same notation as the beginning of Section \ref{sec:blow_up} and let $\ST$ be the simplicial complex from Definition \ref{defn:ST} for $\mc{G}$. We continue to use $\Sigma(s,t)$ to denote the maximal simplex of $\ST$ determined by $s,t \in \ST^{(0)} - T^{(0)}$ (see Lemma \ref{lem:maximal_simplices}) and let  $\sigma_r(s)$ and $P_r(s,t)$ be the sets from Definition \ref{defn:level_surfaces} and \ref{defn:P_sets} respectively.

Fix $r \ge 0$ and $R\geq 0$ large enough that Lemma \ref{lem:W_is_a_G_space} ensures the graph $W_{r,R}$ is connected and has a metrically proper and cobounded action of $G$. Moreover, choose $R$ to be larger than $2\xi$ where $\xi = \xi(r)$ is the constant from  Lemma~ \ref{lem:R_for_sigmas_to_intersect}. With these values of $r$ and $R$ fixed, we let $W = W_{r,R}$.

Our proof that $(\ST,W)$ is a combinatorial HHS is spread over three subsections. Section \ref{subsec:combinatorial_conditions} contains a description of the links of the non-maximal simplices of $\ST$ and verifies parts \eqref{CHHS:finite_complexity}, \eqref{CHHS:comb_nesting_condition}, and \eqref{CHHS:C=C_0_condition} of the definition of a combinatorial HHS (Definition \ref{defn:CHHS}). This section also includes a proof that the action of $G$ on $\ST$ has finitely many orbits of links of simplices. Section \ref{subsec:hyperbolicity} proves the augmented links $\mc{C}(\Delta)$ for simplices are hyperbolic, while Section \ref{subsec:QI_embedding} prove that they quasi-isometrically embed in the space $Y_\Delta$. These are condition \eqref{CHHS:hyp_X} and \eqref{CHHS:geom_link_condition} of Definition \ref{defn:CHHS}.

\subsection{Simplices, links, and the combinatorial conditions}\label{subsec:combinatorial_conditions}
We now describe the combinatorics of simplices and their links in $\ST$ and then verify three of the conditions for $(\ST, W)$ to be a combinatorial HHS. In what follows, $\lk(\cdot)$ denotes the link in $\ST$, while $\lk_T(\cdot)$ denotes the link  in $|T|$, the unoriented graph obtained from $T$ by replacing each pair of oriented edges with an unoriented edge. Similarly, we use $d_T(\cdot,\cdot)$ to denote the distance in $|T|$ between two vertices of $T$.

A basic consequence of the description of  maximal simplices of $\ST$ (Lemma \ref{lem:maximal_simplices}), is that non-empty, non-maximal simplices come in one of the following types.

\begin{cor}\label{cor:simplex_type}
	Every non-maximal, non-empty simplex $\Delta$ of $\ST$ is one of the following 8 types
	\begin{enumerate}[{Type} 1:]
		\item $\Delta = \{v\}$ for some $v \in T^{(0)}$ \label{simplex:T_vertex}
		\item $\Delta = \{s\}$ for some $s \in \ST^{(0)} - T^{(0)}$ \label{simplex:Squid_vertex}
		\item $\Delta = \{v,w\}$ for some $v,w \in T^{(0)}$ \label{simplex:T_edge}
		\item $\Delta = \{s,t\}$ for some $s,t \in \ST^{(0)} - T^{(0)}$ \label{simplex:Squid_edge}
		\item $\Delta = \{s,v\}$ for some $v \in T^{(0)}$ and $s \in \ST^{(0)} - T^{(0)}$ with $\nu(s) \neq v$ \label{simplex:cross_edge}
		\item $\Delta = \{s,\nu(s), t\}$ for some  $s,t \in \ST^{(0)} - T^{(0)}$ \label{simplex:upper_triangle}
		\item $\Delta = \{s,\nu(s)\}$ for some  $s,t \in \ST^{(0)} - T^{(0)}$ \label{simplex:level_surface_type}
		\item $\Delta = \{ s,\nu(s), v\}$ for some $v \in T^{(0)}$ and $s \in \ST^{(0)} - T^{(0)}$  with $\nu(s) \neq v$. \label{simplex:quasi-line_type}
	\end{enumerate} 
\end{cor}

\begin{proof}
	Since every non-maximal simplex can be completed to a maximal simplex by adding vertices, the above list is a consequence of Lemma \ref{lem:maximal_simplices}
\end{proof}

By examining each type of simplex, we also obtain a description of the links of each type of simplex. Figure~\ref{fig:link_types} contains a schematic of each type of simplex along with its link and will be a useful reference through this section.

\begin{figure}[ht]
	\centering
	\begin{tikzpicture}
		\node[anchor=south west,inner sep=0] (image) at (0,0) {\includegraphics[width=\textwidth]{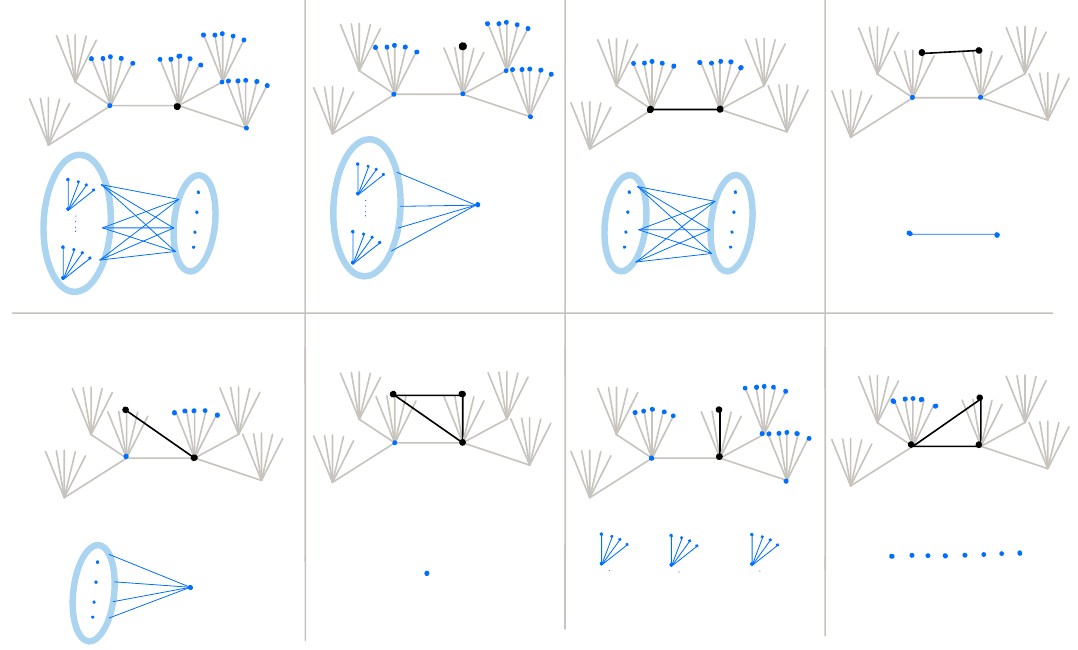}};
		\begin{scope}[x={(image.south east)},y={(image.north west)}]
			\node at (0.03, 0.98) {1};
			\node at (0.3, 0.98) {2};
			\node at (0.54, 0.98) {3};
			\node at (0.78, 0.98) {4};
			\node at (0.03, 0.48) {5};
			\node at (0.3, 0.49) {6};
			\node at (0.54, 0.49) {7};
			\node at (0.78, 0.49) {8};
		\end{scope}
	\end{tikzpicture}
	\caption{\label{fig:link_types} A schematic of each type of simplex and its link. The simplex is drawn in black with the vertices of the link highlighted in blue. Below, a schematic of the link is drawn in blue. To avoid clutter, most edges between vertices $s,t$ with $\nu(s) \neq \nu(t)$ are missing, as can be seen in the links of Type 1 and  Type 2.} \label{Figure:Failure_convexity}
\end{figure}
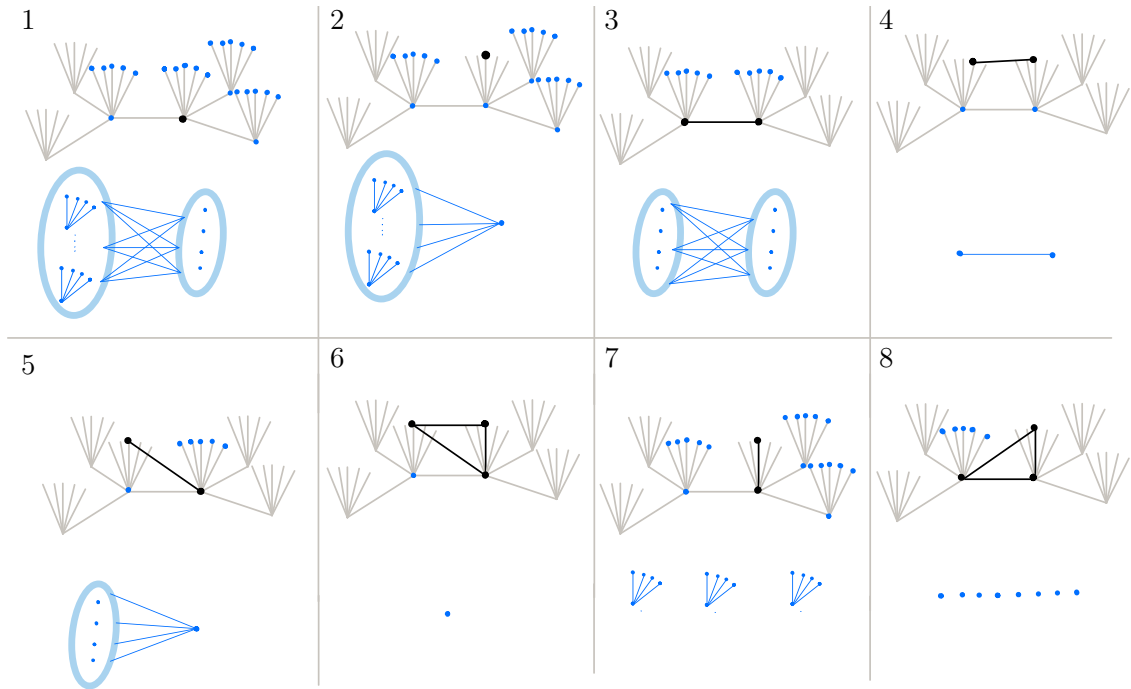

\begin{lem}\label{lem:description_of_links}
	Let $\Delta$ be a non-maximal, non-empty simplex of $\ST$. The link of $\Delta$ is determined by the type of $\Delta$ as follows, where $v,w\in T^{(0)}$ and $s,t \in \ST^{(0)} - T^{(0)}$:
	
	\begin{enumerate}[{Type} 1:]
		\item if $\Delta = \{v\}$, then	$\lk(\Delta)$ is the join of $\{s \in \ST^{(0)} - T^{(0)}: \nu(s) = v\}$ with the span of  $\{ t \in \ST^{(0)}: \nu(t) \in \lk_T(v)\}$.
		\item if $\Delta = \{s\}$, then $\lk(\Delta)$ is the join of $\{\nu(s)\}$ and the span of $$\{ t \in \ST^{(0)}: \nu(t) \in \lk_T(\nu(s))\}.$$
		\item if $\Delta = \{v,w\}$, then $\lk(\Delta)$ is the join of $\{s \in \ST^{(0)} - T^{(0)}: \nu(s) = v\}$ with $\{ t \in \ST^{(0)} - T^{(0)}: \nu(t) = w \}$.
		\item if $\Delta = \{s,t\}$, then  $\lk(\Delta)$ is the edge between $\nu(s)$ and $\nu(t)$.
		\item if $\Delta = \{s,w\}$, then $\lk(\Delta)$ is the join of $\{\nu(s)\}$ and $\{ t \in \ST^{(0)} - T^{(0)}: \nu(t) = w \}$.
		\item if $\Delta = \{s,\nu(s), t\}$, then $\lk(\Delta)$ is the vertex $\nu(t)$.
		\item if $\Delta = \{s,\nu(s)\}$, then $\lk(\Delta)$ is spanned by $\{ t \in \ST^{(0)}: \nu(t) \in \lk_T(v)\}$.
		\item if $\Delta =   \{ s,\nu(s), v\}$, then $\lk(\Delta)$ is $X_v^{(0)} = 
  \{ t \in \ST^{(0)} - T^{(0)}: \nu(t) = v\}$.
	\end{enumerate}
	In particular, if $\Delta$ is not of Type \ref{simplex:level_surface_type} or Type \ref{simplex:quasi-line_type}, then $\mc{C}(\Delta)$ has diameter at most $3$, by virtue of being a single vertex or a non-trivial join with some added edges.
\end{lem}

\begin{proof}
	All cases  are a straightforward exercise using the definitions of edges of $\ST$. The ``in particular'' clause follows as $\mc{C}(\Delta)$ is obtained from adding edges to $\lk(\Delta)$.
\end{proof}

When the link of a simplex is not a non-trivial join, we will need to understand its saturation (Definition \ref{defn:saturation}) in order to understand the space $Y_\Delta$.

\begin{lem}\label{lem:descriptions_of_saturations} Let $\Delta$ be a non-empty, non-maximal simplex of $\ST$. 
	\begin{enumerate}
		\item If $ \Delta = \{s,\nu(s)\}$ is a simplex of Type \ref{simplex:level_surface_type}, then $$\Sat(\Delta)=\{\nu(s)\}\cup \{s' \in \ST^{(0)} - T^{(0)}: \nu(s') = \nu(s)\}.$$	
		\item  If $\Delta =  \{ s,\nu(s), v\}$ is a simplex of Type \ref{simplex:quasi-line_type}, then 
		$$\Sat(\Delta)=\{u\in T^{(0)}: d_T(v,u)\leq 1\}\cup \{t\in \ST^{(0)}- T^{(0)}: d_T(v,\nu(t))=1\}.$$
	\end{enumerate}
\end{lem}

\begin{proof}
	\textbf{Case 1:}  $ \Delta = \{s,\nu(s)\}$  is a simplex of Type \ref{simplex:level_surface_type}. First, suppose that $s'$ is a vertex with $\nu(s)=\nu(s')$ and $s'\neq \nu(s)$, so $\Delta'=\{s',\nu(s)\}$ is 
	a simplex of $\ST$.  If $u \in \ST$ is a vertex adjacent to both $s'$ and $\nu(s') =\nu(s)$, then $\nu(u)$ is adjacent to 
	$\nu(s)$, which makes $u$  adjacent to $s$. Hence $u\in\link(\Delta)$, and we have $\link(\Delta')\subseteq\link(\Delta)$.  
	By a symmetrical argument, $\link(\Delta')=\link(\Delta)$. Thus, every simplex of the form $\{s',\nu(s)\}$ with $\nu(s)=\nu(s')$ and $\nu(s)\neq s'$ has the same link as 
	$\Delta$.  In particular, every such $s'$ is in $\Sat(\Delta)$ and $\nu(s)$ is in $\Sat(\Delta)$.
	
	Conversely, suppose that $\Delta'$ is a simplex with $\link(\Delta')=\link(\Delta)$.  Then 
	$\nu(\Delta')=\nu(s)$ as $\lk_T(\nu(\Delta)) = \lk_T(\nu(s))$.  Now, $\Delta'$ cannot be $\nu(s)\in \ST$, because then its link would contain vertices 
	$s'$ with $\nu(s)=\nu(s')$, which are not in $\link(\Delta)$.  On the other hand, if  $\Delta'=s'$ for some $s' \in \ST$ with $\nu(s') = \nu(s)$, then $\link(\Delta')$ would contain $\nu(s)$, which is not in $\link(\Delta)$.  So $\Delta'$ must be equal to $\{s',\nu(s)\}$ for 
	some $s'\neq \nu(s)$ with $\nu(s)=\nu(s')$.  By definition, $\Sat(\Delta)$ is the union of these 
	$\{s',\nu(s)\}$, which completes the proof  that  \[\Sat(\Delta) = \bigcup_{\nu(s')=\nu(s)}\{s',\nu(s)\}.\]

	\textbf{Case 2:} $\Delta =  \{ s,\nu(s), v\}$ is a simplex of Type \ref{simplex:quasi-line_type}. Let $u\in \Sat(\Delta)$.  If $u\in T^{(0)}$, then $u$ is 
	adjacent to or equal to $v$ in $\ST$ and hence in $T$.  If $u\in \ST^{(0)} - T^{(0)}$, then $\nu(u)$ is adjacent to or equal to $v$.  Conversely, 
	suppose that $u\in T^{(0)}$ and $d_T(u,v)=1$.  Choose any vertex $t\in\nu^{-1}(u)$ with $t \neq u$.  Then $\{u,t,v\}$ is a simplex 
	with link $\lk(\Delta)$.  Next, suppose that $u\in \ST^{(0)} - T^{(0)}$ and $d_T(v,\nu(u))\leq 1$.  Then 
	$\Delta'=\{u,\nu(u),v\}$ is a simplex with  $\lk(\Delta') = \lk(\Delta)$. Together, these show the $$\Sat(\Delta)=\{u\in T^{(0)}: d_T(v,u)\leq 1\}\cup \{t\in \ST^{(0)}- T^{(0)}: d_T(v,\nu(t))=1\}.$$
\end{proof}

We now verify  conditions \eqref{CHHS:finite_complexity}, \eqref{CHHS:comb_nesting_condition}, \eqref{CHHS:C=C_0_condition}  from Definition \ref{defn:CHHS} for $(\ST,W)$ to be a combinatorial HHS.

\begin{lem}\label{lem:finite_complexity}
	If $\Delta_1,\dots,\Delta_n$ are simplices of $\ST$ such that $\lk(\Delta_1) \subsetneq \dots \subsetneq \lk(\Delta_n)$, then $n \leq 5$.
\end{lem}

\begin{proof}
	Corollary~\ref{cor:simplex_type} lists all types of non-maximal, non-empty simplices of $\ST$. Examining the links for each of these different types of simplices (Lemma \ref{lem:description_of_links}) shows that if $\lk(\Delta) \subsetneq \lk(\Delta')$, then $\Delta$ must have strictly more vertices that $\Delta'$. Thus, any chain of strictly nested links of simplices must have length at most $5$ (recall, $\lk(\emptyset) = \ST$ by definition).
\end{proof}

\begin{lem}\label{lem:C=C0_for_W}
	Let $\Delta$ be a simplex of $\ST$ and $x,y \in \lk(\Delta)^{(0)}$ be vertices that are not adjacent in $\ST$, but are adjacent in $\ST^{+W}$. Then there exist two maximal simplices $\Sigma_x, \Sigma_y \subseteq \st(\Delta)$ that respectively contain $x$ and $y$ such that $\Sigma_x$ and $\Sigma_y$ are adjacent in $W$.
\end{lem}

\begin{proof}
	Let $\Delta$ be a simplex of $\ST$ and $x,y \in \lk(\Delta)^{(0)}$  that are not adjacent in $\ST$, but are adjacent in $\ST^{+W}$. Let $s_x,t_x$ and $s_y,t_y$ be the elements of $\ST - T$ so that $x \in \Sigma(s_x,t_x)$, $y \in \Sigma(s_y,t_y)$ and $\Sigma(s_x,t_x)$ is  $W$--adjacent to $\Sigma(s_y,t_y)$. Without loss of generality, assume $x$ and $y$ are respectively contained in the edges $\{t_x,\nu(t_x)\}$ and $\{t_y,\nu(t_y)\}$. It suffices to find $s \in \ST-T$ so that $\Delta \subseteq \st(s)$ and the simplices $\Sigma(s,t_x)$ and $\Sigma(s,t_y)$ are $W$--adjacent.
	
	First assume that $\nu(x) \neq \nu(y)$.  Since $x$ and $y$ are not joined by an $\ST$--edge, $\nu(x)$ cannot be joined to $\nu(y)$ by an edge in $T$. Thus, there must exist $s \in \ST^{(0)}-T^{(0)}$ so that  $\Delta$ is contained in the edge $\{s,\nu(s)\}$ and $\nu(x),\nu(y) \subseteq \lk(\nu(s))$.  The simplices $\Sigma(s,t_x)$ and $\Sigma(s,t_y)$ are therefore $W$--adjacent, since $\Sigma(s_x,t_x)$ and  $\Sigma(s_y,t_y)$ being $W$--adjacent implies that $d_\BS(\sigma_{r}(t_x),\sigma_{r}(t_y)) \leq R+2$ in both cases of edges in $W$.
	
	Now assume that $\nu(x) = \nu(y)$. Since $x$ and $y$ are not joined by an $\ST$--edge, both $x$ and $y$ must be elements of $\ST^{(0)}-T^{(0)}$. This implies that $\Delta$ is contained in a $2$--simplex of the form $\{s,\nu(s),\nu(x)\}$ where $s \in \ST-T$ with $\nu(s) \subseteq \lk(\nu(x))$. Since $x \neq y$ and $\Sigma(s_x,t_x)$ is $W$--adjacent to $\Sigma(s_y,t_y)$, we must have $t_x = x$, $t_y = y$, and $d_\BS(\sigma_{r}(x),\sigma_{r}(y)) \leq R + 2$ in both case of edges in $W$. Thus, the simplices $\Sigma(s,x)$ and $\Sigma(s,y)$ are connected by an edge in $W$.
\end{proof}

\begin{lem}\label{lem:pseduo_wedges}
	For any non-maximal simplices $\Delta$ and $\Omega$ of $\ST$ there exists a (possibly empty) simplex $\Pi$ of $\lk(\Delta)$ such that $\lk(\Delta \star \Pi) \subseteq \lk(\Omega)$ and for all non-maximal simplices $\Lambda$ of $\ST$ so that $\lk(\Lambda) \subseteq \lk(\Delta) \cap \lk(\Omega)$ either 
	\begin{enumerate}
		\item $\lk(\Lambda)$ is a non-trivial join or a vertex; or
		\item $\lk(\Lambda) \subseteq \lk(\Delta \star \Pi)$.
	\end{enumerate}
\end{lem}

\begin{proof}
	First of all, we will implicitly assume throughout the proof that the link of the empty simplex is not contained in $\lk(\Delta)\cap \lk(\Omega)$, for otherwise we have $\Delta=\Omega=\emptyset$, and we can take $\Pi$ to be empty as well.
	
	Let $\Delta$ and $\Omega$ be as in the statement, and let $U$ denote the union of all $\link(\Lambda)\subseteq \lk(\Delta) \cap \lk(\Omega)$ such that $\link(\Lambda)$ is neither a non-trivial join or a single vertex. It suffices to show $U=\link (\Delta \star \Pi)$ for some simplex $\Pi$. Note that if $\Lambda$ is a non-empty simplex of $\ST$ so that $\lk(\Lambda)$ is not a single vertex nor a non-trivial join, then $\Lambda$ is either a Type \ref{simplex:level_surface_type} or Type \ref{simplex:quasi-line_type} simplex.
	
	We say that a subgraph $\mathcal X$ of $\ST$ satisfies property $P$ if the following holds. For all vertices $v$ of $T$, if there exist two vertices $x,y\in \mathcal X$ with $\nu(x)$, $\nu(y)$ at distance 1 from $v$ in $T$, then we have that $\mathcal X$ contains the entire Type \ref{simplex:level_surface_type} link of a simplex $\{s,\nu(s)=v\}$. 
	
	We make two preliminary observations about this property. First, if two subgraphs satisfy property $P$, then their intersection does as well. Secondly, given a subgraph $\mathcal X$ satisfying property $P$, the (possibly empty) union of all links of Type \ref{simplex:level_surface_type} or Type \ref{simplex:quasi-line_type} contained in $\mathcal X$ satisfies property $P$.
	
	By inspection of the list of possible links of a non-empty simplex (Lemma \ref{lem:description_of_links} and Figure \ref{fig:link_types}), we can check that links satisfy property $P$. In view of the observations above, given simplices $\Delta$ and $\Omega$, the subgraph $U$ of $\link(\Delta)$ considered above also satisfies property $P$.
	
	To conclude the proof, we go through the list of possible links one more time and we check that, given any simplex $\Delta$ and any union $U$ of links of Type \ref{simplex:level_surface_type} or Type \ref{simplex:quasi-line_type} contained in $\link(\Delta)$ satisfying property $P$, we have $U=\link(\Delta\star\Pi)$.
	(Note that $U=\link(\Delta\star\Pi)$ is equivalent to $U$ being a link as a subgraph of $\link(\Delta)$, and note also that if $U$ is empty then it suffices to take $\Pi$ to be a maximal simplex in $\link(\Delta)$.)
\end{proof}

We conclude this subsection by verifying that the action of $G$ on $\ST$ has finitely many orbits of link of simplices.
\begin{lem}\label{lem:ST_cocompact}
	The action of $G$ on $\ST$ has finitely many orbits of links of simplices.
\end{lem}

\begin{proof}
	Let $\Delta$ be a simplex of $\ST$.  If $\Delta$ is maximal, then $\lk(\Delta)=\emptyset$, and if $\Delta=\emptyset$, then $\lk(\Delta)=\ST$, and we are done.
	
	If $\Delta$ is spanned entirely by vertices of $T$ (Type \ref{simplex:T_vertex} or  Type \ref{simplex:T_edge}), then $\Delta$---and hence $\lk(\Delta)$---belongs to one of finitely many $G$--orbits.  Similarly,  because the $G$--stabiliser of a vertex $v \in T^{(0)}$ acts cofinitely on the set $X_v^{(0)} = \{s \in \ST^{(0)} - T^{(0)} : \nu(s) = v\}$, there are finitely many  $G$--orbits of vertices of $\ST$ (Type \ref{simplex:Squid_vertex} simplices) and   simplices of Type \ref{simplex:level_surface_type}, i.e., $\Delta = \{s,\nu(s)\}$ for $s \in \ST^{(0)}- T^{(0)}$. 	Hence, there finitely many $G$--orbits of these types of simplices and their links.

	If $\Delta$ is of Type \ref{simplex:Squid_edge} or Type \ref{simplex:upper_triangle}, then $\lk(\Delta)$ is either an unoriented edge or a vertex of $T$ (Lemma \ref{lem:description_of_links}), of which there are finitely many $G$--orbits of both.
	
	If  $\Delta =\{s,\nu(s),v\}$ is a simplex of Type \ref{simplex:quasi-line_type}, then $\lk(\Delta)$ is $$X_v^{(0)} = \{t \in \ST^{(0)} - T^{(0)} : \nu(t) = v\}.$$ There are only finitely many $G$--orbits of these sets as there are finitely many $G$--orbits of vertices in $T$.
	
	Finally, let $\Delta_1 = \{s_1,v_1\}$ and $\Delta_2= \{s_2,v_2\}$ be two simplices of  Type \ref{simplex:cross_edge}. For each $\Delta_i$, there is an oriented edge $e_i$ of $T$ from $\nu(s_i)$ to $v_i$. If $g \in G$ so that $g e_1 = e_2$, then $\lk(g \Delta_1) = \lk(\Delta_2)$ (even though $g\Delta_1$ might not equal $\Delta_2$). As there are finitely many $G$--orbits of edges of $T$, this shows there are only finitely many $G$--orbits of links of Type \ref{simplex:cross_edge} simplices.

	Examining the lists of types of simplices in Corollary~\ref{cor:simplex_type}, we see that the preceding discussion exhausts all the possibilities.
\end{proof}

\subsection{Hyperbolicity of non-join links}\label{subsec:hyperbolicity}
Recall from Section \ref{subsec:HHS} that $\mc{C}(\Delta)$ is the graph obtained from $\lk(\Delta)$ by adding an edge between every pair of of vertices $x,y$ for which there exists maximal simplices $\Sigma_x$, $\Sigma_y$  that are joined by an edge of $W$ and contain $x$ and $y$ respectively. In this section, we verify that each $\mc{C}(\Delta)$ is hyperbolic, which is condition \eqref{CHHS:hyp_X} of Definition \ref{defn:CHHS} for $(\ST,W)$ to be a combinatorial HHS. Since there are only finitely many $G$--orbits of the $\mc{C}(\Delta)$ by Lemma \ref{lem:ST_cocompact}, the hyperbolicity constant will automatically be uniform over all simplices of $\ST$ (although this fact is independently explicit in our proof).  We only need to verify $\mc{C}(\Delta)$ is hyperbolic  for the empty simplex and simplices of Type \ref{simplex:level_surface_type} and \ref{simplex:quasi-line_type} as Lemma \ref{lem:description_of_links} showed $\mc{C}(\Delta)$ has diameter 2 in all other cases.

\begin{prop}[Unbounded augmented links are hyperbolic]\label{prop:Hyperbolicity_of_links}
	Let $\Delta$ be a simplex of $\ST$.
	\begin{enumerate}
		\item If $\Delta = \emptyset$, then $\ST^{+W} =  \mc{C}(\emptyset)$ is ($G$--equivariantly) quasi-isometric to $|T|$. Hence $\ST^{+W}$ is a quasi-tree.\label{item:1}
		\item If $\Delta = \{s,\nu(s), v\}$ is a simplex of Type \ref{simplex:quasi-line_type}, then the identity map on vertices gives a uniform quasi-isometry from $\mc{C}(\Delta)$ to the quasi-line $L_v$.  \label{item:2}
		\item If $\Delta = \{s, \nu(s)\}$ is a simplex of Type \ref{simplex:level_surface_type}, then $\mc{C}(\Delta)$ is uniformly 
		hyperbolic.  Moreover, if every vertex group in $\mathcal G$ is virtually free, then $\mc{C}(\Delta)$ is 
		quasi-isometric to a tree.   \label{item:3} 
	\end{enumerate}
\end{prop}

\begin{proof}
	We prove the three  case separately.
	
	\textbf{Proof of~\eqref{item:1}.}  The inclusion $|T|\to \ST^{+W}$ is simplicial and hence Lipschitz, thus it suffices to find a coarsely Lipschitz quasi-inverse for the inclusion. This quasi-inverse is provided by the  map $\nu:\ST^{(0)}\to T^{(0)}$, where $\ST^{(0)}$ is equipped with 
	the metric inherited from $\ST^{+W}$.  	To show that the map $\nu$ is  coarsely Lipschitz it suffices to prove that $d_T(\nu(x),\nu(y))$ is uniformly bounded whenever $x,y\in \ST^{(0)}$ are joined by an edge of $\ST^{+W}$.

	If $x$ and $y$ are joined by an edge of $\ST$ 
	then $\nu(x)$ and $\nu(y)$ are equal or joined by an edge of $\ST$ as well, hence $d_T(\nu(x),\nu(y)) =1$.
	Now assume $x,y$ are joined by a $W$--edge.   This means that $x,y$ respectively belong to maximal 
	simplices $\Sigma(s,t)$ and $\Sigma(s',t')$ that are adjacent in $W$. The definition of edges of $W$ (Definition \ref{defn:W}), allows us to assume that $\nu(s) = \nu(s')$ without loss of generality.
	Hence $\nu(x)$ and $\nu(y)$ are both either equal to $\nu(s)$  or joined by an edge of $T$ to $\nu(s)$. Hence $d_T(\nu(x),\nu(y)) \leq 2$ as desired. 
	
	Since $\nu$ and the inclusion are $G$--equivariant, the quasi-isometry is also $G$--equivariant. This completes the proof of~\eqref{item:1}.
	
	\textbf{Proof of~\eqref{item:2}.} Let $\Delta =\{s,\nu(s), v\}$ be a Type \ref{simplex:quasi-line_type} simplex, and let $\mu = \check{v}$. By Lemma \ref{lem:description_of_links}, the vertex set of $\mc{C}(\Delta)$ is exactly   $\lk(\Delta)$, which is the set of vertices $$X_v^{(0)} =  \{t\in \ST^{(0)} -T^{(0)} : \nu(t)=v\}.$$    Recall that $L_v$ is a copy of the vertex space $X_v$ with extra edges between vertices that make $L_v$ a quasi-line. Let $I\colon \mc{C}(\Delta)^{(0)}\to L_v^{(0)}$ be the identity on the vertex set. 
	
	We first show that $I^{-1}$ sends edges of $L_v$ to edges of $\mc{C}(\Delta)$.
	
	\begin{claim}\label{claim:L_v_edges_are_C(Delta)_edge}
		If $t_1,t_2 \in L_v^{(0)}$ are joined by an edge of $L_v$, then $I^{-1}(t_1)$ and $I^{-1}(t_2)$ are joined by an edge of $\mc{C}(\Delta)$. In particular, $\mc{C}(\Delta)$ is connected and $I^{-1}$ is 1-Lipschitz.
	\end{claim}
	
	\begin{proof}
		Let $t_1,t_2\in L_v^{(0)}$ be joined by an edge of $L_v$.  By Lemma~\ref{lem:R_for_sigmas_to_intersect}, the there is a $(\xi,\xi)$--coarsely Lipschitz $\xi$--coarse  map $L_{\nu(s)}\times L_v\to X$ that sends $(s,t_1)$ to $P_{r}(s,t_1)$ and $(s,t_2)$ to $P_{r}(s,t_2)$. 
		Because we chose $R$ to be greater that $2\xi$, this implies $d_{X}(P_{r}(s,t_1),P_{r}(s,t_2))\leq R$. Thus, the maximal simplices $\{s,\nu(s),t_1,v\}$ and 
		$\{s,\nu(s),t_2,v\}$ are joined by an edge in $W$, which implies $t_1,t_2$ are joined by an edge in $\mc{C}(\Delta)$.
	\end{proof}
	
	We now prove that $I$ is also coarsely Lipschitz. This will complete the proof of~\eqref{item:2}.
	
	\begin{claim}\label{claim:I_coarse_lip}
		The map $I$ is coarsely Lipschitz.
	\end{claim}
	
	\begin{proof}
		It suffices to show that  whenever $t_1,t_2\in \mc{C}(\Delta)^{(0)}$ are adjacent in $\mc{C}(\Delta)$, that $I(t_1)$ and $I(t_2)$ are uniformly close in $L_v$. Since $\lk(\Delta)$ contains no edges in this case, the only way $t_1,t_2$ can being joined by an edge in $\mc{C}(\Delta)$ is for them to be joined by a $W$--edge. Since $t_1,t_2$ both belong to $\lk(\Delta)$,  Lemma~\ref{lem:C=C0_for_W} provides maximal simplices 
		$\Sigma_{t_1}=\{s,\nu(s),v,t_1\}$ and $\Sigma_{t_2}=\{s,\nu(s),v,t_2\}$ that are joined by an edge in $W$. By the definition of the edges of $W$, we have $d_{X}(\sigma_{r}(t_1),\sigma_{r}(t_2))\leq R+2$.  Using Lemma \ref{lem:uniform_distortion}, there is then a constant $\kappa \geq 1$ (determined by $r$ and $\mc{G}$) so that $d_{X_v}(\sigma_{r}(t_1),\sigma_{r}(t_2))\leq \kappa$. As the map $p_v\colon X_v \to L_v$ is distance non-increasing, $I(t_i) \in p_v(\sigma_{r}(t_i))$, and $\diam(p_v(\sigma_{r}(t_i)))\leq 2r$,  we have $$ d_{X_v}(\sigma_{r}(t_1),\sigma_{r}(t_2))\leq \kappa \implies d_{L_v}(I(t_1),I(t_2)) \leq \kappa +4r. \qedhere$$
	\end{proof}

	\textbf{Proof of~\eqref{item:3}.}  Let $\Delta=\{s,\nu(s)\}$  be a simplex of Type \ref{simplex:level_surface_type}, and 	let $v=\nu(s)$. By multiplying by an element of $G$,  we can assume that $\stab_G(v) = G_\mu$ for some vertex $\mu \in \mc{G}$.

	Let $Y$ be the graph obtained from $\link_T(v)$ by joining distinct $x,y\in\link_T(v)^{(0)}$ by an edge if and only 
	if there exist $x',y'\in\link(\Delta)^{+W}$ with $\nu(x')=x,\nu(y')=y$, and $x',y'$ adjacent in 
	$\mc{C}(\Delta)$.  Note that $G_\mu$ acts on $Y$, and $\nu \colon \ST^{(0)} \to T^{(0)}$ induces a $G_\mu$--equivariant simplicial map 
	$\eta\colon \mc{C}(\Delta) \to Y$.
	
	We first show $Y$ is connected and quasi-isometric to $\mc{C}(\Delta)$. 
	
	\begin{claim}\label{claim:connected_link}
		The graphs $\mc{C}(\Delta)$ and $Y$ are connected.
	\end{claim}
	
	\begin{proof}
		Because there is a simplicial surjection $\eta \colon \mc{C}(\Delta) \to Y$, connectedness of $\mc{C}(\Delta)$ will imply connectedness of $Y$.  
		
		Let $x,y\in\mc{C}(\Delta)^{(0)}$ and let $\Sigma_x,\Sigma_y$ be maximal simplices of $\ST$ containing $x$ and $y$ 
		respectively. Since $x,y\in\link(\Delta)$, we can use Lemma~\ref{lem:C=C0_for_W} to assume that $\Sigma_x,\Sigma_y$ are maximal simplices of the 
		star of $\Delta$. By Lemma~\ref{lem:W_is_a_G_space}, there is a path $\Sigma_x=\Sigma_0,\Sigma_1,\ldots,\Sigma_n=\Sigma_y$ in 
		$W$, where each $\Sigma_i$ is a maximal simplex of $\ST$ and  $\Sigma_i,\Sigma_{i+1}$ are joined be an edge of $W$ for $0\leq i\leq n-1$. We now argue by induction on $n$ that $x$ can be joined to $y$ by a path in $\mc{C}(\Delta)$.  
		
		If $n=0$, then both $x,y$ are contained in $\Sigma_x$ and hence are either equal or are adjacent in $\ST$ and 
		therefore in $\ST^{+W}$.  Since $x,y\in\link(\Delta)$, either $x=y$ or $x,y$ are adjacent in 
		$\mc{C}(\Delta)$ by Lemma \ref{lem:C=C0_for_W}.
		
		If $n=1$, then $\Sigma_0$ is joined by an edge of $W$ to $\Sigma_1$, hence $x,y\in\link(\Delta)$ are either equal or adjacent in $\ST^{+W}$. Thus, by Lemma \ref{lem:C=C0_for_W}, $x,y$ are either equal or adjacent in $\mc{C}(\Delta)$.
		
		Suppose $n>1$.	  Since $\Sigma_0$ and 
		$\Sigma_n$ are simplices of the star of $\Delta$, the edges $\nu(\Sigma_0)$ and $\nu(\Sigma_n)$ contain $v$.  Let $v_{-1}$ be the vertex of 
		$\nu(\Sigma_0)$ different from $v$, and let $v_{n+1}$ be the vertex of $\nu(\Sigma_n)$ different from $v$. Since $x,y \in \lk(\Delta)$, we must have $\nu(x)=v_{-1}$ and $\nu(y) = v_{n+1}$.
		If $\nu(x) = v_{-1} = v_{n+1} = \nu(y)$, then $x,y$ are in the link of the Type \ref{simplex:quasi-line_type} simplex $\Delta' =\{v_{-1}, s, \nu(s)\}$. Claim \ref{claim:L_v_edges_are_C(Delta)_edge} therefore implies $x$ is connected to $y$ as $\mc{C}(\Delta')$ has an injective simplicial inclusion into $\mc{C}(\Delta)$.

		Now suppose $\nu(x) \neq \nu(y)$. This implies  
		$v_{-1}$ and 
		$v_{n+1}$ must lie in different components of $T-\{v\}$. The definition of edges in $W$ (Definition~\ref{defn:W}) ensures that the edges $\nu(\Sigma_i)$ and $\nu(\Sigma_{i+1})$ 
		share a vertex $v_i$ for each all $i \in \{0,\dots, n-1\}$. The sequence		
		$v_{-1},v_0,\ldots,v_{n},v_{n+1}$ is then a sequence of vertices of $T$ where consecutive vertices are either equal or adjacent in $T$. 	Because $v_{-1}$ and 
		$v_{n+1}$ are in different components of $T-\{v\}$,  there exists $i\in\{1,\ldots,n-1\}$ such that 
		$\nu(\Sigma_{i})$ contains $v$.  Choose $z\in \Sigma_i^{(0)}$ such that $\nu(z)\in\nu(\Sigma_i)-\{v\}$.  Then 
		$z\in\link(\Delta)$, and $z$ is contained in the maximal simplex $\Sigma_i$.  The sequence 
		$\Sigma_0,\ldots,\Sigma_i$  is a path in $W$ with $i<n$, and has 
		$x\in\Sigma_0,z\in\Sigma_i$.  So, by induction, $x$ can be joined to $z$ by a path in $\mc{C}(\Delta)$.  
		Similarly, considering $\Sigma_i,\ldots,\Sigma_n$ shows that $z$ can be joined by a path in 
		$\mc{C}(\Delta)$ to $y$.  So $x,y$ are connected in $\mc{C}(\Delta)$, as required.
	\end{proof}
	
	We now prove that $Y$ is quasi-isometric to $\mc{C}(\Delta)$.
	
	\begin{claim}\label{claim:Z_qi}
		The map $\eta\colon \mc{C}(\Delta)\to Y$ induced by $\nu$ is a quasi-isometry with constants independent of 
		$\Delta$.
	\end{claim}
	
	\begin{proof}
		As mentioned above, $\eta$ is simplicial and hence 1-Lipschitz.  Consider the composition of inclusions $$Y^{(0)}\hookrightarrow 
		T^{(0)}\hookrightarrow \ST^{(0)}.$$  The image of this map is in $\link(\Delta)$, and the map is a quasi-inverse 
		for $\eta$.  Now, if $x,y\in Y^{(0)}$ are $Y$--adjacent, then let $ \Sigma(s,x')$ and $\Sigma(t,y')$
		be $W$--adjacent simplices where $\nu(x') = x$, $\nu(y') = y$, and $\nu(s) = \nu(t) = v$.  Then $\nu(x')=x$ and $\nu(y')=y$ are adjacent in $\mc{C}(\Delta)$, so 
		the map $Y\to \mc{C}(\Delta) \subseteq \ST$ induced by the above inclusions is uniformly coarsely Lipschitz.  Thus 
		$\eta$ is a quasi-isometry.
	\end{proof}
	
	In view of Claim~\ref{claim:Z_qi}, it suffices to prove that $Y$ is $\delta$--hyperbolic (and that $Y$ is uniformly quasi-isometric to a tree when the vertex groups are 
	$\mathbb Z$--by-virtually free). For this we use the action of $G_\mu$ on $Y$.

	\begin{claim}\label{claim:Z_cocompact}
		The action of $G_\mu$ on $Y$ is cocompact. 
	\end{claim}
	
	\begin{proof}
		Because  $G_\mu$ acts on $Y$ with finitely many orbits of vertices, it suffices to prove that for each vertex $u\in Y$, there are finitely many $\stab_{G_\mu}(u)$--orbits of edges of $Y$ incident to $u$. Note, $\stab_{G_\mu}(u) = \stab_G(e_u)$ where $e_u$ is the (oriented) edge of $T$ from $u$ to $v$.
		
		Let $y$ be an element of $Y$ that is joined by an edge of $Y$ to $u$.  Let $e_u$ and $e_y$ be the (oriented) edges of $T$ from $u$ or $y$ to $v$ respectively. By construction of $Y$, each of  $u$ and $y$ are  contained in a maximal simplex of $\ST$ that contains $v$ and are adjacent in $W$. The definition of edges in $W$ then requires that the edge space $\BS_{e_y}$ must intersect the $(R+2)$--neighborhood of $\BS_{e_u}$ inside the Bass--Serre space $\BS$.  Hence, each vertex $y$ of $Y$ that is adjacent to $u$ in $Y$ has a corresponding edge spaces $X_{e_y}$ of $\BS$ that is within $R+2$ of $\BS_{e_u}$. We will argue that there are only a finite number of $\stab_{G_\mu}(u)$--orbits of such edge spaces, which implies there is a  finite number of $\stab_{G_\mu}(u)$--orbits of vertices of $Y$ adjacent to $u$.
		
		Because $\stab_{G_\mu}(u) = \stab_{G}(e_u)$ acts cocompactly on $\BS_{e_u}$, it also acts cocompactly on the $(R+2)$--neighborhood of $\BS_{e_u}$.  Since two edges spaces intersect if and only if they are equal, there can only be a finite number of $\stab_{G_\mu}(u)$--orbits of vertex spaces of $\BS$ that intersect the $(R+2)$--neighborhood of $\BS_{e_u}$ as desired. 
	\end{proof}
	
	Claims~\ref{claim:connected_link} and \ref{claim:Z_cocompact} show that $Y$ is a connected graph (hence a length space) with a cocompact action 
	by $G_\mu$.   Let $\{y_i\}$ be a finite set of 
	vertices of $Y$ containing exactly one element of each $G_\mu$--orbit, and let $\mathcal H$ be the collection of 
	stabilisers in $G_\mu$ of the vertices $y_i$. Given our fixed finite generating set $J_\mu$ of $G_\mu$, 
	Theorem~5.1 of \cite{charney_crisp_rel_hyp}
	implies that any orbit map $G_\mu\to Y$ induces a quasi-isometry $\Gamma\to Y$ (with constants just depending on 
	$J_\mu$), where $\Gamma$ is the Cayley graph of $G_\mu$ with respect to the infinite set $J_\mu \cup \{H\}_{H \in \mc{H}}$.
	If $y$ is a vertex of $Y$, then the stabiliser in 
	$G_\mu$ of $y$ is exactly the stabiliser of some edge $e$ of $T$ with $e^+ = v$. Hence, each $H \in \mc{H}$ is conjugate to the image of some edge group $\tau_\alpha(G_\alpha)$ where $\alpha^+ = \mu$. Thus $\Gamma$ is quasi-isometric to the Cayley graph of $G_\mu$ with respect to the generating set $J_\mu \cup \{\tau_{\alpha}(G_\alpha) : \alpha^+ = \mu\}$. By Lemma~\ref{lem:cone-off_edge_groups} the later is always hyperbolic and is a quasi-tree when $F_\mu$ is virtually free. As $\Gamma$ is quasi-isometric to $Y$, this completes the proof of  \eqref{item:3}.
\end{proof}

\subsection{Quasi-isometric embedding of augmented links}\label{subsec:QI_embedding}
The goal of this section is to check condition \eqref{CHHS:geom_link_condition} of Definition \ref{defn:CHHS}, that is, that each augmented link $\mc C(\Delta)$ is quasi-isometrically embedded in the corresponding space $Y_\Delta$  from Definition \ref{defn:Y_Delta}.  Because there are finitely many orbits of links of simplices by Lemma~\ref{lem:ST_cocompact}, we will be able to choose the quasi-isometry constants uniformly over all simplices $\Delta$.

Because the quasi-isometric embedding condition automatically holds when $\mc C (\Delta)$ is bounded, we only have to check simplices of Type \ref{simplex:level_surface_type} and \ref{simplex:quasi-line_type}.

\subsubsection{Type \ref{simplex:level_surface_type} links}

\begin{lem}
	\label{lem:tentancles_qi}
	There exists $\kappa \geq 1$ so that if $\Delta = \{s,\nu(s)\} \subset \ST$ is a simplex of Type \ref{simplex:level_surface_type}, then $\mc C(\Delta)$ is $(\kappa,\kappa)$--quasi-isometrically embedded in $Y_\Delta$.
\end{lem}

\begin{proof}
	By Lemma \ref{lem:coarse_retract}, it suffices to define a coarsely Lipschitz coarse retraction $\rho\colon Y_\Delta\to\mathcal C(\Delta)$, with constants independent of $\Delta$. We define $\rho$ on the vertex set as follows: for $y\in Y_\Delta^{(0)}$, we let $\rho(y)$ be the unique vertex of $T$ (regarded as a vertex of $\mathcal S(T)$) at distance $1$ from $\nu(s)$ and on the geodesic in $T$ from $\nu(y)$ to $\nu(s)$. This is well-defined because $T$ is a tree and $y\neq\nu(s)$ if $y\in Y_\Delta^{(0)} = \ST^{(0)} - \Sat(\Delta)$. Moreover, because $\rho$ coincides with $\nu$ on the vertices of $\mc{C}(\Delta)$, the distance between $\rho(y)$ and $y$ is at most 1 for vertices $y \in \mc{C}(\Delta)$. Hence, $\mathcal C(\Delta)$ will be a coarse retract if $\rho$ is coarsely Lipschitz.
	
	If we can uniformly bound  $d_{\mathcal C(\Delta)}(\rho(y_1),\rho(y_2))$ whenever $y_1, y_2$ are joined by a edge of $Y_\Delta$, then $\rho$  can be extended to a coarsely Lipschitz map $Y_\Delta\to\mathcal C(\Delta)$. We will obtain $d_{\mathcal C(\Delta)}(\rho(y_1),\rho(y_2)) \leq 3$ for such $y_1,y_2$.
	
	If $y_1,y_2$ are joined by an  $\ST$--edge, then  $\nu(y_1)$ and $\nu (y_2)$ are  either equal or  joined by an edge of $T$.  Since $\nu(y_i) \neq \nu(s)$ for each $i =1,2$,  this implies $\rho(y_1) = \rho(y_2)$ because $T$ is a tree.  Hence we have $d_{\mathcal C(\Delta)}(\rho(y_1),\rho(y_2)) =0$. If instead $y_1,y_2$ are joined by a $W$--edge, then  $\nu(y_1)$ and $\nu(y_2)$ lie at distance at most $2$ in $T$ by Definition \ref{defn:W}. If $\nu(y_1)$ and $\nu(y_2)$ are  at most 1 apart in $T$, then $\rho(y_1) = \rho(y_2)$ as in the previous case. If the $\nu(y_1)$ and $\nu(y_1)$ are exact $2$ apart in $T$, then there exists a unique vertex $z \in T$ at distance 1 from both $\nu(y_1)$ and $\nu(y_2)$. If $z \neq \nu(s)$, then $\nu(y_1)$ and $\nu(y_1)$ are in the same component of $T - \nu(s)$, which implies  $\rho(y_1) = \rho(y_2)$. If $z = \nu(s)$, then $y_1$ and $y_2$ are in $\lk(\Delta)$. By Lemma \ref{lem:C=C0_for_W}, this implies $y_1$ and $y_2$ are joined by an edge in $Y_\Delta$. Because $\rho(y_i) = \nu(y_i)$, we have $d_{\mathcal C(\Delta)}(\rho(y_1),\rho(y_2)) \leq  d_{\mathcal C(\Delta)}(\rho(y_1),y_1)  + d_{\mathcal C(\Delta)}(y_1,y_2)  +d_{\mathcal C(\Delta)}(y_2,\rho(y_2)) \leq  3$.
\end{proof}

\subsubsection{Type \ref{simplex:quasi-line_type} links}
We now consider simplices of the form  $\Delta = \{s,\nu(s), v\}$, where $s\in \ST^{(0)} - T^{(0)}$ and $v \in T^{(0)}-\{\nu(s)\}$. 
\begin{lem}\label{lem:quasilines_qi}
	There exists $\kappa \geq 1$ with the following property.  Let $\Delta = \{s,\nu(s), v\}$ be a Type \ref{simplex:quasi-line_type}  simplex of $\ST$. The inclusion of $\mathcal C (\Delta)$ into  $Y_\Delta$ is a $(\kappa,\kappa)$--quasi-isometric embedding.
\end{lem}

By Proposition \ref{prop:Hyperbolicity_of_links}.\eqref{item:1}, the Type \ref{simplex:quasi-line_type} simplices are the simplices whose augmented links, $\mc{C}(\Delta)$, are quasi-isometric to the quasi-lines $L_v$.  As in the previous case, we will show quasi-isometric embedding by providing a coarse retraction. However, since the identity map on vertices gives a quasi-isometry $L_v \to \mc{C}(\Delta)$, it suffices to build a coarsely Lipschitz coarse map $\eta \colon Y_\Delta \to L_v$, that is the the identity on the vertices $L_v^{(0)} = X_v^{(0)} \subseteq Y_\Delta^{(0)}$.

To define this map, we need to assign to each vertex space $X_u$ of $X$ a projection onto a hyperbolic space. Given $u \in T^{(0)}$, let $\vartheta = \check{u}$ and choose a coset representative $g$ for $gG_{\vartheta}$; recall the vertices of $X_u$ are the elements of $gG_{\vartheta}$.  We now define a graph $H_u$ as follows: the vertices of $H_u$ are the elements of $gG_{\vartheta}$ and there is an edge between two elements $x,y$ if $x^{-1}y \in J_\vartheta \cup Z_\vartheta$, where $J_\vartheta$ is our fixed finite generating set for $G_\vartheta$ and $Z_\vartheta$ is the center of $G_\vartheta$. Since $H_u$ is a copy of $X_u$ with extra edges attached, there is a simplicial inclusion $\iota_u \colon X_u \to H_u$.   By construction,  multiplying every vertex of $H_u$ by $g^{-1}$ produces an  isometry to the Cayley graph of $G_\vartheta$ with respect to the generating set $J_\vartheta \cup Z_\vartheta$. Thus, Lemma~\ref{lem:cone-off_edge_groups} implies that $H_u$ is a hyperbolic graph.

Lemma~\ref{lem:cone-off_edge_groups} also shows that $H_u$ is hyperbolic relative to the collection $$\{ \iota_u(\tau_e(X_e)): e \text{ an edge of } T \text{ with } e^+ = u\}.$$
For an edge $e$ with  $e^+ = u$, define $\ell_e := \iota_u \circ \tau_e(X_e)$. As a peripheral subset in a relatively hyperbolic space, each $\ell_e$ has a coarse closest point projection  $ \mf{p}_{e} \colon H_u \to \ell_{e}$; see, e.g., \cite{S:proj}. This map is coarsely Lipschitz with constants independent of $e$ or $u$.

The key property about the $\ell_e$ that we shall need is that they have a coarsely Lipschitz map onto $L_{e^-}$. One can show that this is in fact a quasi-isometry, but it will not be needed in the proof.

\begin{lem}\label{lem:psi_coarsely_lipschitz}
	Let $v,u \in T^{(0)}$ and $e$ be an edge of $T$ with $e^+ = v$ and $e^- =u$. Let $\psi_e \colon \ell_{\bar e} \to L_v$ be the map given by restricting $p_v \circ \tau_e \circ \tau_{\bar e}^{-1} \circ \iota_u^{-1}$ to $\ell_{\bar e}$. Equipping $\ell_{\bar e}$ with the induced metric from $H_u$, the map $\psi_e\colon \ell_{\bar e} \to L_v$ is coarsely Lipschitz with constants determined by $\mc{G}$.
\end{lem}

\begin{proof}
	Let $\vartheta = \check{u}$ and $\alpha = \check{e}$. Recall $J_\vartheta$ and $J_\alpha$ are our fixed generating sets for the vertex groups $G_\vartheta$ and the edge group $G_\alpha$.
	
	Let $g \in G$ so that the vertices of $X_u$ (and $H_u$) are the elements of $gG_\vartheta$. Since $\iota_u \circ \tau_{\bar e}\colon  X_e \to H_u$ is a simplicial map, $\ell_{\bar e} = \iota_u \circ \tau_{\bar e}(X_e)$ is a connected subgraph of $H_u$. Hence it suffices to verify that whenever $x,y \in \ell_{\bar{e}}$ differ by an edge of $H_u$, that $d_{L_v}(\psi_e(x),\psi_e(y))$ is uniformly bounded. Let $x,y$ be vertices of $\ell_{\bar e}$ that differ by an edge of $H_u$. Hence $x^{-1}y$ is either an element of $J_\vartheta$ or of $Z_\vartheta$.
	
	If $x^{-1}y \in J_\vartheta$, then $x,y$ are elements of $\tau_{\bar e}(X_e)$ that are joined by an edge of $X_u$. Hence $x^{-1}y \in J_\vartheta \cap \tau_{\bar \alpha}(G_\alpha)$. Since there is a uniform bound on the number of elements of $\tau_{\bar \alpha}(J_\alpha)$ that are needed to write any element of $J_\vartheta \cap \tau_{\bar \alpha}(G_\alpha)$, there is  a uniform bound on the distance between $\tau_{\bar e}^{-1} \circ \iota_u^{-1}(x)$ and  $\tau_{\bar e}^{-1} \circ \iota_u^{-1}(y)$ in $X_e$ (independence of $\alpha$ and $\vartheta$ comes from considering the finitely many vertices and edges of $\mc{G}$). Since  $p_v$ and $\tau_e$ are distance non-increasing from $X_v$ and $X_e$ respectively, this shows $d_{L_v}(\psi_e(x),\psi_e(y))$ is uniformly bounded.
	
	If instead $x^{-1}y \in Z_\vartheta$, then $x,y$ are elements of the same coset $gZ_\vartheta$ and  $gZ_\vartheta \subseteq \tau_{\bar e}(X_e)$. Proposition \ref{prop:geometry_of_Zs}.\eqref{item:bounded_Zs} provides a uniform bound on the diameter of $p_v \circ \tau_e \circ \tau_{\bar e}^{-1}(gZ_\vartheta)$ in $L_v$.  Hence $d_{L_v}(\psi_e(x),\psi_e(y)) $ as uniformly bounded.
\end{proof}

We can now use the map $\psi_e$ from Lemma \ref{lem:psi_coarsely_lipschitz} to define a map $\rho_v$ for vertices of $T$ that are at least distance 2 from $v$.  We start with the case where $w \in T^{(0)}$ is exactly distance 2 from $v$. In this case, there is a unique vertex $u$ at distance 1 from both $v$ and $w$. If $f$ is the oriented edge of $T$ from $w$ to $u$ and $e$ is the oriented edge from $u$ to $v$, define $\beta_e^f$ to be $\mf{p}_{\bar{e}}(\ell_f)$.  We then define $$\rho_v(w) := \psi_e(\beta_e^f) = p_v \circ \tau_e \circ \tau_{\bar e}^{-1} \circ \iota_u^{-1}(\beta_e^f).$$
To define $\rho_v(w)$ when $w$ is more than 2 away from $v$ in $T$, let $\bar w$ be the unique vertex of $T$ that is distance exactly 2 from $v$ and on the geodesic in $T$ from $w$ to $v$. Define $\rho_v(w) := \rho_v(\bar w)$.

The first thing to verify is that $\rho_v(w)$ is uniformly bounded.

\begin{lem}
	There exist $\kappa_0 \geq 0$ so that for any $v\in T^{(0)}$, if $w \in T^{(0)}$ with $d_T(v,w) \geq 2$, then $\diam(\rho_v(w)) \leq \kappa_0$.
\end{lem}

\begin{proof}
	By the definition of $\rho_v$, it suffices to verify the lemma when $d_T(v,w) =2$. Let $u$ be the unique vertex of $T$ that is distance 1 from both $v$ and $w$. Let $e$ be the edge of $T$ from $u$ to $v$ and $f$ be the edge from $w$ to $u$. Since $\ell_{\bar e}$ and $\ell_f$ are distinct peripheral  subsets in the relatively hyperbolic space $H_u$, there is a uniform bound on the diameter of $\mf{p}_{\bar e}(\ell_f) = \beta_e^f$; see, e.g., \cite{S:proj}. Because the map $\psi_e$ is coarsely Lipschitz (Lemma \ref{lem:psi_coarsely_lipschitz}), this implies $\rho_v(w) = \psi_e(\beta_e^f)$ will be uniformly bounded in $L_v$.
\end{proof}

Next we verify that when two vertex spaces $X_w$ and $X_{w'}$ are close in $X$, we have that $\rho_v(w)$ and $\rho_v(w')$ are close in $L_v$. This will be a key step to showing that pairs of vertices of $Y_\Delta$ that are joined by a $W$--edge are sent to  uniformly bounded diameter set in $L_v$. 

\begin{lem}\label{lem:2-away}
	For every $q \geq 0$ there exists $\kappa_1\geq 0$ such that the following holds for each $v \in T^{(0)}$. Let $w,w'$ be vertices of $T$ with $d_T(w,w') \leq 2$ and $d_T(w,v), d_T(w',v) \geq 2$. If $d_{\BS}(X_w,X_{w'})\leq q$,  then $d_{L_v}(\rho_v(w),\rho_v(w'))\leq \kappa_1.$
\end{lem}
\begin{proof}
	Let $\bar w$ be the vertex of $T$ at distance exactly $2$ from $v$ and along the geodesic from $w$ to $v$, let $u$ be the unique vertex of $T$ at distance $1$ from $v$ and $\bar w$. Let $f$ and $e$ be the oriented edges of $T$ from $u$ to $v$ and from $\bar w$ to $u$ respectively. Define $\bar w',u',f',e'$ analogously, using $w'$ rather than $w$.
	
	If $\bar w=\bar w'$, then $\rho_v(w)=\rho_v(w')$ by definition and we are done. Otherwise, because $d_T(w,w')\leq 2$,   we must have  $w= \bar{w}$ and $w'= \bar{w}'$ and $u= u'$.  This implies $e=e'$ as well.

	Because each  edge and vertex space of $\BS$ separates $\BS$,  $$d_{\BS}(X_w,X_{w'})\leq q \implies d_{X}(\tau_{f}(X_f),\tau_{f'}(X_{f'})) \leq q.$$  Applying Lemma \ref{lem:uniform_distortion}  produces  a $\kappa = \kappa(r,\mc{G})\geq 1$ so that $d_{X_u}(\tau_{f}(X_f),\tau_{f'}(X_{f'})) \leq \kappa$. As the map $\iota_u \colon X_u \to H_u$ is distance non-increasing, we have $d_{H_u}(\ell_f,\ell_{f'}) \leq \kappa$. Because  $\mf{p}_{\bar e}$ is coarsely Lipschitz,  there is now a  uniform bound on the distance between $\beta_e^f$ and $\beta_e^{f'}$. Since $\psi_e \colon \ell_{\bar e} \to L_v$ is a coarsely Lipschitz (Lemma \ref{lem:psi_coarsely_lipschitz}),  this implies $d_{L_v}(\rho_v(w),\rho_v(w'))$ is uniformly bounded as well.
\end{proof}

We now present the proof of the quasi-isometric embedding of $\mc{C}(\Delta)$ into $Y_\Delta$.

\begin{proof}[Proof of Lemma \ref{lem:quasilines_qi}]
	By Proposition \ref{prop:Hyperbolicity_of_links}, the identity map on vertices is a quasi-isometry $L_v \to \mc{C}(\Delta)$ with constants independent of $\Delta$. Hence, the composition of this quasi-isometry with a  coarsely Lipschitz coarse map $\eta \colon Y_\Delta \to L_v$ that is the the identity on the vertices $L_v^{(0)} = X_v^{(0)} \subseteq Y_\Delta^{(0)}$ will produce a coarse retraction $Y_\Delta \to \mc{C}(\Delta)$. By Lemma \ref{lem:coarse_retract}, this suffices to prove the inclusion is a quasi-isometric embedding.
	
	By Lemma~\ref{lem:descriptions_of_saturations},  $\Sat(\Delta) = \{v\} \cup \{ t\in \ST : \nu(t) \in \lk_T(v)\}$. Since $ Y^{(0)}_\Delta = \ST^{(0)} - \Sat(\Delta)$, we have  $$Y_\Delta^{(0)}  = \left\{t \in \ST^{(0)} - \{v\} \colon \nu(t) =v \text{ or } d_T(\nu(t),v) \geq 2\right\}.$$ We now use the $\rho_v(w)$ from above to define the desired map $\eta \colon Y_\Delta \to L_v$.  If $t \in Y_\Delta^{(0)}$ and  $d_T(\nu(t),v) \geq 2$, then  we can define $\eta(t) = \rho_v(\nu(t)) \subseteq L_v$. If instead $\nu(t) =v$, then $t$ is a vertex of both $X_v$ and $L_v$, and we define  $\eta_v(t) = p_v(t) = t \in L_v$. Lemma \ref{lem:psi_coarsely_lipschitz} ensures $\diam(\eta(t))\leq \kappa_0$ for all $t \in Y_\Delta^{(0)}$. We can extend this definition of $\eta$ to a coarsely Lipschitz map on all of $Y_\Delta$ if we can show that $d_{L_v}(\eta(t_1),\eta(t_2))$ is uniformly bounded whenever $t_1$ and $t_2$ are joined by an edge of $Y_\Delta$.
	
	Let $t_1,t_2 \in Y_\Delta^{(0)}$ be joined by an edge. By the definition of the $W$--edge (Definition \ref{defn:W}), this implies $d_T(\nu(t_1),\nu(t_2)) \leq 2$. First assume both $\nu(t_1)$ and $\nu(t_2)$ are $v$. Thus $t_1,t_2 \in \mc{C}(\Delta)$ and are joined be an edge. Since $\eta(t_1) =t_1$ and $\eta(t_2) =t_2$, the quasi-isometry between $\mc{C}(\Delta)$ and $L_v$ ensures $d_{L_V}(\eta(t_1),\eta(t_2))$ is uniformly bounded. 
	
	Next suppose neither $\nu(t_1)$ or $\nu(t_2)$ equals $v$. If $d_T(\nu(t_1),\nu(t_2)) = 0$, then $\eta(t_1) = \eta(t_2)$ by definition. If $d_T(\nu(t_1),\nu(t_2)) =1$, then, with out loss of generality, the geodesic in $T$ from $\nu(t_1)$ to $v$ must contain $\nu(t_2)$.  Since each $\nu(t_i)$ are at least distance 2 from $v$, the definition of $\rho_v(\cdot)$ then implies $\eta(t_1) = \rho_v(\nu(t_1)) = \rho_v(\nu(t_2)) = \eta(t_2)$. Finally, if $d_T(\nu(t_1),\nu(t_2)) = 2$, then the edge between $t_1$ and $t_2$ must be a $W$--edge. This implies $X_{\nu(t_1)}$ and $X_{\nu(t_2)}$ are uniformly close in $X$. Hence the desired bound on $d_{L_v}(\eta(t_1),\eta(t_2))$ is a consequence of Lemma \ref{lem:2-away}.
	
	Finally consider the case where $\nu(t_1) = v$, but $\nu(t_2) \neq v$. In this case, $d_T(\nu(t_1), \nu(t_2)) =2$, and so the edge between $t_1$ and $t_2$ must be a $W$--edge. Hence, $d_X(\sigma_{r}(t_1), \sigma_{r}(t_2)) \leq R +2$ in either case of  Definition \ref{defn:W}. Let $u$ be the vertex distance 1 from both $v =\nu(t_1)$ and $\nu(t_2)$, then let $f$ be the oriented edge of $T$ from $\nu(t_2)$ to $u$ and $e$ be the oriented edge from $u$ to $\nu(t_1)$. 
	
	Let $\sigma_{e} =\tau_e^{-1}(\sigma_{r}(t_1))$ and $\sigma_f = \tau_{\bar f}^{-1}(\sigma_{r}(t_2))$. Our choice of $r$ is large enough that Lemma~\ref{lem:R_for_sigmas_to_intersect} ensures $\sigma_{e}$ and $\sigma_f$ are both non-empty.  Recalling that $\psi_e$ is the restriction of $p_v \circ \tau_e \circ \tau_{\bar e}^{-1} \circ \iota_u^{-1}$, we have that 
	\begin{equation*}
		\psi_e( \iota_u\circ \tau_{\bar e} (\sigma_e)) = p_v(\sigma_{r}(t_1)). \tag{$\ast$} \label{eq:psi_into_sigma}
	\end{equation*}

	\begin{claim}\label{claim:beta_close_to_sigma}
		There exists  $\kappa' \geq 1$ depending only on $\mc{G}$ so that $d_{H_u}(\iota_u\circ \tau_{\bar e} (\sigma_e),\beta_e^f)\leq \kappa'$.
	\end{claim}
	\begin{proof}
		Because  $d_X(\sigma_{r}(t_1), \sigma_{r}(t_2)) \leq R +2$, we have $ d_X( \tau_{\bar{e}}(\sigma_e),\tau_f(\sigma_f)) \leq R + 6$.
		Applying Lemma \ref{lem:uniform_distortion} produces  $\kappa = \kappa(R,\mc{G}) \geq 1$ so that $d_{X_u}( \tau_{\bar{e}}(\sigma_e),\tau_f(\sigma_f)) \leq \kappa$.  As $\iota_u \colon X_u \to H_u$ is distance non-increasing, we have $$d_{H_u}( \iota_u\circ \tau_{\bar{e}}(\sigma_e)),\iota_u \circ \tau_f(\sigma_f)) \leq \kappa.$$
		Recall that $\beta_e^f = \mf{p}_{\bar e}(\ell_f)$, that $\mf{p}_{\bar e}$ is a coarse closest point projection to $\ell_{\bar e}$, which is a quasiconvex subset of a hyperbolic space. Hence, there is some $\kappa'$, determined by $\kappa$ and the hyperbolicity constant, so that $d_{H_u}(\iota_u\circ \tau_{\bar e} (\sigma_e),\beta_e^f)\leq \kappa'$.
	\end{proof}

	Since $\psi_e$ is a coarsely Lipschitz, Claim \ref{claim:beta_close_to_sigma} plus \eqref{eq:psi_into_sigma} implies that $\psi_e(\beta_e^f) = \eta(t_2)$ is uniformly close to $p_v(\sigma_{r}(t_1))$. Since $t_1 \in \sigma_{r}(t_1)$ and $\diam(p_v(\sigma_{r}(t_1))) \leq 2r$, this implies $\eta(t_2)$ is uniformly close to $\eta(t_1) = p_v(t_1)$ in $L_v$ as desired.
\end{proof}

\bibliography{squiding}{}
\bibliographystyle{alpha}
\end{document}